\documentclass[3p, preprint]{elsarticle}
\usepackage{subcaption}

\usepackage{amsmath,
  enumerate,amsfonts,subcaption,overpic,rotating,multirow,color,graphicx,epsfig,
  amsthm}
\usepackage[usenames,dvipsnames]{xcolor}

\usepackage{epstopdf}
\usepackage{tabularx}
\usepackage{soul}

\usepackage{mathtools,amssymb}
\usepackage{url}%\usepackage{hyperref}
\usepackage{enumerate}
\usepackage{natbib}

\newcommand{\Gee}{\mathcal{G}}
\newcommand{\Ell}{\mathcal{L}}

\renewcommand{\H}{\mathcal{H}}

\newcommand{\reals}{\mathbb{R}}

\newcommand{\IND}[1]{{\bf 1}_{#1}}
\newcommand{\absv}[1]{\left| #1 \right|}
\newcommand{\norm}[2]{\left\| #1 \right\|_{#2}}
\newcommand{\Dee}{\mathcal{D}}

\newcommand{\innerprod}[3]{ \left( #1 , #2 \right)_{#3} }
\newcommand{\tdelta}{\tilde{\delta}}

\newcommand{\bigO}{\mathcal{O}}
\newcommand{\Span}{\text{\bf span }}
\newcommand{\mb}[1]{\mathbf{#1}}

\newtheorem{definition}{Definition}

\newtheorem{proposition}{Proposition}
\newtheorem{theorem}{Theorem}

\begin{document}

\title{On regularizations of the Dirac delta distribution\tnoteref{t1}}

\tnotetext[t1]{This work was
    supported in part by the Natural Sciences and Engineering Research
    Council of Canada and the Canada Research Chairs program.}

 \author[sfu]{Bamdad Hosseini \corref{corauth}}
    \ead{bhossein@sfu.ca}
  \author[sfu]{Nilima Nigam}
    \ead{nigam@math.sfu.ca}
  \author[sfu]{John M.~Stockie}
    \ead{stockie@math.sfu.ca}

\address[sfu]{Department of Mathematics, Simon Fraser University, 
  8888 University Drive, Burnaby, BC, V5A 1S6, Canada} 
\cortext[corauth]{Corresponding author}

\begin{abstract} { In this article we consider regularizations of the
    Dirac delta distribution with applications to prototypical elliptic
    and hyperbolic partial differential equations (PDEs). We study the
    convergence of a sequence of distributions $\mathcal{S}_H$ to a
    singular term $\mathcal{S}$ as a parameter $H$ (associated with the
    {support size} of $\mathcal{S}_H$) shrinks to zero. We characterize
    this convergence in both the weak-$\ast$ topology of distributions,
    as well as in a weighted Sobolev norm.  These notions motivate a
    framework for constructing regularizations of the delta distribution
    that includes a large class of existing methods in the
    literature. This framework allows different regularizations to be compared. The convergence of solutions of PDEs with
   these regularized source terms is then studied in various topologies such
    as pointwise convergence on a deleted neighborhood and weighted
    Sobolev norms. We also examine the
    lack of symmetry in tensor product regularizations and effects of
    dissipative error in hyperbolic problems.}
\end{abstract}
%% Keywords


\maketitle

\pagestyle{myheadings}
\thispagestyle{plain}
\markright{B.~Hosseini, N.~Nigam, J. M.~Stockie / On regularizations
  of the Dirac delta distribution}

%%%%%%%%%%%%%%%%%%%%%%%%%%%%%%%%%%%%%%%
%%%%%%%%%%%%%%%%%%%%%%%%%%%%%%%%%%%%%%%
%%%%%%%%%%%%%                                %%%%%%%%%%%%%%
%%%%%%%%%%%%%      Introduction              %%%%%%%%%%%%%%
%%%%%%%%%%%%%                                %%%%%%%%%%%%%%
%%%%%%%%%%%%%%%%%%%%%%%%%%%%%%%%%%%%%%%
%%%%%%%%%%%%%%%%%%%%%%%%%%%%%%%%%%%%%%%

\section{Introduction}
\label{sec:intro}

Many phenomena in the physical sciences are modelled by
partial differential equations (PDE) with singular source terms.  The
solutions of such PDE models are often studied using numerical
approximations. In some computational approaches, the singular source
terms are represented {\it exactly}, such as in~\cite{Morin,
  boffi-gastaldi-2003, boffi-gastaldi-2014, shu}.  A more common
approach is to approximate the source term using some regularized
function, and then obtain the numerical solution using a discretization
of the PDE with the approximate source.  One prominent example of the
utility of singular sources in applications is the immersed boundary
method~\cite{peskin}, wherein a Dirac delta distribution supported on an
immersed fiber or surface is used to capture the two-way interaction
between a dynamically evolving elastic membrane and the incompressible
fluid in which it is immersed.  In immersed boundary simulations, the
Dirac delta is replaced by a continuous approximation that is 
designed to satisfy a number of constraints that guarantee certain
desirable properties of the analytical and numerical solution.  Related
approximations are also employed in connection with the level set
method~\cite{osher} and vortex methods~\cite{majda,CortezMinion}.

Suppose we represent the original problem of interest in an abstract
form as follows:
\begin{subequations} \label{Problems}
  \begin{enumerate}[{\it {Problem}~1:}]
  \item {\it Find $u$ such that
      \begin{equation} 
        \mathcal{L}(u) = \mathcal{S}, 
      \end{equation} 
      where $\mathcal{L}$ is a PDE operator and $\mathcal{S}$ is a distribution which is used to model a singular source.}
  \end{enumerate}
  Then let $\mathcal{S}_H$ denote some approximation of
  $\mathcal{S}$ and consider the associated problem
  \begin{enumerate}[{\it {Problem}~2:}]
  \item {\it Find $u_H$ such that}
    \begin{equation} 
      \mathcal{L}(u_H) = \mathcal{S}_H.
    \end{equation} 
  \end{enumerate}
\end{subequations}    
Here, $H>0$ is some small parameter for which $\mathcal{S}_H \rightarrow \mathcal{S}$ in
some sense as $H\rightarrow 0$ (a sense that will be made concrete later
on).  One may then apply an appropriate numerical scheme (e.g., finite
difference, finite volume, finite element, spectral, etc.) with a
discretization parameter $h>0$, thereby obtaining two discrete solution
approximations: $u_h$ to $u$ in {\em Problem~1}; and $u_{H,h}$ to $u_H$
in {\em Problem~2}. We are free, of course, to pick one numerical scheme
for {\em Problem~1} and a different scheme for {\em Problem~2}.  If the
numerical schemes are suitably well-chosen, then both $u_h \rightarrow
u$ and $u_{H,h}\rightarrow u_H$ as $h\rightarrow 0$.
  
In practical computations, it may not be possible to construct  $u_h$. Indeed, it is typically only $u_{H,h}$ that is computed, by
first prescribing some approximation  to the source term and then
discretizing the PDE with the approximate source.  Ideally, what we hope
to obtain is that as both $h,H\rightarrow 0$, the discrete approximant
$u_{H,h} \approx u$.  In the immersed boundary method, for example, the
source term $\mathcal{S}$ is a line source and both the approximation
of the source term and the discretization of the PDE are performed with
reference to the same underlying spatial grid, so that parameters $H$
and $h$ are identical. Convergence of $u_{H,h} \to u$ in the context of
the immersed boundary method has been the subject of detailed analysis
in the works of Liu and Mori \cite{liumori, liumori2, mori-CPAM}.  However, these
authors only focus on convergence of the discrete regularizations and do
not consider $u_H \to u$. 

In this article we are concerned primarily with two questions:
\begin{enumerate}[{\it {Question}~1.}]
\item How do we construct `good' approximations $\mathcal{S}_H$ to
  $\mathcal{S}$?
\item How does the choice of approximation $\mathcal{S}_H$ affect the
  the convergence of $u_{H,h} \rightarrow u$?
\end{enumerate}
Before we can formulate answers to the above, we have to first answer
the two related questions: 
\begin{enumerate}[{\it {Question}~1.}]
  \setcounter{enumi}{2}
\item What form of convergence should be used to examine $\mathcal{S}_H
  \rightarrow \mathcal{S}$? 
\item What form of convergence should be used to examine $u_{H,h}
  \rightarrow u$?
\end{enumerate}

In this paper, we restrict our attention to the particular case of
$\mathcal{S}=\delta$ which denotes the well-known point source
distribution (or Dirac delta distribution) having support at the origin.

{\it Questions~1} and {\it 2} are fairly well-studied in some contexts
\cite{engquist, liumori, liumori2, beyer-leveque, tornberg,
  tornberg-engq, tornberg-quad, tornberg-zahedi} but {\it Questions~3} and {\it 4} have
not been the subject of much scrutiny in the literature.
%\todo{give references}. 
A common approach for approximating $\mathcal{S}_H$ (via regularization)
 is to construct a {\it discrete regularization} that is tailored to
 specific quadrature methods.  Wald\'en~\cite{Walden} presents an
 analysis of discrete approximations of the delta distribution,
 restricting his attention to applications to PDEs in one dimension.
 Tornberg and Engquist~\cite{tornberg} analyze discrete approximations
 to the delta distribution in multiple dimensions { with
   compact support} and draw a connection between the discrete moment
 conditions and the order of convergence of the solution of a PDE with
 the discrete $\mathcal{S}_H$ as source term. They also consider
 approximations of line sources using a singular source term or a
 collection of delta distributions in a chain. The analyses of
 Tornberg~\cite{tornberg-quad} and Tornberg and
 Engquist~\cite{tornberg-engq} for the discrete approximations
 $\mathcal{S}_H$ rely on the choice of mesh and quadrature rules.  They
 also restrict $H=\mathcal{O}(h)$ and compare $u_{H,h}$ directly to $u$
 so that, $\mathcal{S}_H$ is de based on the numerical method used to
 compute $u_{H,h}$.  More recently, Suarez et al.~\cite{suarez}
 considered regularizations of the delta distribution that are tailored
 to spectral collocation methods for the solution of hyperbolic
 conservation laws. Their approach to constructing polynomial
 regularizations using the Chebyshev basis has a similar flavor to our
 approach, as will be seen in
 Section~\ref{section:regularizations}. { In a different
   approach, Benvenuti et al.~\cite{Benvenuti-reg-sing} study the case
   of regularizations that are not compactly supported but have rapidly
   decaying Fourier transforms in the context of extended finite element
   methods (XFEM) \cite{Benvenuti-XFEM}.  The authors demonstrate that
   such regularizations lead to lower numerical errors since they can be
   integrated using common quadrature methods such as Gauss quadrature.}

{ In this article} we demonstrate firstly how to develop regularizations $\mathcal{S}_H$
independent of the choice of numerical discretization. For example,
in answering {\it Question~1} we derive {\it piecewise smooth}
approximations $\mathcal{S}_H$.  We can then examine the intermediate
errors $\|u-u_{H}\|_X$ and $\|u_{H}-u_{H,h}\|_X$ and use the triangle
inequality to give a bound on 
\begin{equation}\label{reg_error_and_dis_error}
  \|u-u_{H,h}\|_X \leq 
  \underbrace{\|u-u_H\|_X}_{\mbox{\footnotesize regularization error}} 
  + \underbrace{\|u_H-u_{H,h}\|_X}_{\mbox{\footnotesize discretization error}}, 
\end{equation}
where $\|\cdot\|_X$ refers to a suitably chosen norm; the choice of norms is discussed below.  For fixed $H>0$ the discretization errors
$\|u_{H,h}-u_H\|_X$ are analyzed using properties of the numerical
scheme and regularity of solutions of {\it Problem~2}.  The resulting
discretization errors are well-understood for specific problems
and specific schemes, and so we focus our attention here on the
regularization errors. 

{ We propose a unified
  approach for construction
and analysis of
  regularizations.} This
 has three advantages: First, we are able to provide a simple strategy
 for constructing { new} regularizations suitable for a given
 application (and not constrained to a specific numerical method for
 that application). { Second, our
 framework is flexible and allows us to study the effect of
 additional constraints on the regularizations. For example, it is 
common in the immersed boundary method to impose the even-odd
conditions on the discrete regularizations, but the
effect of these conditions on the regularization error is subtle \cite{mori-CPAM}.}
{ Third, our framework provides a unified platform from which
  to compare different regularizations in the literature. For example, in
 Section 3 we show that both radially symmetric and tensor product
 regularizations lead to the same regularization errors, but their
 discretization errors may differ. This provides more insight into the
 result of \cite{tornberg} indicating that
 regularizations based on distance functions can lead to large
 errors. Our results indicate that this is due 
 to the discretization error and not the regularization error. An
 important point to note is that we do not account for errors due to the
 use of discrete quadrature rules; these are embedded in the analysis of
 the discretization errors, and our goal is instead to examine the
 errors purely due to regularization.}

The construction procedure requires us to first obtain answers to {\it
  Questions~3} and {\it 4}.  The mode of convergence and (where
appropriate) the norm $\|\cdot \|_X$ are chosen in a problem-dependent
manner.  More precisely, if we seek approximations of solutions to {\em
  Problem~1} in some Hilbert space $\H$, and if the PDE operator
$\mathcal{L}:\H \rightarrow \H^*$, then we need a strategy to construct
regularizations $\mathcal{S}_H$ from suitable elements in~$\mathcal{H}^\ast$
so that $u_H$ converges to $u$ in $\mathcal{H}$.

We begin by constructing linear functionals $\tilde{\delta}_H$ via smooth and
compactly supported elements
of $\mathcal{H}$.
We then discuss the convergence of $\tdelta_H \to \delta$ in the {\it
  weak-$\ast$ topology} and in a class of {\it weighted Sobolev
  norms}. Then we derive a set of continuous moment conditions that are
entirely analogous to the discrete moment conditions of~\cite{tornberg}.
A strategy for solving a finite number of moment equations is
presented using a basis of our choice, after which we consider
convergence of $u_H \to u$ in the {\it sup-norm in a deleted
  neighbourhood} and a different class of {weighted Sobolev norms}
for some operators $\mathcal{L}$.

{ We emphasize that in this article our focus is on regularizations that have
  compact support. However, some of our results such as the
  moment conditions
  of Definition 3 can be readily generalized to the case of
  regularizations that are not compactly supported.}
We also note that in general, the continuous regularizations we construct are
not unique. In the specific case where Legendre polynomials are chosen
as a basis while solving the moment problem, 
the continuous regularizations are polynomials on the support of
$\delta_H \in \H$. These $\delta_H$ can then be made to have arbitrary smoothness by adding extra continuity
conditions; this should be compared with the piecewise polynomial
approximants that are obtained using discrete regularizations.

In most applications of interest, the domains on which we numerically
compute our approximate PDE solutions are polygonal.  We also focus on
PDE operators $\mathcal{L}$ associated with free space or zero Dirichlet
boundary data, although our analysis and solution construction are
easily extended to more general cases. 

The framework we present in this paper is independent of the specific
choice of numerical method used to approximate solutions of the
regularized problem. To this end, we present computations using
spectral and finite element methods within readily-available implementations.

Throughout the remainder of this paper, $\Omega$
will represent an open Lipschitz domain in $\mathbb{R}^n$ with polygonal
boundary that contains the origin.

%%%%%%%%%%%%%%%%%%%%%%%%%%%%%%%%%%%%%%%
%%%%%%%%%%%%%%%%%%%%%%%%%%%%%%%%%%%%%%%
%%%%%%%%%%%%%                                  %%%%%%%%%%%%%%
%%%%%%%%%%%%%      Section 2                   %%%%%%%%%%%%%%
%%%%%%%%%%%%%                                  %%%%%%%%%%%%%%
%%%%%%%%%%%%%%%%%%%%%%%%%%%%%%%%%%%%%%%
%%%%%%%%%%%%%%%%%%%%%%%%%%%%%%%%%%%%%%%

%\input{Section2}
\section{Background}
\label{section:background}

We begin by fixing notation. We denote the set of {\it test functions}
by $\Dee(\Omega)$. This is the space of real-valued $C^\infty$ functions
that are compactly supported in $\Omega$ and equipped with the usual
topology of the test functions. We denote the {\it distributions} by
$\Dee^\ast(\Omega)$, the dual of $\Dee(\Omega)$.  The point source or
delta distribution $\delta\in\Dee^\ast(\Omega)$ (also called the Dirac
delta functional) is defined by
\begin{equation*}
  \label{deltadefinition}
  \delta(\phi) := \phi(0) \quad \text{ for } \phi \in \Dee(\Omega).  
\end{equation*}

A natural setting in which to study solutions of {\it Problem 1} is in a
Sobolev space $\H_0^s(\Omega)$ for $s\in \mathbb{R}$ and $s>0$.   The particular choice of $s$ is
dictated by the differential operator $\mathcal{L}$. Since
$\H_0^s(\Omega)$ is a Hilbert space, and since the continuous inclusion
map from $\Dee(\Omega)\rightarrow \H_0^s(\Omega)$ is dense, the spaces
$(\Dee(\Omega), \H_0^s(\Omega), \Dee^\ast(\Omega))$ form a Gel'fand
triple (see \cite{gelfand}). This means that we can identify the dual of
$\H^s_0(\Omega)$ with elements of $\Dee^\ast(\Omega)$,
i.e. $(\H^s_0)^\ast(\Omega)\subset \Dee^\ast(\Omega)$. Instead of
working with distributions directly, we compute with piecewise
polynomials or trigonometric functions. We connect the desired
distributions with these functions {using an $L^2-$inner product}.

For the concrete situation in this paper, the goal is to approximate the delta distribution $\delta\in\Dee^\ast(\Omega)$
using other distributions $\tdelta_H \in \Dee^\ast(\Omega)$ for $H>0$, with the property that
$\tdelta_H \rightarrow \delta$ in some suitable sense. {These approximations are typically {\it
  regularizations}.  We construct regularizations
of a specific kind: we start with a suitably regular function $\delta_H\in \H^s_0(\Omega)$, and define the linear functional
 $\tdelta_H \in
 \H^{-s}(\Omega)\equiv (\H^s_0)^\ast(\Omega)$  
as
\begin{equation}
  \label{rieszrep}
  \tdelta_H(\phi) := \innerprod{ \delta_H
  }{u}{\Omega} \quad \forall \phi \in \mathcal{H}^s_0(\Omega),  
\end{equation}
where $\innerprod{\cdot}{\cdot}{\Omega}$ is the $L^2(\Omega)$ inner
product. In other
words, we aim to identify $\delta_H \in \H^s_0(\Omega)$ that are
convenient to work with, which also satisfy rules that guarantee a high
rate of convergence of the associated distributions $\tdelta_H
\rightarrow \delta$ in specific norms.  The definition \eqref{rieszrep}
allows us to relate functions $\delta_H$ with their associated
distributions $\tdelta_H$; in what follows we will occasionally use
these interchangeably when the context is clear.}

{ A key feature of this framework is its flexibility, deriving from the
fact that we choose a value of $s>0$ to
ensure a desired regularity for $\delta_H$ and then define the regularization $\tdelta_H$ via \eqref{rieszrep}.} This allows us to design a
regularization that is best suited not only for a given problem but also
for a given choice of the numerical method. Another important feature is
that we can decouple the rate of convergence of $\tdelta_H\rightarrow
\delta$ from the regularity of $\tdelta_H$. Finally, we can {\it
  compare} the quality of different regularizations in the literature in
this framework.

Previous works of Tornberg~\cite{tornberg}, Tornberg and
Engquist~\cite{tornberg-engq} and Walden \cite{Walden} follow a similar
strategy of construction in the sense that a distribution is
approximated using inner products. However, an essential difference from
our approach is that they construct the discrete approximant by
replacing $(\cdot, \cdot)_{\Omega}$ with a suitable quadrature scheme, so that
$\tilde{\delta}_H$ is specified in terms of point values at appropriate points
in $\Omega$.

In the following two subsections we briefly describe two modes of
convergence of a sequence of distributions $\{\mathcal S_H\}_H $ to
$\mathcal{S}$: {\it convergence in the weak-$\ast$ topology} and {\it
  convergence in a weighted Sobolev norm.}

\subsection{Weak-$\ast$ convergence and moment conditions}
\label{sec:conv-weakstar}

Recall first the  standard notion of convergence of distributions
$\tdelta_H\in\Dee^\ast(\Omega)$.
\begin{definition}
  The sequence $\tdelta_H$ of distributions converges in the weak-$\ast$
  topology (i.e., converges in distribution) to $\delta$ as
  $H\rightarrow 0$ if and only if
  \begin{equation}
  \label{distributionconv}
    \tdelta_H (\phi) \rightarrow \delta(\phi)=\phi(0) \qquad \forall
    \phi \in \Dee(\Omega)  \quad \text{as } H \to 0.
  \end{equation} 
\end{definition}
\noindent For specific sequences $\tdelta_H$, we are also interested in the {\em
  rate} at which $\tdelta_H\rightarrow \delta$. 

Let $\epsilon>0$, $s > 0$ and $m\in \mathbb{N}$ be given. If $\phi\in
\Dee(\Omega)$ then we can find $r=r(\phi)>0$ such that
$\|\partial^\alpha \phi(x)-\partial^\alpha\phi(0)\|<\epsilon$ for all
$|x|<r$ and $|\alpha|\leq m+1$. Now suppose that for each $H>0$ we have
$\tdelta_H \in (\H^s_0)^\ast$, and furthermore that the associated element
$\delta_H\in \H^s_0(\Omega)$ is integrable and has compact support within
a ball of radius $H$.  Using Taylor's theorem~\cite{evans} and \eqref{rieszrep}, we have for all $0<H<r$ that
\begin{equation*}
  \tdelta_H(\phi)= \innerprod{\delta_H}{\phi}{\Omega}
  = \innerprod{\delta_H}{\phi(0) +  
    \sum_{1 \le \absv{\alpha}\le m} \frac{\partial^\alpha \phi
      (0)}{\alpha !} \mb{x}^\alpha +\sum_{|\beta|=m+1}R_\beta({\mathbf
      y}) (\mb{x})^\beta} {\Omega},   
\end{equation*} 
from which it is clear that
\begin{equation} 
  \label{taylorexpand}
  \tdelta(\phi)= \phi(0) \innerprod{\delta_H}{\IND{\Omega}}{\Omega} +
  \sum_{1 \le \absv{ \alpha}\le m} 
  \frac{\partial^\alpha \phi (0)}{\alpha !}
  \innerprod{\delta_H}{\mb{x}^\alpha }
  {\Omega}+\innerprod{\delta_H}{\sum_{|\beta|=m+1}R_\beta(\mathbf{y})
    (\mb{x})^\alpha}{\Omega},  
\end{equation}
where $\IND{\Omega}$ denotes the characteristic function of $\Omega$ and
we have used standard multi-index notation for $\alpha$ and
$\beta$. Here, $R_{\beta}(\mathbf{y})$ is given by
\begin{equation*}
  R_\beta(\mathbf{y}) = \frac{|\beta|}{\beta!}\int_0^1
  (1-t)^{|\beta|-1}\partial^\beta \phi(t\mathbf{y})\, dt
  \quad \Longrightarrow \quad
  |R_\beta(\mathbf{y})| \leq \frac{1}{\beta!}
  \max_{|\alpha|=|\beta|} \max_{|y|<r} |\partial^\alpha \phi(y)|.
\end{equation*}
Then, recalling \eqref{distributionconv} and \eqref{taylorexpand} we immediately
obtain
\begin{multline}
  \label{deltaerror}
  \absv{\delta(\phi)-\tdelta_H(\phi)}\leq \absv{  
    \phi(0)\left( 1 - \innerprod{\delta_H}{ \IND{\Omega} }{\Omega}
    \right) +  \sum_{1 \le \absv{\alpha}\le m}
    \frac{\partial^\alpha \phi (0)}{\alpha !}
    \innerprod{\delta_H}{\mb{x}^\alpha } {\Omega}}
  \\
  + \absv{\sum_{|\beta|\ge m+1}\innerprod{\delta_H}{R_\beta(\mathbf{y})
      (\mb{x})^\beta}{\Omega} }.
\end{multline}
Clearly, $\tdelta_H \rightarrow \delta$ in the sense of distributions if
the right hand side of \eqref{deltaerror} can be made small as
$H\rightarrow 0$. In practice, we wish to design $\tdelta_H$ to achieve
a given rate of convergence; for example, for fixed $H>0$ we seek to
minimize the first term on the right of \eqref{deltaerror}. To this end,
given $m\in \mathbb{N}$ and $H>0$, we choose regularizations
$\delta_H\in \H_0^s(\Omega)$ to be compactly supported within a ball of
radius $H$ such that
\begin{equation}
  \label{momentsgeneral}
  \innerprod{\delta_H}{ \IND{\Omega} }{\Omega}  = 1 
  \quad \text{and} \quad
  \innerprod{\delta_H}{\mb{x}^\alpha } {\Omega} = 0 
  \quad \text{for } 1 \le \absv{\alpha} \le m. 
\end{equation}
These equations are called the {\it moment conditions}. Such
regularizations $\tdelta_H$ will, thanks to \eqref{deltaerror}, satisfy
\begin{equation}\label{weak-star-estimate}
  \absv{\delta(\phi)-\tdelta_H(\phi)}\leq \frac{1}{\beta!}
  \max_{|\alpha|=|\beta|} \max_{|y|<r} |\partial^\alpha
  \phi(\mathbf{y})| \cdot |\left(\delta_H,\xi^\alpha\right)_{\Omega}| =
  C(\phi, m) H^{m+1}.
\end{equation}
Note that we never need to differentiate $\delta_H$, and hence even if
the Sobolev index $s>0$ is small, we may still ask for a higher rate $m$
of convergence for $\tdelta_H \rightarrow \delta$.

For the construction of discrete regularizations (such as the
methodology in Tornberg et~al.~\cite{tornberg-engq}) the $L^2$-inner
products $(\cdot, \cdot)_{\Omega}$ appearing in our moment conditions
are replaced by quadratures, and the corresponding equations take the
form 
\begin{equation} 
  \label{momentsdiscrete}
  \sum_{k=1}^M \delta_H(\mb{x}_k)\, \omega_k = 1
  \qquad  \text{and} \qquad
  \sum_{k=1}^M \delta_H(\mb{x}_k) \mb{x}_k^\alpha \omega_k = 0 
  \quad \text{for } 1 \le \absv{\alpha} <m , 
\end{equation}
where $\mb{x}_k$ and $\omega_k$ are quadrature nodes and weights
respectively.  The discrete approximants $\delta_H$ are piecewise smooth
functions that are chosen to ensure that these {\it discrete moment
  conditions} are satisfied.

\subsection{Convergence in ${(W_{-\alpha})^\ast}$ and moment
  conditions}  
\label{sec:conv-sobolev}

Another mode of convergence for $\tdelta_H\rightarrow \delta$ derives
from the use of weighted Sobolev norms, for which we follow the
treatment in~\cite{Morin}. These weighted norms offer the advantage that
they behave like the usual energy norm on subdomains away from the
location of the point source.  The weight is chosen to suitably weaken
the norm, or more informally to \emph{cancel the singular behaviour} of
solutions to {\it Problem~1} as we approach the point source.
Consequently, using these weighted norms allows us to develop a
posteriori error estimates for solutions of {\it Problem~1} if
$\mathcal{L}$ is an elliptic operator in a certain class. Throughout the
rest of this article we make use of the notation $a \lesssim b$ whenever
there exists a positive constant $C$ such that $a \le C b$ and $C$ is
independent of $b$. We now recall a number of definitions
from~\cite{Morin}.
\begin{definition}
  For $\Omega \subset \reals^n, n \ge 2$ and constant $\beta \in (-\frac{n}{2},
  \frac{n}{2})$, the space $L_\beta^2(\Omega)$ is defined as the set of
  measurable functions $u$ such that
  \begin{equation}
    \label{weightednorm}
    \norm{u}{L_\beta^2(\Omega)} := \left( \int_\Omega \absv{u(x)}^2
      \absv{x}^{2 \beta} dx\right)^{\frac{1}{2}} < \infty. 
  \end{equation}
  The weighted Sobolev space $\H^1_\beta (\Omega)$ is the space of weakly 
  differentiable functions $u$ such that 
  \begin{equation}
    \label{weightedsobolevspace}
    \norm{u}{\H^1_\beta(\Omega)} := \norm{u}{L_\beta^2(\Omega)}
    + \norm{ \nabla u}{L_\beta^2(\Omega)} < \infty.
  \end{equation}
  Finally, we define the subspace with zero boundary values as
  \begin{equation}
    \label{W-space}
    W_\beta := \{ u \in\H^1_\beta(\Omega) : u|_{\partial \Omega} = 0\}.
  \end{equation}
\end{definition}%

Due to our choice of weight function $|x|^{2\beta}$, the Poincar\'{e}
inequality holds in $W_\beta$ (refer to Section~2 of~\cite{Morin} and
Theorem~1.3 of~\cite{fabes}). We may therefore define a norm
$\norm{u}{W_\beta} := \norm{\nabla u}{L^2_\beta}$ that is equivalent to
the full norm $\norm{u}{\H^1_\beta(\Omega)}$ for $u\in W_\beta$. It
follows that if $\frac{n}{2}-1<\alpha<\frac{n}{2}$, then $W_{-\alpha}\subset
\H^1_0(\Omega)\subset W_{\alpha}$ and hence test functions are dense in
$W_\alpha$ (see~\cite{Heinonen}).

We now apply a variant of the central result from Morin's Theorem
4.7~\cite{Morin}.  Let $\frac{n}{2}-1<\alpha<\frac{n}{2}$ and let $B_R$
be a ball of radius $R$ contained in $\Omega$. Then for all $v \in
W_{-\alpha}(\Omega)$,
{
\begin{equation} 
  |\delta(v)| \lesssim R^{\alpha-n/2}\|v\|_{L^2_{-\alpha}(B_R)} +
  C_\alpha R^{\alpha +(2-n)/2} \|\nabla
  v\|_{L^2_{-\alpha}(B_R)} ,
\end{equation}}
\noindent 
from which we conclude that $\delta\in (W_{-\alpha})^\ast$.  This
suggests that if we have a sequence of regularizations $\tdelta_H \in
\H^{-1}(\Omega) = (\H^1_0(\Omega))^\ast \subset (W_{-\alpha})^\ast$, then
for any $0<H<R$, $\tdelta_H$ converges to $\delta$ in
$(W_{-\alpha})^\ast$ if and only if $\|\delta
-\delta_H\|_{(W_{-\alpha})^{\ast}}\rightarrow 0$.  Specifically, using
the definition of the dual norm, we have
\begin{equation}
  \label{weakdistributionalconvergence} \| \tdelta_H -\delta\|_{
    (W_{-\alpha})^\ast} = \sup_{0\not=u\in W_{-\alpha}}\left\{
    \frac{\tdelta_H(u)-\delta(u)}{\|u\|_{W_{-\alpha}}} \right\}.
\end{equation}
{ 
  In order to use an argument similar to that of Section 2.1, we need to 
  regularize elements of $W_{-\alpha}$. To this end, let
  $\chi_\epsilon$ be the $C^\infty$ cutoff function that takes the value 1
  on $B(0,\epsilon)$, 0 outside $ B(0,2\epsilon)$, and varies smoothly
  between 1 and 0 on the annular region $B(0,2\epsilon)\setminus
  B(0,\epsilon)$ and let $\varphi_\epsilon(x) := \epsilon^{-n}
  \varphi(\mb{x}/\epsilon)$ where $\varphi$ is the usual mollifier
  supported on $B(0,1)$: 
  $$ \varphi(\mb{x}) = \left\{ 
    \begin{aligned}
      &  \exp \left(\frac{-1}{1 - |\mb{x}|^2}\right) \quad &&\text{if} \quad |\mb{x}| \le
      1, \\
      & 0 \quad &&\text{if} \quad |\mb{x}| >1.
    \end{aligned} \right.
  $$

  Fix $\alpha>0$ as above, and let $\mu>\alpha$ be given. For any $w\in W_{-\alpha}$, define
  \begin{equation}
    \label{cinfapprox}
    w_{\epsilon, \mu} := ( \varphi_\epsilon \ast w ) (1-\chi_{2\epsilon}) + (
    \varphi_\epsilon \ast \absv{x}^\mu ) \chi_\epsilon,  
  \end{equation}
  which is in $C^\infty(\Omega)$. We observe that $
  w_{\epsilon,\mu}\equiv \phi_\epsilon \ast w$ outside the ball of radius
  $4\epsilon$ centered at the origin. Within this ball,
  $w_{\epsilon,\mu}$ behaves as $|x|^\mu$ as $x\rightarrow 0$ (see
  Figure \ref{fig:proof-schematic}). We now show that
  $w_{\epsilon,\mu}\rightarrow w$ as $\epsilon\rightarrow 0$.
  
  \begin{figure}[htp]
    \centering
    \includegraphics[height=4cm,clip=true, trim = 2cm 7cm 4cm 7cm]{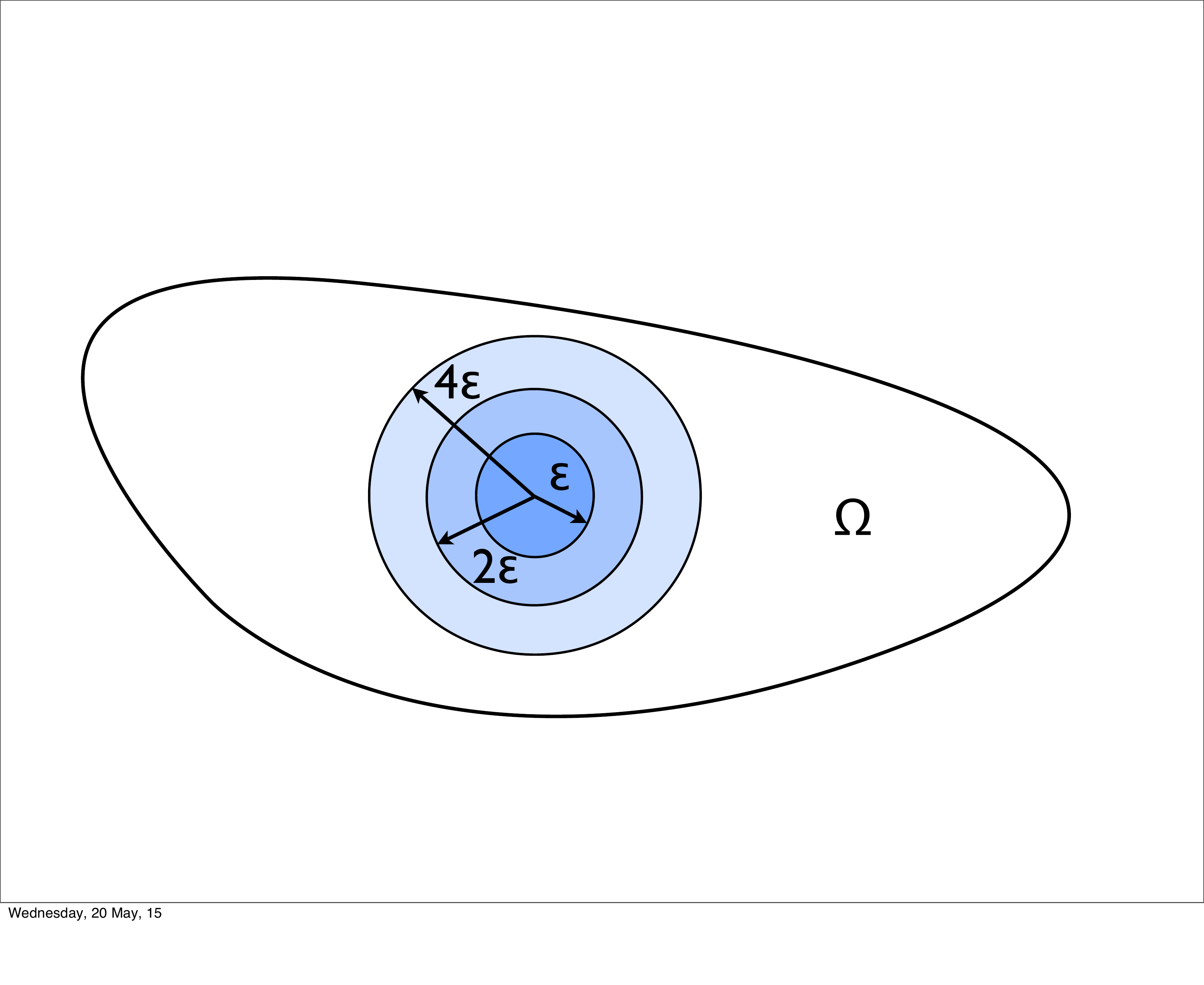}
    \caption{A Schematic showing the different sections of the approximation $w_{\epsilon, \mu}$ in equation \eqref{cinfapprox}.}
    \label{fig:proof-schematic}
  \end{figure}
 
  \begin{proposition}
    Let $n \ge 2$. For all $\mu > \alpha$ with $\frac{n}{2}-1<\alpha<\frac{n}{2}$, we
    have that $\norm{w_{\epsilon, \mu}-w}{W_{-\alpha}} \to 0$ as $\epsilon
    \to 0$.
  \end{proposition}
  \begin{proof}
    Using the definition of the norm and \eqref{cinfapprox}, we can expand
    $ \norm{w_{\epsilon, \mu} - w}{W_{-\alpha}(\Omega)}$ in terms of
    contributions in three subdomains:
    \begin{multline}
      \label{prop1eq1}
      \norm{w_{\epsilon, \mu} - w}{W_{-\alpha}(\Omega)}^2 =  
      \norm{ \nabla ( \varphi_\epsilon \ast w - w)
      }{L^2_{-\alpha}(\Omega \setminus B(0,4\epsilon))}^2
      \\ + \norm{ \nabla ( (\varphi_\epsilon \ast w)(1-\chi_{2\epsilon})
        - w)}{L^2_{-\alpha}(
        B(0,4\epsilon)\setminus B(0,2\epsilon))}^2 +
      \norm{ \nabla ( (\varphi_\epsilon \ast |x|^\mu)\chi_\epsilon - w)}{L^2_{-\alpha}(
        B(0,2\epsilon)\setminus B(0,\epsilon))}^2 + \\ \norm{ \nabla (
        \varphi_\epsilon \ast \absv{x}^\mu  - w )}{L^2_{-\alpha}
        B(0,\epsilon))}^2 .
    \end{multline}
    The first term vanishes since $\varphi \ast w \rightarrow w$ as
    $\epsilon \rightarrow 0$ outside $B(0,4\epsilon)$, while the last
    term also vanishes since both $\absv{x}^\mu$ and $w$ belong to
    $W_{-\alpha}(\Omega)$.  Before treating the remaining terms, we first
    recall the following bounds (see Theorem 3.6 in \cite{mclean}):
    \begin{equation}
      \label{cutoffbounds}
      \absv{\partial^k \chi_\epsilon} \lesssim \epsilon^{-\absv{k}} 
      \qquad \text{and} \qquad
      \absv{ \partial^k \varphi_{\epsilon}} \lesssim \epsilon^{-\absv{k}} 
    \end{equation}
    for all multi-indices $k$. Let us also define $S(4\epsilon):=
    B(0,4\epsilon) \setminus B(0,2\epsilon)$, Now using the triangle
    inequality for the second term in \eqref{prop1eq1} we get
    \begin{multline*}
      \norm{ \nabla ( (\varphi_\epsilon \ast w)(1-\chi_{2\epsilon}) - w)}{L^2_{-\alpha}(
        S(4\epsilon))}  \\ \le \norm{  (\varphi_\epsilon \ast \nabla
        w)(1-\chi_{2\epsilon}) - \nabla w}{L^2_{-\alpha}(
        S(4\epsilon))} + \norm{  (\varphi_\epsilon \ast
        w) \nabla (1-\chi_{2\epsilon}) }{L^2_{-\alpha}(
        S(4\epsilon))} 
    \end{multline*}
    We then use \eqref{cutoffbounds} and add and subtract the mean
    $\overline{\varphi_\epsilon \ast w} := (4\epsilon)^{-1} \int_{B(0,4\epsilon)}
    (\varphi_\epsilon \ast w ) |x|^{-\alpha} dx$ to get 
    \begin{multline*}    \label{prop1secondterm}
      \norm{ \nabla ( (\varphi_\epsilon \ast w)(1-\chi_{2\epsilon}) - w)}{L^2_{-\alpha}(
        S(4\epsilon))}\\
      \le \norm{  (\varphi_\epsilon \ast \nabla
        w) - \nabla w}{L^2_{-\alpha}(
        S(4\epsilon))} + \epsilon^{-1} \norm{\varphi_\epsilon \ast w - \overline{\varphi_\epsilon \ast w} }{L^2_{-\alpha}(
        B(0,4\epsilon))} + \epsilon^{-1} \norm{\overline{\varphi_\epsilon
          \ast w}}{L^2_{-\alpha}(
        B(0,4\epsilon))}.
    \end{multline*}
    The first term vanishes due to density of test functions, the
    second term vanishes following the weighted Poincar\'e inequality (Theorem~4.5 in~\cite{Morin}) and the third
    term also vanishes trivially. A similar argument (without adding
    and subtracting the weighted mean as above) shows that the
    third term in \eqref{prop1eq1} will also vanish as $\epsilon \to
    0$.
  \end{proof} }

Having established the density in $W_{-\alpha}$ of approximations of
type~\eqref{cinfapprox}, we now derive a bound for the error
$\big\|\delta - \tdelta_H\big\|_{(W_{-\alpha})^\ast}$. For a fixed $\mu$
consider the approximating sequence $w_{\epsilon, \mu} \to w$ for $w \in
W_{-\alpha}$. {Choose a {$\delta_H \in H^s_0(\Omega)$} which is compactly
supported in a ball of size $H>0$ around the origin, and let $\tdelta_H \in (H^s_0)^\ast \cap (W_{-\alpha})^\ast$
be the associated linear functional via \eqref{rieszrep}}. Then we have by
Taylor's theorem that
{\begin{align*}
  \absv{
    \frac{(\tdelta_H -
      \delta)(w_{\epsilon,\mu})}{\|w_{\epsilon,\mu}\|_{W_{-\alpha}}}} 
  = \Bigg| w_{\epsilon,\mu}(0)
  \innerprod{\delta_H}{\IND{\Omega}}{\Omega} 
  & + \sum_{1 \le \absv{\alpha}\le m} 
  \frac{\partial^\alpha w_{\epsilon,\mu} (0)}{\alpha !}
  \innerprod{\delta_H}{\mb{x}^\alpha }
  {\Omega}
  + \bigg( \delta_H, \sum_{|\beta|=m+1}R_\beta(\mathbf{y})
  (\mb{x})^\alpha \bigg)_\Omega \Bigg|, \nonumber
\end{align*} }
where $R_\beta$ is the Taylor remainder of $w_{\epsilon,\mu}$.  If we
choose $\delta_H$ to satisfy the continuous moment conditions
\begin{equation}
    \innerprod{\delta_H}{ \IND{\Omega} }{\Omega} = 1
    \qquad \text{and} \qquad 
    \innerprod{\delta_H}{\mb{x}^\alpha } {\Omega} = 0 
    \qquad \text{for } 1 \le \absv{\alpha} <m,  
\end{equation} 
up to order $m$, then it follows that
 $ \absv{(\delta - \tdelta_H)(w_{\epsilon,\mu})} \lesssim H^{m+1} \sup 
  \absv{ \partial^{m+1} w_{\epsilon,\mu}}.$
But from \eqref{cinfapprox} and \eqref{cutoffbounds} we see that
\begin{equation}
  \label{derivativebound}
  \absv{\partial^{m+1} w_{\epsilon,\mu}} \leq 
  C(w) \epsilon^{\mu - m - 1},   
\end{equation}
where $C(w)>0$ depends on $w$, and so
\begin{equation*}
  \absv{(\delta - \tdelta_H)(w_{\epsilon,\mu})} \leq  H^{m+1} C(w)
  \epsilon^{\mu - m -1}. 
\end{equation*}

In applications, the parameter $\epsilon$ can be { informally} identified with the PDE
discretization parameter $h$, and the support of $\delta_H$ is varied
with $h$.  For example, consider a specific relationship between the
support of the regularized delta and $\epsilon$, letting $H =
\bigO(\epsilon^\beta)$ for fixed $\beta>0$. Since
$\inf_{\mu>\alpha}H^{\frac{1}{\beta}(\mu + (\beta-1)( m +1))} =
H^{\frac{1}{\beta}(\alpha + (\beta-1)( m +1))}$, we have the bound
\begin{equation}
  \label{uniformboundedness}
  \absv{(\delta - \tdelta_H)(w_{\epsilon,\mu})} \le C(w,
  \mu)H^{\frac{1}{\beta}(\alpha + (\beta-1)( m +1))}, 
\end{equation}
where $C(w,\mu)$ depends on $w$ and $\mu$ but not on $H$.  Applying the
uniform boundedness principle (Lemma 2.3 in~\cite{Brezis}) gives
\begin{equation}
  \label{uniformboundedness2}
  \begin{aligned}
   \lim\limits_{H\rightarrow0} \absv{(\delta - \tdelta_H)(w_{\epsilon,\mu})} 
     \leq C(w,\alpha)
    \liminf\limits_{H\rightarrow0} H^{\frac{1}{\beta}(\alpha + (\beta
      -1)(m+1))}.
  \end{aligned}
\end{equation}
From this result, we can see that if $\tdelta_H$ satisfies the moment
conditions, then $\tdelta_H\rightarrow \delta$ at
$\bigO(H^{\frac{1}{\beta}(\alpha + (\beta -1)(m+1))})$ in the
$W_{-\alpha}^\ast$ norm.  It is important to note that this is a {\it
  different rate} than that obtained in the weak-$\ast$ sense.

Estimate \eqref{uniformboundedness} is a central result of this
paper. It allows us to precisely quantify the interplay between the
support $H$ of the regularized delta distribution, and the mesh size $h$
of the PDE discretization. One can take $\epsilon\equiv h$ to be of the
same order as the resolution of the numerical scheme used to solve {\it
  Problem~2}; then the case $\beta <1$ corresponds to taking the
support of $\delta_H$ much bigger than the spatial resolution of the
PDE. From \eqref{uniformboundedness} we see immediately that
convergence of $\tdelta_H\rightarrow \delta$ is poor in this case.  The case when
$\beta \approx 1$ is the most interesting one, wherein the support of
$\delta_H$ and the mesh spacing $h$ are comparable in size,
corresponding to the approach most commonly used in practical
applications \cite{beyer-leveque, tornberg-engq, tornberg}. Then we
expect to obtain order $\alpha$ convergence regardless of how many
moments $m$ are satisfied. We provide numerical evidence of this
convergence rate in Section~\ref{Numerics:weighted-sobolev}. The case
$\beta>1$ can lead to improved convergence as long as $\delta_H$
satisfies moment conditions for large enough $m$.  This situation is
difficult to observe in practice because it corresponds to the support
of the regularization $H\ll h$. Hence, if we use a quadrature rule based
on the PDE mesh size $h$, the quadrature error will typically overwhelm
the calculation. We emphasize that in this setting, the error
$\|u-u_H\|_X$ due to {\it regularization} would be well-controlled in
any reasonable norm, and it is the {\it discretization error}
$\|u_H-u_{H,h}\|$ that dominates due to a poor choice of quadrature.

\subsection{Moment conditions}
\label{subsection:moment-conditions}

We have just seen that if $\delta_H$ is restricted to be a sufficiently
smooth and integrable function and is compactly supported inside a set
of diameter less than $H$, then satisfying the moment conditions
\eqref{momentsgeneral} will lead to a $\bigO(H^{m+1})$ rate of
convergence for $\tdelta_H\rightarrow \delta$ in the weak-$\ast$ sense,
and $\bigO(H^{\frac{1}{\beta}(\alpha + (\beta -1)(m+1))})$ in the
$W_{-\alpha}^\ast$ norm, with the latter estimate degenerating as
$\alpha \rightarrow \frac{n}{2}-1$. The number $m$ of moment conditions
satisfied depends on the (problem-specific) choice of smoothness
parameter $s$. If one uses a weighted Sobolev norm, this choice is
related to the weight $\alpha$.

In this subsection, we recall important features of the continuous moment
problem that will be employed in our constructions. We make concrete the
class of regularizations we are concerned with.

\begin{definition}
  \label{compactmomentconditions} 
  Let $\delta_H\in \H^s_0(\Omega)$ and $\tdelta_H\in \H^{-s}$,
  defined as in \eqref{rieszrep}, be a regularization of $\delta$.
  We say that $\tdelta_H$
  is a distribution satisfying \emph{compact $m$-moment conditions} if
  and only if $\delta_H\in \H_0^s(\Omega)$ satisfies the following
  conditions: $\delta_H$ is compactly supported in a ball
  $B(0,H)\subset \Omega$ and
  \begin{subequations}
    \label{momenttrunc}
    \begin{align}
      \label{momenttrunc1}
      \left(\delta_H,\IND{\Omega}\right)_{\Omega}
      &= \innerprod{\delta_H}{\IND{\Omega} }{B(0,H)} = 1, \\
      \label{momenttrunc2}
      \innerprod{\delta_H}{\mb{x}^\alpha }{\Omega} 
      &= \innerprod{\delta_H}{\mb{x}^\alpha }{B(0,H)} = 0, 
    \end{align}
  \end{subequations}
  for $1 \le \absv{\alpha} \le m$.
\end{definition}%

As we have already seen in \eqref{momentsgeneral}, if a regularization
$\tdelta_H$ satisfies the compact $m$-moment condition for some fixed $m
\ge 1$, then from \eqref{deltaerror} we have that $ \absv{\delta(\phi)-
  \tdelta_H(\phi)} \le C(\phi, m) H^{m+1}.$
This shows that satisfying a larger number $m$ of
compact $m$-moment conditions implies weak-$\ast$ convergence
(convergence in the sense of distributions) of higher order as $H \to
0$. A similar result holds for convergence in the $W_{-\alpha}^\ast$
norm for $\beta >1$.

The problem of identifying a $\tdelta_H$ that satisfies the compact
$m$-moment conditions for fixed $m \in \mathbb{N}$ and $H>0$ lies in a
well-known problem class that we define next.
\begin{definition}{\bf (Finite dimensional moment
    problem~\cite[p.~141]{kabanikhin})} Let $\mathcal{X}$ be a Hilbert
  space and fix $m\in \mathbb{N}$.  Given a set of linearly independent
  functions $\{\varphi_i\}_{i=0}^m$ in $\mathcal{X}$ and scalars $c_i$ for
  $i=0,\dots, m$, find the element $q \in \mathcal{X}$ such that
  \begin{equation}
    \label{momentproblem}
    \innerprod{q}{\varphi_i}{\mathcal{X}} = c_i \qquad \text{for } i
    = 0,\dots,m. 
  \end{equation}
\end{definition}%
Clearly, the existence of solutions to the finite dimensional moment
problem depends on the choice of $\{\varphi_i\}_{i=0}^m$, for which 
solvability is well-studied~\cite{krein, kabanikhin}.
Suppose that $\{\varphi_i\}_{i=0}^m$ form a linearly independent set in
a Hilbert space $\mathcal{X}$, and suppose further that
$\{\psi_k\}_{k=0}^\infty$ form a Riesz basis of unit vectors in
$\mathcal{X}$. Moreover, let $\Span\{\varphi_i\}_{i=0}^{m} = \Span \{
\psi_{k}\}_{k=0}^m$. Then, any $\bar{q} \in \mathcal{X}$ of the form
$\bar{q} = \sum_{j=0}^m \beta_j \psi_j$
is a solution to the finite moment problem provided
that $\beta_j$ for $j=0,\dots,m$ solve the linear system 
\begin{equation} 
  \label{linearsysmoments}
  \sum_{i=0}^m \innerprod{ \varphi_i}{\psi_j}{\mathcal{X}} \beta_j =
  c_i \qquad \text{for } i=0,\dots,m.
\end{equation}
If $\mathcal{X}$ is infinite dimensional then the solution $\bar{q}\in
\mathcal{X}$ is not unique because 
\begin{equation}
  \label{nonuniqueness}
  \tilde{q} = \bar{q} + \beta_{j+1} \psi_{j+1}
\end{equation}
also satisfies the finite moment problem if $\psi_{j+1}$ is
orthogonal to $\Span \{ \psi_0, \dots \psi_m\}$.

We solve the finite dimensional moment problem with $\mathcal{X}=L^2(\Omega)$.
We emphasize that \eqref{momenttrunc} are the {\it continuous} moment
conditions.
 One
can replace the $L^2$ inner products in \eqref{momentproblem} by 
suitable quadrature rules and thus obtain analogous {\it discrete}
moment conditions. However, the use of the continuous moment conditions
allows greater flexibility in designing problem-specific regularized
approximations $\tdelta_H$.

%%%%%%%%%%%%%%%%%%%%%%%%%%%%%%%%%%%%%%%
%%%%%%%%%%%%%%%%%%%%%%%%%%%%%%%%%%%%%%%
%%%%%%%%%%%%%                                     %%%%%%%%%%%%%%
%%%%%%%%%%%%%         Section 3                   %%%%%%%%%%%%%%
%%%%%%%%%%%%%                                     %%%%%%%%%%%%%%
%%%%%%%%%%%%%%%%%%%%%%%%%%%%%%%%%%%%%%%
%%%%%%%%%%%%%%%%%%%%%%%%%%%%%%%%%%%%%%%

%\input{Section3}
\section{High-order approximations to the delta distribution} 
\label{section:regularizations}

The previous discussions provide a unified framework within which to
construct regularizations $\tdelta_H$. We will use \eqref{momenttrunc}
and the solvability of
the finite moment problem to construct $\delta_H\in \H^s_0(\Omega)$,
and use \eqref{rieszrep} to obtain $\tdelta_H\in \H^{-s}(\Omega)$ that satisfy the
compact $m$-moment conditions. If needed, we can take advantage of the
non-uniqueness demonstrated in \eqref{nonuniqueness} to ensure that
$\delta_H\in \H^{-s}(\Omega)$. Concretely, for some given instance of
{\em Problem~1}, we need to perform the following steps:
\begin{itemize}
\item Decide on a regularity index $s>0$ appropriate for the specific
  problem. Then fix $m\in \mathbb{N}$ to determine the approximation
  order for $\tdelta_H$.
\item Pick a domain $\Omega_H \subset \Omega$ with diameter $H$ that
  contains the origin.
\item Pick a convenient orthonormal basis $\{\psi_k\}_{k} \in
  L^2(\Omega_H)$. For example, if $\Omega_H$ is an interval then one may
  use a basis of polynomials or trigonometric functions.
\item Set $\delta_H = \sum_{j=0}^p \beta_j \psi_j$ with $p>m$.
\item Solve the linear system consisting of
  \begin{equation}
    \begin{aligned}
      \sum \beta_k \innerprod{\psi_k}{ \IND{\Omega} }{\Omega}  = 1
      \qquad \text{and} \qquad
      \sum \beta_k \innerprod{\psi_k}{\mb{x}^\alpha }{\Omega} = 0,
    \end{aligned}
  \end{equation} 
  for $1 \le \absv{\alpha} <m$, with possibly added constraints to
  ensure that $\delta_H$ has the desired regularity as a function on
  $\Omega$.
\end{itemize}
Additional symmetry constraints may also be incorporated into this
construction. In the next two subsections, we examine the
following two classes of regularizations in $\mathbb{R}^n$:
{\it radially symmetric} and {\it tensor-product} distributions. The latter
possess mirror symmetry across the Cartesian axes.  For both classes,
we show how simple ideas underlying the solution of the continuous
finite moment problem can be used to construct various polynomial
and trigonometric regularizations that satisfy the compact $m$-moment
conditions. It is clear that for both classes, if $\tdelta_H$ satisfies
the compact $2m$-moment condition, it automatically satisfies the
compact $(2m+1)$-moment conditions due to symmetry.

\subsection{Radially symmetric approximations}

The delta distribution $\delta$ is radially symmetric and so it is
natural to seek regularizations $\tdelta_H $ that are also radially
symmetric, as well as homogeneous in the following sense.  Let
$\delta_H$ be defined as
\begin{equation}
  \label{scaleddelta}
  \delta_H(\mb{x}) = \begin{cases}
    \displaystyle
    \frac{1}{H^n} \, \eta_{m} (\mb{x}/H), & \text{for } \mb{x} \in B(0,H),\\[0.1cm]
    0, & \text{otherwise},
  \end{cases}
\end{equation}
where $\eta_m(\mb{z}) : B(0,1) \rightarrow \mathbb{R}$. Here, the
subscript $m$ denotes the number of moment conditions satisfied by
$\delta_H$.  We will need to add additional conditions on $\eta_m$ to
ensure $\delta_H$ has certain desirable continuity properties; for
example, that $\delta_H$ belongs to $\mathcal{H}^s_0(\Omega)$.

Next, we change variables by letting $\mb{x}=H\mb{z}$ so that the moment
conditions \eqref{momenttrunc} can be written in terms of the rescaled
function $\eta_m$ on the unit ball $B(0,1)$ centered at the origin
\begin{equation}
  \label{momentscaled}
    \innerprod{\eta_{m}}{ \IND{\tilde{\Omega}} }{B(0,1)} = 1,  \quad\mbox{and}\quad 
    \innerprod{\eta_m}{\mb{z}^\alpha } {B(0,1)} = 0, \quad \mbox{for}\quad 1
    \le \absv{\alpha} \le m.   
\end{equation}
where $\bar{\Omega}:= \{\mb{z}~\vert~H \mb{z} \in \Omega\}$. 
Because $\eta_m$ is radially symmetric, we can write $\eta_m(\mathbf{z})
\equiv \eta_m(|\mathbf{z}|)$ where the context is clear. Working in the
unit ball in spherical coordinates, the moment conditions
\eqref{momentscaled} reduce to the following one-dimensional integrals
\begin{equation}\label{momentsymm}
    \nu(n) \innerprod{\eta_{m}(r)}{ r^{n-1} }{r \in [0, 1]} = 1 \quad \mbox{and}\quad 
    \innerprod{\eta_{m}(r)}{r^{|\alpha|+n-1} } {r \in [0,1]} = 0 \quad 
    \text{for } 1 \le |\alpha| \le m, 
\end{equation}
where $r$ is the radial coordinate and $\nu(n)$ is the area of the unit
ball $B(0,1)$ in $\reals^n$.
The role of $\{\phi_k\}_{k=0}^m$ in the finite moment problem \eqref{momentproblem} is now
played by the monomials $r^k$.
  
We now have considerable freedom in how to solve the finite moment
problem.  Different solutions $\eta_m$ can be obtained depending on the
choice of basis functions $\{\psi_j\}$.  We can enforce additional
continuity conditions on $\eta_m$ by adding contributions from the
orthogonal complement of $\Span{\phi_k}$. We note that this freedom is
an advantage that derives directly from our working with continuous
moment conditions. We also emphasize that the support $H$ of the
regularizations can be chosen {\it independently of the mesh-size}
$h>0$. Indeed, at this juncture there is no underlying discretization
yet applied to the PDE operator.  This feature allows us to explore good
choices of $H$, an issue that will be addressed later in this section.

\subsection{Tensor-product approximations}

Another approach that has been used in the literature for constructing
approximants $\tdelta_H$ in higher dimensions is the product formula for
discrete regularizations of the delta distribution due to Peskin
\cite{peskin}.  As well as being employed in the immersed boundary
framework, this tensor-product approach has been applied more widely in
the literature~\cite{tornberg, Walden}.  We study this class of
approximations next, developing the analogous formulation for continuous
regularizations.  In particular, we aim to construct an approximation
$\tdelta_H(\mb{x})$ to the delta distribution in $\reals^n$ using
a tensor product of lower-dimensional approximations $\tdelta_H(x_i)$.
\begin{definition}
  {\bf (Tensor product of distributions \cite{Friedman})}
  Consider a test function $\phi(\xi, \zeta) \in \Dee(\reals^\ell \times
  \reals^k)$. Given two distributions $S_\xi \in \Dee^*(\reals^\ell)$ and
  $T_\zeta \in \Dee^*(\reals^k)$, their tensor product in
  $\Dee^*(\reals^\ell \times \reals^k)$ is defined as
  \begin{equation*}
    \label{tensorproddist}
    S_\xi \otimes T_\zeta (\phi) = S_\xi ( T_\zeta (\phi(\xi,\zeta))).
  \end{equation*}
  If the test function has the separable form 
  $\phi(\xi, \zeta) = \phi_1(\xi)\phi_2(\zeta),$
  then we have 
  \begin{equation*}
 S_\xi \otimes T_\zeta (\phi) = S_\xi (\phi_1) T_\zeta (\phi_2) =
    T_\zeta \otimes S_\xi (\phi). 
  \end{equation*}
\end{definition}%
Based on this definition, a {tensor-product approximation} $\tdelta_H$
of $\delta$ has the form
\begin{equation*}
  \label{generaltensorproddelta}
  \tdelta_{H}(\mb{x}) := \tdelta_H(x_1) \otimes \tdelta_H(x_2)
  \otimes \cdots \otimes \tdelta_H( x_n) , 
\end{equation*}
where $\tdelta_H(x_k)$ are 1-D regularizations of the delta distribution
in $\mathbb{R}$.  We require that each 1D approximation $\tdelta_H$
satisfies the compact $m$-moment conditions, with the associated
$\delta_H$ being supported
on $[-H,H]$. The tensor product $\tdelta_H(\mb{x})$ is therefore
supported on the hypercube $[-H,H]^n$.

The problem of constructing the tensor product regularizations has now
been reduced to that of finding a 1D regularization, which can be done
in many ways. For example, one could construct an even function
$\tdelta_H(x)$ on $[-H,H]$ satisfying the compact $m$-moment
conditions, which will converge to the 1D $\delta$ with order
$\bigO(H^{m+1})$.  Taking a tensor product of these leads to a
regularization which we might hope converges to $\delta$ in
$\mathbb{R}^n$.  Unfortunately, since $[-H, H]^n \not \subset
B(0,H;\reals^n)$, the product of such distributions does {\it not}
satisfy even the compact 0-moment condition in \eqref{momenttrunc} if we
view this as a distribution in $\mathbb{R}^n$.
 
One remedy for this problem is to take the 1D approximations to have
support on a smaller interval $[-\tilde{h}, \tilde{h}]$ so that
$[-\tilde{h}, \tilde{h}]^n \subset B(0,H;\reals^n)$. For example, to
construct a tensor-product approximation in two dimensions ($n=2$) we
could simply take $\tilde{h} = H/\sqrt{2}$ so that the square
$[-\tilde{h}, \tilde{h}]^2$ fits inside the ball of radius $H$. Another
solution is to redefine the compact $m$-moment conditions in
\eqref{momenttrunc} so that they hold on the hypercube $[-H,H]^n$.
 
Another potential issue is more problematic: because the tensor product
construction is based on a Cartesian frame, the orientation of the
underlying grid will affect the approximant $\tdelta_H$.  In
Section~\ref{Numerics:wave-2D} we shall see instances where this feature
of a tensor-product distribution leads to unwanted numerical errors,
particularly when solving hyperbolic PDEs.  We emphasize that
this issue arises with {\it any} tensor-product approximation: one must
be careful of the support size when constructing regularizations of the
delta distribution out of lower-dimensional approximations. Suppose that
$\tdelta_H(x_1,\dots, x_m)$ is an approximation in $\reals^m$ supported
on $B(0,H;\reals^m)$, while $\tdelta_H(x_{m+2}, \dots , x_n) $ is an
approximation in $\reals^{n-m}$ that is supported on
$B(0,H;\reals^{n-m})$.  Then the tensor product $\bar{\delta}_H(\mb{x})
:= \delta_H(x_1, \dots, x_m) \otimes \delta_H(x_{m+1}, \dots, x_n)$ is
supported on $B(0,Hh;\reals^m) \times B(0,H;\reals^{n-m}) \not \subset
B(0, H;\reals^n)$.
 
{ So far our discussion indicates that radially symmetric
  and tensor product approximations should yield the same rates of
  convergence in the weak-$\ast$ and weighted Sobolev topologies
  assuming that the definitions of the moment conditions are
  consistent. However, in \cite{tornberg} it was
  observed that the discrete radial approximations can lead to lower order
  numerical convergence. It is clear that this unexpected error is due
to the choice of the quadrature method. Integrating a radially
symmetric regularization on a Cartesian grid can lead to large
numerical errors. Similarly, integrating a tensor product
approximation on a radial grid is likely to be inaccurate. Then in
practical applications the
choice of the regularization depends on the grid, the quadrature
method and also the PDE operator (see section 4.3.1 below).}

\subsection{Examples}

We now illustrate the preceding ideas through a number of concrete
examples.  It will prove useful to introduce some new notation, denoting
by $\eta_{m,p}(r)$ a polynomial of degree $p$ in $r\in [0,1]$ that
satisfies the finite moment problem \eqref{momentsymm} with $m$
conditions.
 
\subsubsection{Radially symmetric polynomial regularizations} 
\label{Section:etadefinitions}

We begin with the simplest case, namely that of the radially symmetric
regularizations satisfying the compact $m$-moment conditions. From
\eqref{scaleddelta}, we aim to find $\delta_H$ of the form
$\delta_H({\mathbf x}) = \eta_{m,p}(|\mathbf{x}|/H)/{H^n}$ for
$|\mathbf{x}|<H$, and equal to 0 in the rest of $\Omega$.
Problem~\eqref{momentsymm} is a finite moment problem with $\phi_k=r^k$
for $k=1, \dots, m+1$ and $r\in[0,1]$.  Therefore, a good choice of
basis $\psi_k$ is provided by the {\it shifted Legendre polynomials}
on $[0,1]$:
\begin{equation*}
    \label{legendrebasis}
    P_k (r) = (-1)^k \sqrt{2k+1}\; \sum_{j=0}^k \binom{k}{j}
    \binom{k+j}{j} \left(-r\right)^j,
\end{equation*}
which are $L^2-$orthonormal on $[0,H]$. Since we want $\eta_{m,p}(r)$ to
be a polynomial of maximal degree $p$, it must have a {\it finite}
expansion $\eta_{m,p}(r) \equiv \eta_m(r) = \sum_{j=1}^{p}
\beta_jP_j(r)$ in the shifted Legendre polynomials $\{P_k(r)\}_{k=1}^m$.
Then by orthogonality of the polynomials we can reduce
the problem \eqref{momentsymm} to finding $\beta_j$ such that
\begin{align}
  \label{reducedmomentsys}  
  \sum_{j=1}^p \beta_j \innerprod{r^{ n -1}}{P_j}{B(0,1)} = \frac{1}{\nu(n)} 
  \qquad \text{and} \qquad 
  \sum_{j=1}^p \beta_j  \innerprod{r^{\theta + n -1}}{P_j}{B(0,1)}  = 0    \qquad \text{for }
  \theta=1,\dots,m. 
\end{align}
We require that $p=m$ for the system of \eqref{reducedmomentsys} to be
solvable and, as discussed above, the freedom to choose $p \ge m$ allows
us to impose additional smoothness conditions at $r=0$ or $1$.  For
example, if we wanted $\delta_H \in \H_0^s(\Omega)$ for some $s\in
\mathbb{N}$, then we would take $p = m+2s-1$ and append the extra
equations
\begin{align}
  \label{contcond}
  \left. \frac{\partial^k \eta_{m}}{\partial r^k} \right|_{r=1} = 0
  \;\; \text{for } k= 0,\dots,s-1
  \quad \text{and} \quad
  \left. \frac{\partial^{\ell} \eta_{m}}{\partial r^{\ell}} \right|_{r=0} =
  0 \;\; \text{for } \ell= 1,\dots,s-1. 
\end{align} 

In one dimension, the ball $B(0,H)\equiv [-H, H]$ and we seek
$\delta_H(x)$ satisfying the scaling in \eqref{scaleddelta}. Following
the discussion above, if we require a radially-symmetric, compact,
1-moment approximation in $\mathbb{R}$, then $\delta_H$ must be an even
function that satisfies
\begin{equation*}
  \label{1D-m1}
  \begin{aligned}
    \innerprod{\delta_{H}}{ \IND{}}{[-H,H]} & = 1
    \qquad \text{and} \qquad 
    \innerprod{\delta_{H}}{ x}{[-H,H]} & = 0.
  \end{aligned}
\end{equation*}
We can set $\delta_H(x) = \frac{1}{H}\, \eta_{m,p}
\left({|x|}/{H}\right)$. Two common choices for $\eta_{1,p}$ in this
class are obtained by setting the polynomial degree $p=0$ and $1$ in
\eqref{reducedmomentsys}, which yield respectively the piecewise
constant approximation and the hat function. To see this, observe that
the constant function is a zeroth-order polynomial function
$\eta_{1,0}(r)$ on $[0,1]$ that satisfies the compact 1-moment
condition, and from which $\delta_H(x)$ is obtained using
\eqref{scaleddelta}:
\begin{equation*}
  \label{piecewiseconst}
  \eta_{1,0}(r) = \frac{1}{2} \quad \text{for }  r \in [0, 1] 
  \quad \Longrightarrow \quad \delta_H(x) :=
  \begin{cases}
    \frac{1}{2H}, & |x|\leq H, \\ 
    0,            & |x|>    H.
  \end{cases}
\end{equation*}
This approximation automatically satisfies the compact 1-moment
condition due to symmetry, and so the regularization converges to
$\delta$ in the weak-$\ast$ sense at $O(H^2)$.

Similarly, suppose we seek a degree-1 polynomial approximation
$\delta_H(x)$ satisfying the compact 1-moment condition on $[-H,H]$. We
can set $\delta_H(x) = \frac{1}{H}\, \eta_{1,1}({|x|}/{H})$, where
$\eta_{1,1}(r)$ is a linear combination of shifted Legendre polynomials
$P_0(r)$ and $P_1(r)$ for $r\in [0,1]$.  We enforce continuity at $r=1$
(equivalently, $|x|=H$) by enforcing $\eta_{1,1}(1)=0$.  A polynomial
with these properties is
\begin{equation}
  \label{hat-eta}
  \eta_{1,1}(r)= \frac{1}{2}P_0(r)
  -\frac{1}{2\sqrt{3}}P_1(r).
\end{equation}
We then obtain $\delta_H(x)$ as a scaled, even extension of
$\eta_{1,1}(r)$, which is the hat function
\begin{equation*}
  \label{hatdelta}
 \delta_H(x) =
 \begin{cases}
   \frac{1}{H} \eta_{1,1} \left(\frac{|x|}{H}\right) = 
   \left(1 - \frac{x}{H}\right), & \text{if } |x|\leq H, \\
   0, & \text{if } |x|>H.
 \end{cases}
\end{equation*} 
Note that despite being continuous, the hat function only satisfies the
first moment condition from the construction of
$\eta_{1,1}(r)$. Therefore, this regularization converges to $\delta$ at
the same rate, $\bigO(H^2)$, in both the weak-$\ast$ sense and
for the piecewise constant regularization.

Proceeding analogously, we can construct regularizations that satisfy
the compact 2-moment conditions. We work once again with the scaled
distribution $\eta_{m,p}(z)$ and solve the moment equations
\eqref{reducedmomentsys} with $m=2$ and desired $p$ for $z\in [0,1]$. We
recover $\delta_H$ by using \eqref{scaleddelta}. A degree 2 polynomial
that solves \eqref{reducedmomentsys} with $m=2$ is given by
\begin{equation}
  \label{secondmomentdiscont}
  \eta_{2,2}(r) = \frac{1}{2} \left( P_0(r) - \sqrt{3}P_1(r) + 
    \sqrt{5}P_2(r) \right) = \frac{9}{2} - 18 r + 15 r^2, \qquad r\in
  [0,1]. 
\end{equation}
Note that this $\eta_{2,2}(r)$ is non-vanishing at $r=1$, which means
that the corresponding $\delta_H(x)$ is not continuous at $\pm
H$. However, exploiting the non-uniqueness of solutions of the finite
moment problem, we are free to impose an extra condition, which we do in
this case by enforcing continuity at $r=1$. This leads to a {\it cubic}
polynomial
\begin{equation}
  \label{secondmomentcont}
  \begin{aligned}
    \eta_{2,3}(r) & = \frac{1}{2} \left( P_0(r) - \sqrt{3}P_1(r) +
      \sqrt{5}P_2(r) - \frac{3}{\sqrt{7}}P_3(r)
    \right),  \\
    & = -30 r^3 + 60 r^2 - 36 r + 6,
  \end{aligned}
\end{equation}
from which we obtain $\delta_H(x) = \frac{1}{H}
\left(-30\frac{|x|^3}{H^3} + 60 (\frac{|x|}{H})^2 - 36\frac{|x|}{H}
  +6\right)$ for $|x|\leq H$ and $\delta_H(x)=0$ elsewhere.  We
emphasize the fact that the forgoing calculations are quite simple to
perform.

These and a number of other radially-symmetric regularizations in 1 and
2 space dimensions are summarized in Table~\ref{table:delta-summary}.
Some of these are well-known in the literature, but to the best of our
knowledge the other regularizations are reported here for the first
time.

\renewcommand\arraystretch{1} 
\newcolumntype{L}{>{\raggedright\arraybackslash}X}%
\begin{table}[htp]
  \centering\scriptsize
  \begin{tabularx}{1 \textwidth}{@{}| L | L | L| L| L| c| L|@{}}
    \hline
 Symbol & Dim & Type & Moment & Smoothness & Definition & Reference  \\
    \hline\hline
    $\eta_{1,0}$ & 1D & Legendre & $1$ & $L_1$ & $\frac{1}{2}$ & --  \\ \hline
    $\eta_{1,1}$ & 1D & Legendre & $1$ & $C^0$ & $1-r$ & \cite{tornberg, beyer-leveque, Walden} \\ \hline
    $\eta_{2,2}$ & 1D & Legendre & $2$ & $L_1$ & $\frac{9}{2} - 18 r + 15r^2$ & -- \\ \hline
    $\eta_{2,3}$ & 1D & Legendre & $2$ & $C^0$ & $-30 r^3 + 60r^2 - 36r + 6$ & -- \\ \hline
    $\eta_{2,5}$ & 1D & Legendre & $2$ & $C^1$ & $168 r^5 - \frac{945}{2}r^4 + 450 r^3 - 150 r^2 + \frac{9}{2}$ & -- \\ \hline
    $\eta_{1,1}$ & 2D & Legendre & $1$ & $L_1$ & $\frac{6}{\pi} (3-4r)$ & -- \\ \hline 
    $\eta_{1,2}$ & 2D & Legendre & $1$ & $C^0$ & $\frac{12}{\pi} (5 r^2 -8r +3)$ & -- \\ \hline
    $\eta_{2,2}$ & 2D & Legendre & $2$ & $L_1$ & $\frac{12}{\pi} (15 r^2 -20r +6)$ & -- \\ \hline
    $\eta_{2,3}$ & 2D & Legendre & $2$ & $C^0$ & $\frac{-60}{\pi} (7 r^3 -15r^2 +10r -2)$ & -- \\ \hline
    $\eta_{2,5}$ & 2D & Legendre & $2$ & $C^1$ & $\frac{84}{\pi} (24r^5 - 70 r^4 + 70 r^3 - 25 r^2 +1)$ & -- \\ \hline
    $\eta_{1,cos}$ & 1D & Trig. & $1$ & $C^0$ & $\frac{1}{2} (1 - \cos(\pi r))$ & 
    \cite{tornberg, tornberg-quad, beyer-leveque} \\ \hline
    $\eta_{2,cos}$ & 1D & Trig. & $2$ & $C^0$ & 
    See equation \eqref{cos-2} & -- \\ \hline
    $\eta_{2,cos}$ & 2D & Trig. & $2$ & $C^0$ &See equation \eqref{cos-2d-2}
    & -- \\ \hline
  \end{tabularx}
  \caption{Polynomials $\eta_{m,p}(r)$ of degree $p$ (first 10 rows) and
    trigonometric polynomials (last 3 rows) that satisfy
    \eqref{reducedmomentsys}.  The corresponding regularized delta in
    each case is $\delta_H(\mb{x}) := \frac{1}{H^n}\eta_{m,p}
    \left({|\mb{x}|}/{H}\right)$ for $|\mb{x}|\leq H$, with
    $\delta_H(\mb{x})=0$ elsewhere.}  
  \label{table:delta-summary}
\end{table}

\begin{figure}[htbp]
  \centering
  \includegraphics[width=0.75\textwidth]{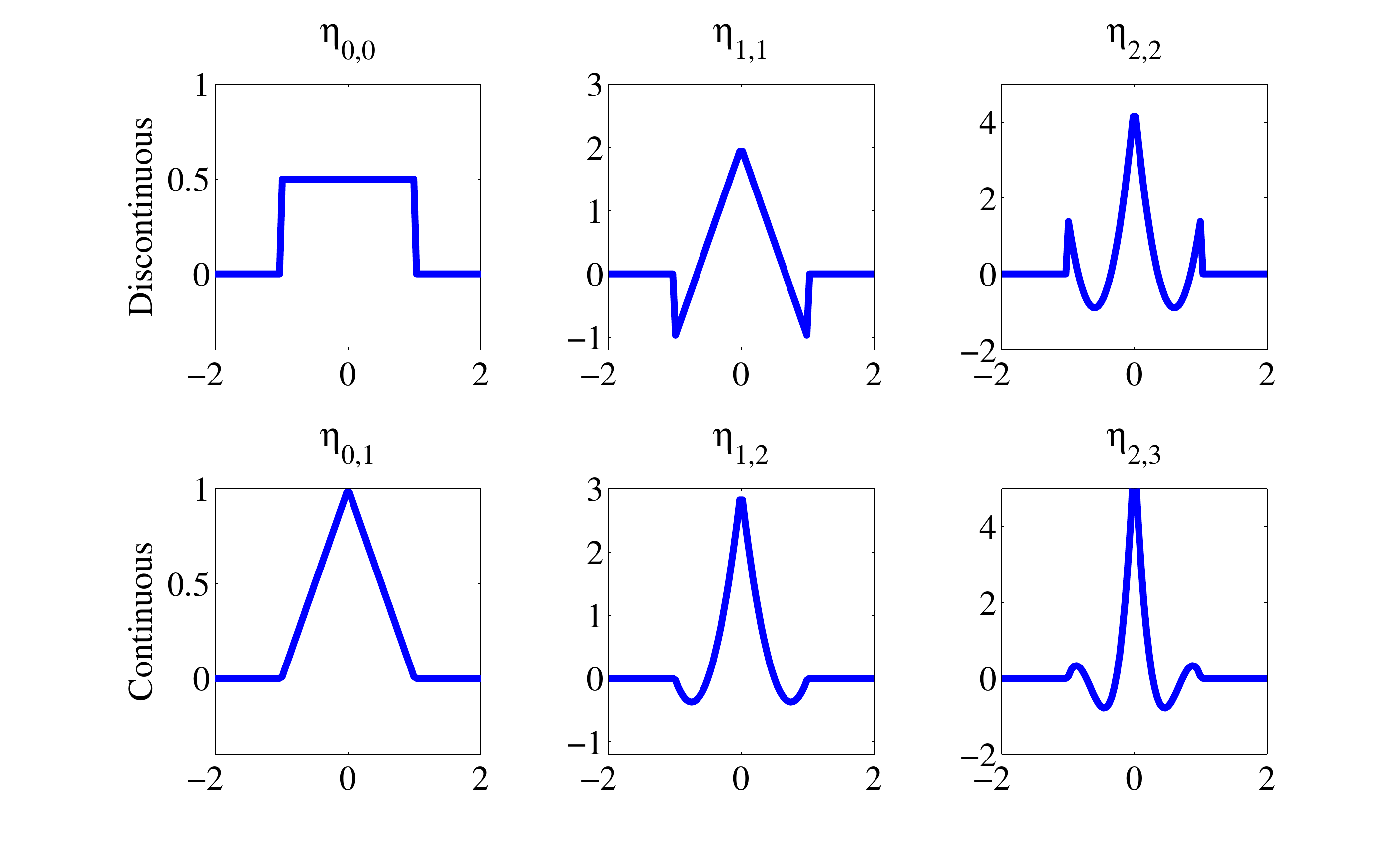} 
  \caption{Various 1D approximations of the delta distribution.}
\end{figure}

\begin{figure}[htbp]
  \centering
  \includegraphics[width=0.95\textwidth]{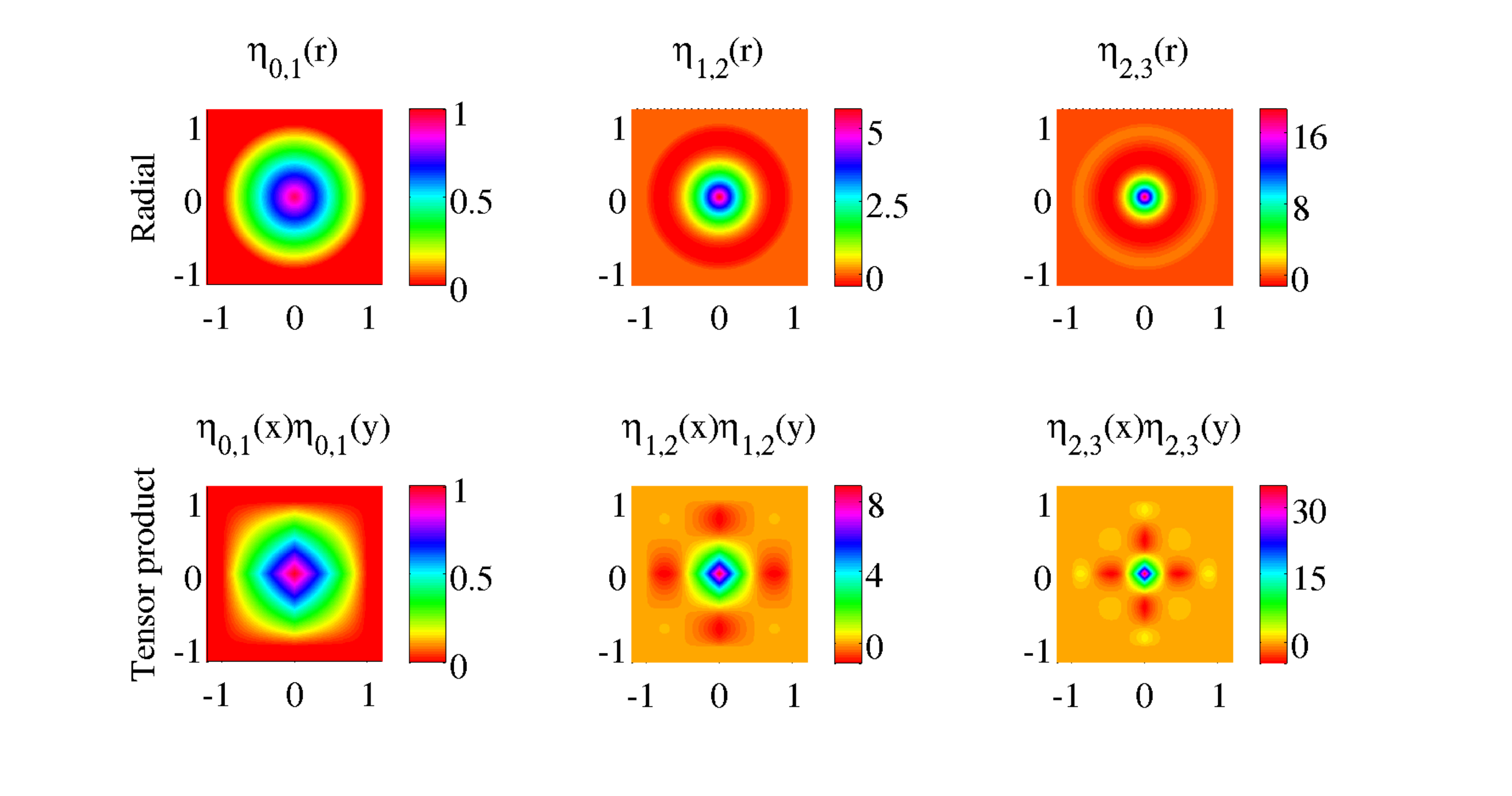} 
  \caption{Comparison of radial approximations to the delta distribution 
    with their tensor product counterparts in 2D.}
\end{figure}

\subsubsection{Trigonometric regularizations}

The shifted Legendre polynomials provide a good basis for solving the
moment problem as long as the approximations we are seeking are likewise
polynomials.  As mentioned before, the solution to the finite moment
problem is not unique and so any other orthogonal basis functions can be
used to solve the system. As an example, we can use trigonometric
polynomials and specifically the cosine basis for even functions in
$L^2(-1,1)$, taking $\{\psi_k(r)\}_k = \{\cos(k\pi r)\}_{k=0}^\infty$.
We seek a regularization of $\delta$ in terms of the scaled,
radially-symmetric functions $\eta_{m,\cos}(r)$ and stipulate that
$\eta_{m,cos}(r)$ has a finite expansion in the cosine basis with
\begin{equation}
  \label{cos-eta}
  \eta_{m,\cos} := \sum_{j=0}^{m+s-1} \beta_j \psi_j(r).
\end{equation}
Here, $s$ is the degree of regularity of the solution to be specified at
$r=1$. Note that in this case we do not need to impose any regularity
condition at the origin since the solution is already in
$C^\infty(B(0,1))$. We may then project the moment conditions onto the
trigonometric basis, similar to our approach in \eqref{reducedmomentsys}, to obtain the moment conditions
\begin{equation*}
  \label{trig-moment-sys} 
  \sum_{j=1}^p \beta_j  \left( r^{\theta + n-1}, \psi_j \right)_{B(0,1)}
  = \frac{1}{\nu(n)}
  \qquad  \text{and} \qquad
  \sum_{j=1}^p \beta_j  \left( r^{\theta + n-1}, \psi_j \right)_{B(0,1)}
  = 0
\end{equation*}
for $\theta= 1,\dots,m$.  

For instance, consider the radially-symmetric trigonometric
regularization that satisfies the compact 0-moment condition
\begin{equation}
  \label{cos-1}
  \eta_{0,\cos} = \frac{1}{2} (1 - \cos(\pi r)) \qquad \text{for } 0\leq
  r\leq 1.  H
\end{equation}
This is the well known cosine function approximation \cite{tornberg, Walden} of the delta
distribution, which automatically satisfies the compact 1-moment
condition as well because of symmetry; consequently, the regularization
$\tdelta_H := \frac{1}{2H}(1-\cos(\pi x/H))$ actually converges to
$\delta$ with rate $\bigO(H^{2})$.  A second trigonometric
regularization that is radially-symmetric and satisfies the compact
2-moment condition is
\begin{equation}
  \label{cos-2}
  \eta_{2,\cos} = \left( \frac{3}{64}\pi^2+ \frac{9}{16}   \right)
  \cos(3 \pi r) + \frac{1}{6} \pi^2 \cos (2 \pi r)+ \left(
    \frac{23}{192} \pi^2 
    - \frac{1}{16} \right) \cos(\pi r) + \frac{1}{2}.
\end{equation}

Performing the analogous calculations in 2D we obtain scaled,
radially-symmetric trigonometric regularizations that satisfy the
compact 0-moment (and 1-moment) conditions
\begin{equation}
  \label{cos-2d-1}
  \eta_{0,\cos} = \frac{2\pi}{\pi^2 - 4} \big(\cos(\pi r)+1\big),
\end{equation}
with the corresponding 2-moment approximation
\begin{multline}
  \label{cos-2d-2}
  \eta_{2,\cos} = \frac{-1}{9 \pi^4 - 104 \pi^2 + 48} \Bigg(
  \frac{81 \pi( 3 \pi^4 - 32 \pi^2 + 48)}{16} \cos(3 \pi r)\\     
  + 2\pi( 9 \pi^4 - 80 \pi^2 + 48) \cos(2 \pi r) 
  + \frac{\pi( 45 \pi^4 + 32 \pi^2 - 48)}{16} \cos(\pi r)
  + 144 \pi \Bigg).
\end{multline}

%%%%%%%%%%%%%%%%%%%%%%%%%%%
%%%%%%%%%%%Numerics for distributions%%%%%%%%%%%%
%%%%%%%%%%%%%%%%%%%%%%%%%%%%%%%%%%%

\subsection{Numerical results: Convergence of $\tdelta_H \rightarrow
  \delta$}
\subsubsection{Weak-$\ast$ convergence}

The results of Section 2.1 are based on the idea of finding distributions
$\tilde{\delta}_H$ that converge to the delta distribution in the
weak-$\ast$ topology.  Therefore, it is fitting at this point to
consider some numerical examples that investigate this weak-$\ast$
convergence.  We use a squared exponential function 
\begin{equation}
  \label{mollifier}
  \phi(\mb{x}) = e^{-\absv{\mb{x}}^2}, \quad \mb{x} \in \reals^2,
\end{equation}
to test the action of the different approximations and measure the
error using
\begin{equation}
  \label{weak-error}
  E_{\text{weak}}(H) := \absv{ \tilde{\delta}_H(\phi) - \phi(0) } =
  \absv{\int \delta_H(\mb{x}) \phi(\mb{x}) d\mb{x} - \phi(0) } , 
\end{equation}
as $H \to 0$. In order to the improve convergence of the quadrature
scheme when $H$ is small, we integrate over twice the support
of the associated $\delta_H$ so the discontinuities remain inside the
domain.  We also apply the following change of variables
\begin{equation}
  \label{quadchangeofvariables}
  \int_{B(0,2H)} \delta_H(\mb{x}) \phi( \mb{x}) d\mb{x} =
  \int_{B(0,2)} \eta_m(\mb{y}) \phi( H\mb{y}) d\mb{y}. 
\end{equation}
MATLAB's integral2 function is used to perform the quadrature with
relative tolerance set to $10^{-10}$.  
We first report the weak-$\ast$ convergence of the radial approximations
and some of the tensor product approximations.  Results of this
experiment are presented in Figure \ref{fig:weakconv}, from which it is
clear that satisfying more moment conditions results in faster
convergence.
\begin{figure}[htbp]
  \begin{subfigure}[t]{0.45\textwidth}
    \centering
    \includegraphics[width=\textwidth]{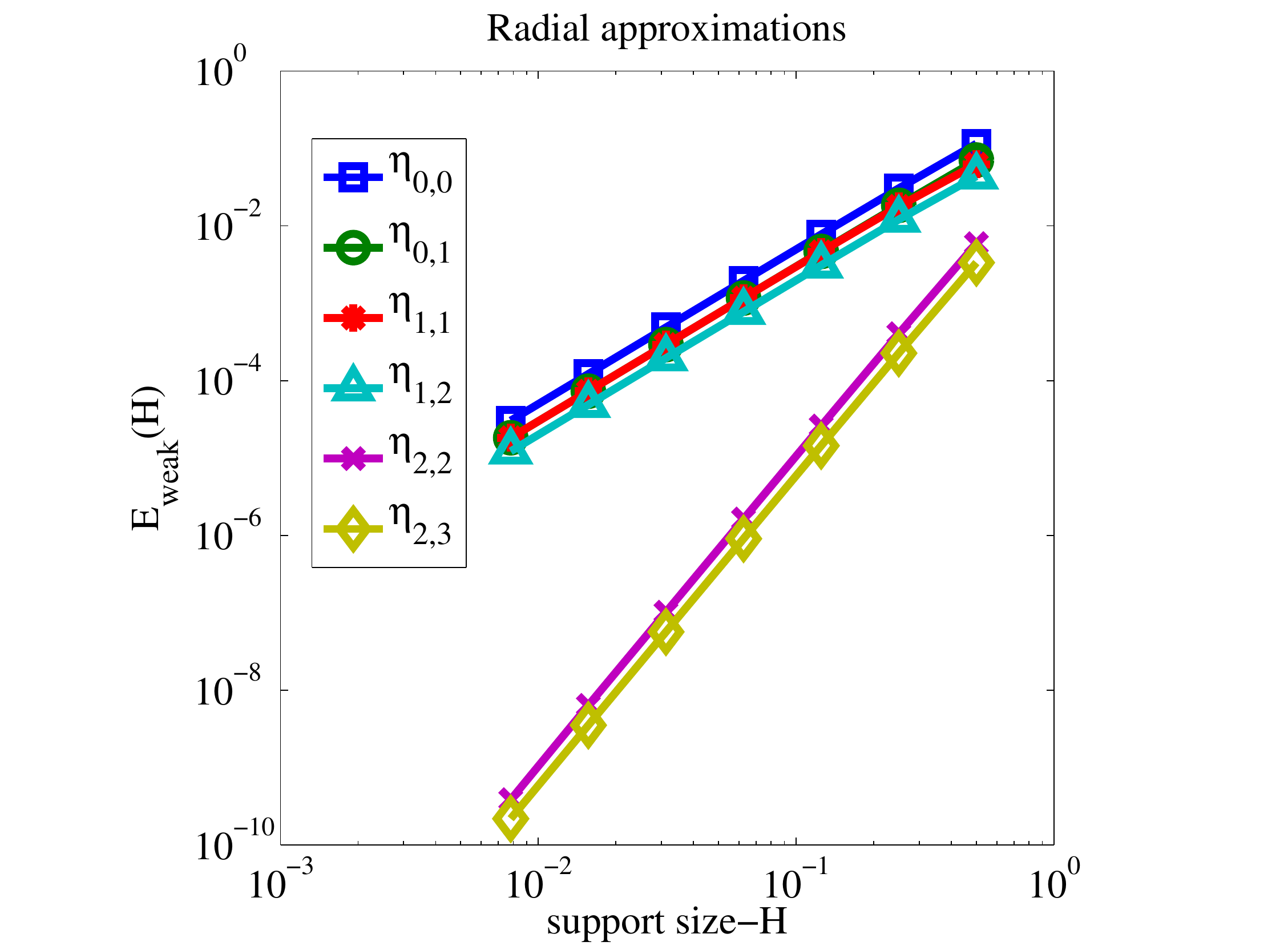} 
  \end{subfigure}
  \begin{subfigure}[t]{0.45\textwidth}
    \centering
    \includegraphics[width=\textwidth]{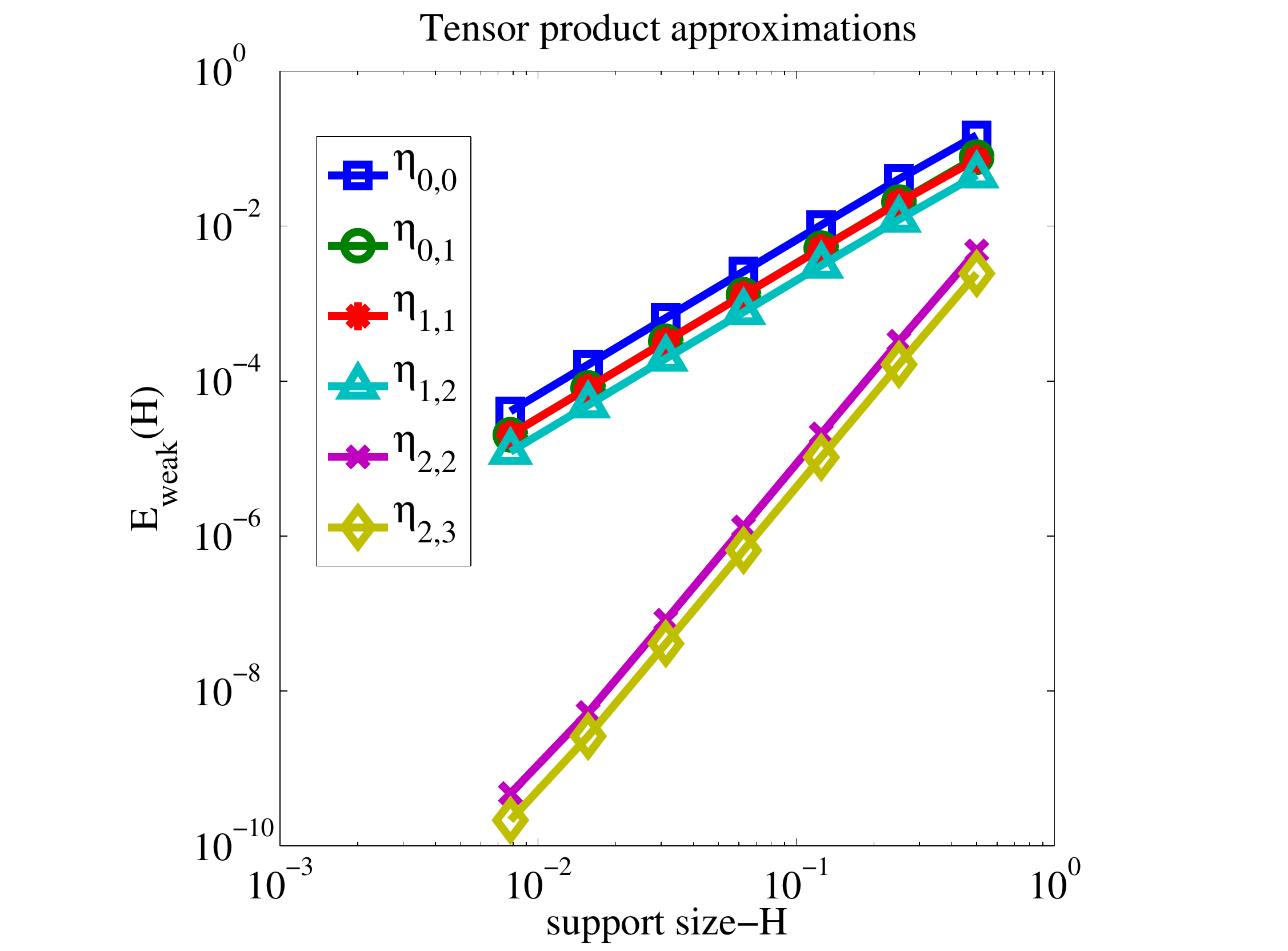}  
  \end{subfigure}
  \caption{Weak-$\ast$ convergence of radial and tensor product
    distributions.}
  \label{fig:weakconv}
\end{figure}

Next, we compare the accuracy of our approximations to a number of other
delta distribution approximants that are commonly used in the literature
(where these regularizations can be derived from discrete moment
conditions \eqref{momentsdiscrete} along with some other additional
conditions being imposed). Specifically, we compare the regularizations
constructed in Section~\ref{Section:etadefinitions} with four common
regularizations obtained by solving discrete moment conditions (see
\cite{tornberg}, for example). In 2D, we construct all regularizations
using tensor products.  Following \cite{tornberg}, we consider two hat
functions with support $[-H,H]$ and $[-2H,2H]$ respectively:
\begin{gather}
  \label{L1}
  \eta_{\text{hat},1} (r) := 
  \begin{cases}
    1 - \absv{r},& |r|<1,\\
    0,& \mbox{otherwise}, 
  \end{cases}
  \qquad \mbox{and} \qquad 
  \eta_{\text{hat},2} (r) := 
  \begin{cases}
    \frac{1}{4} \left(2 - \absv{r} \right), &|r|<2,\\
    0,& \mbox{otherwise}.   \end{cases}
\end{gather}
Both $\eta_{\text{hat},1}$ and $\eta_{\text{hat},2}$ are equivalent
(modulo scaling) to the regularization $\eta_{1,1}(r)$ obtained in
\eqref{hat-eta}.  Both of these hat functions satisfy the same number of
compact $m$-moment conditions, but in practice $\eta_{\text{hat},2}$ is
preferred since more quadrature points are present within its support
\cite{tornberg}.

We consider two additional regularizations defined in \cite{tornberg}
based on the cosine function
\begin{equation}
  \label{cos}
  \eta_{\cos} (z) := \frac{1}{4} \left(1 + \cos(\pi z/2) \right),
\end{equation}
and piecewise cubic function 
\begin{equation}
  \label{cubic}
  \eta_{\text{cubic}} (z) := \left\{
    \begin{aligned}
      & 1 - \frac{1}{2} \absv{z} - \absv{z}^2 + \frac{1}{2} \absv{z}^3,
      \qquad && \absv{z} \le 1,\\
      & 1 - \frac{11}{6} \absv{z} + \absv{z}^2 - \frac{1}{6} \absv{z}^3,
      \qquad &&1 <\absv{z} \le 2,
    \end{aligned}
  \right.
\end{equation}
both of which are supported on $ [-2H, 2H]$.  It is easy to see that
$\eta_{\cos}$ is equivalent to our first moment approximation
$\eta_{1,\cos}$ obtained using the trigonometric basis \eqref{cos-1}. We
expect this approximation to have second-order weak-$\ast$ convergence
{\it if all quadratures were exact}. However, as discussed in
\cite{tornberg}, $\eta_{\cos}$ only has first-order convergence when
used to approximate the delta distribution \cite{tornberg} on a uniform
grid and using a trapezoidal rule discretization for the moment
conditions.

Figure \ref{fig:tensorcomparison} compares the weak-$\ast$
convergence of our tensor product distributions with the other discrete
approximations mentioned above. It is clear that $\eta_{L1}$,
$\eta_{L2}$ and $\eta_{\cos}$ all have second order convergence as
expected. The only nontrivial case is $\eta_{\text{cubic}}$, which is not
directly related to any of our previous approximations. It is
interesting that even though this approximation is built to satisfy
three discrete moment conditions \cite{tornberg}, it appears to satisfy
three continuous approximations as well.
\begin{figure}[htbp]
  \centering
  \includegraphics[width=0.5\textwidth]{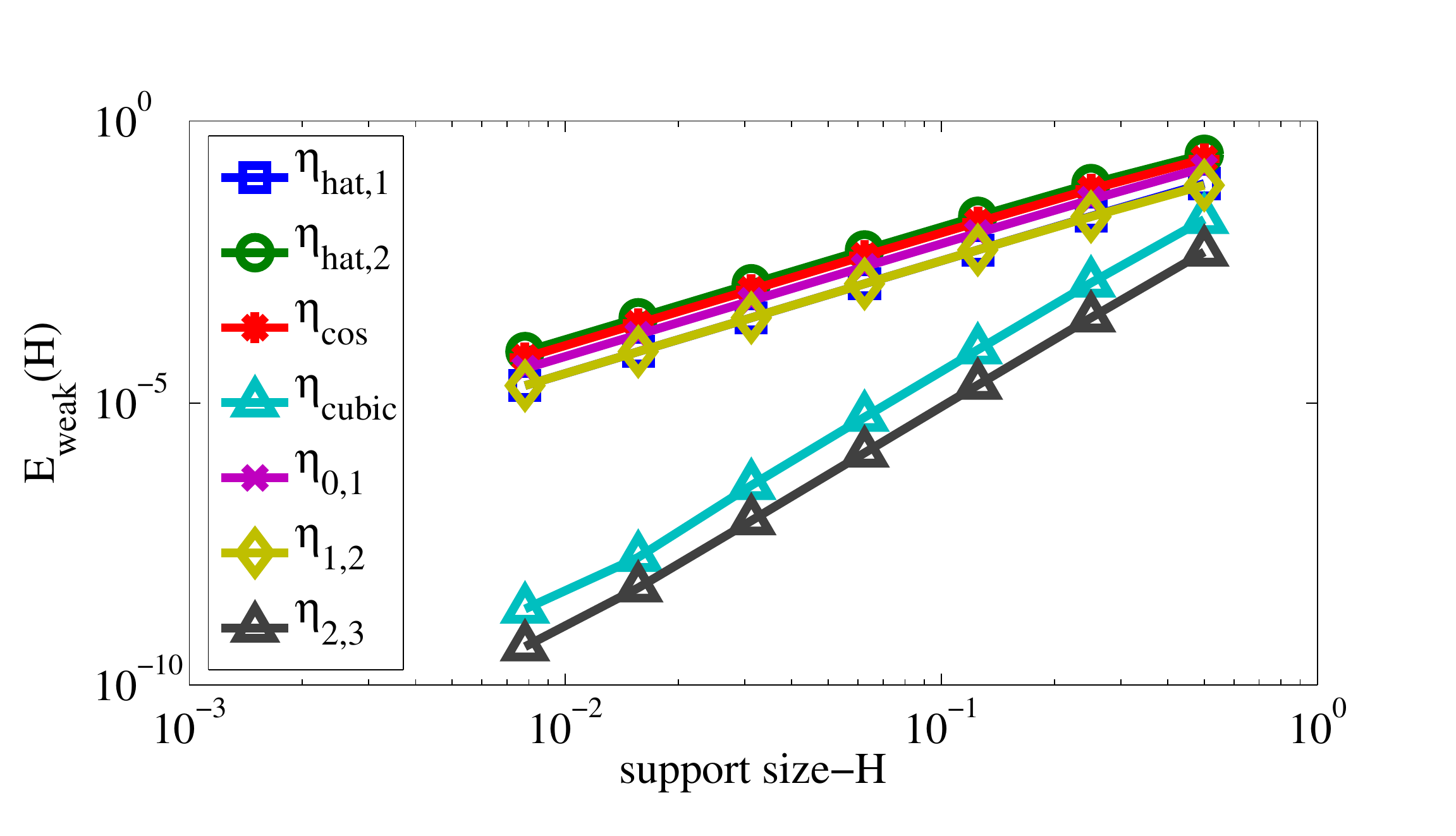}
  \caption{Comparing weak-$\ast$ convergence of the discrete
    approximations in \cite{tornberg} against our continuous tensor
    product distributions. }
  \label{fig:tensorcomparison}
\end{figure}

%%%%%%%%%%%%%%%%%%%%%%%%%%%%%%%%%%%%%%%
%%%%%%%%%%%%%%%%%%%%%%%%%%%%%%%%%%%%%%%
%%%%%%%%%%%%%                                     %%%%%%%%%%%%%%
%%%%%%%%%%%%%         Section 4                   %%%%%%%%%%%%%%
%%%%%%%%%%%%%                                     %%%%%%%%%%%%%%
%%%%%%%%%%%%%%%%%%%%%%%%%%%%%%%%%%%%%%%
%%%%%%%%%%%%%%%%%%%%%%%%%%%%%%%%%%%%%%%

%\input{Section4}
\section{Applications to prototypical PDEs with approximate point
  sources} 
So far, we have been concerned with constructing approximate
distributions that converge to the delta distribution in two specific
modes. We now turn our attention to the error inherited by the solutions
of a PDE where point sources are replaced by approximations
$\tdelta_H$. That is, we wish to examine the error $\|u-u_H\|_X$ in some
norm $X$, where $u$ and $u_H$ refer to solutions of \eqref{Problems} and $H$
denotes the support of $\delta_H$.

In practice, a PDE discretization with mesh size $h$ is used to
approximate solutions to $\Ell u_H = \delta_H$. Suppose we fix $H>0$,
then $\lim_{h\rightarrow 0} u_{H,h} = u_H$ provided we have chosen a
convergent numerical method. If $\tdelta_H$ we choose is regular enough,
then $u_H$ may itself be smooth enough for pointwise comparisons to be
meaningful.  The limit must be taken in an appropriate norm; however in
practice we {\it simultaneously} vary $H$ and $h$, and must therefore be
able to directly compare $u_{H,h}$ with $u$. The solution $u$ of $\Ell
u=\delta$ may possess singularities. For example, if $\Ell$ is the
Laplace operator with Dirichlet conditions on a disk, then $u$ is the
Green's function on the disk.  Because this $u$ has a logarithmic
singularity, we cannot compare $u_{H,h}$ to $u$ in a pointwise sense.
 
We must therefore address {\it Question 4} raised in the Introduction:
in what norm should we measure the convergence of $u_{H,h}\rightarrow
u$?  As might be expected, the answer depends on the PDE operator
$\Ell$.
One choice of norm $\|\cdot \|_H$ comes from taking $\|u -
u_{H,h}\|_{\H_0^s(\Omega)}$, where the value of $s$ depends on $\Ell$
and must be sufficiently large so that $\delta\in \H^{-s}(\Omega)$.
This choice is equivalent to comparing the Fourier coefficients of the
two approximations.  We use this approach to study scalar hyperbolic
problems in Sections~\ref{wave-in-1D} and~\ref{Numerics:wave-2D}, and it
is particularly instructive in the KdV equation which we consider in
Section~\ref{Numerics:KDV}.
  
Another choice of norm is based on comparing functions pointwise in
$\Omega$ away from the support of $\tdelta_H$. In other words, we use
\begin{equation}\label{BHnorm} 
  \|w\|_{BH} := \| w (1 - \chi_{B(0,H)}) \|_{L^{\infty}(\Omega)},
\end{equation} 
where $\chi_\Gamma$ is a usual $C^{\infty}$ cut-off function that takes
the value $1$ on $\Gamma$ and smoothly decays to zero away from
$\Gamma$. We apply this norm in Section~\ref{Section:Helmholtznumerics}
when studying the Helmholtz equation.  The disadvantage of this choice
is that as $H$ changes, so does the definition of the norm.
Another norm that is intermediate between $\| \cdot \|_{BH}$ and
$\| \cdot \|_{\H_0^s(\Omega)}$ is the 
$W_{\alpha}$--norm from Definition \eqref{weightedsobolevspace}.
We can use any of these to replace the norm $\|\cdot\|_X$
in the expression \eqref{reg_error_and_dis_error}.
  
For fixed $H>0$, the behaviour of the discretization error
$\|u_{H}-u_{H,h}\|_X$ will depend on the choice of numerical method and
on the grid parameter $h$. However, even if we pick an excellent
numerical method, if the error due to the regularization $\|u-u_H\|_X$
is not properly controlled then $u_{H,h}$ will not be a good
approximation to $u$. We present examples highlighting this point below.

\subsection{Elliptic PDEs}

In this subsection, we consider the case when $\Ell$ is a linear
second-order elliptic operator with zero boundary conditions. We first
consider the simple situation corresponding to a constant coefficient
operator, and then generalize to the situation where the coefficients
may vary.

Suppose first that we denote by $L$ a constant-coefficient
elliptic
operator. Then, for $H>0$, let $\tdelta_H$ denote a regularization that
satisfies the compact $m$-moment conditions. We then consider the
problems
\begin{align}
  &L u = \delta,     & \text{for } x \in \Omega,\qquad  
  && \text{with } u=0 \quad \text{for } x \in \partial \Omega, \label{ellpde}\\
  &L u_H = \tdelta_H,& \text{for } x \in \Omega,\qquad 
  && \text{with } u_H=0 \quad \text{for } x \in \partial \Omega. \label{ellpderegularized}
\end{align}

\begin{theorem} 
  \label{PDEpointwiseapprox} 
  Let $u$ and $u_H$ solve problems \eqref{ellpde} and 
  \eqref{ellpderegularized} respectively, and let $m$ be the number of
  compact $m$-moment conditions satisfied by $\tdelta_H$.  For all $x \in
  \Omega \backslash B(0,H)$ we have
  \begin{equation}
    \label{pointwiseestimate}
    \absv{u(x) - {u}_H(x)} \le C_m H^{m+1} ,
  \end{equation}
  and therefore $\|u-u_H\|_{BH} = C_m H^{m+1}$, where $C_m>0$ is a
  constant that depends on $\Omega$ and $m$ but not on $H$.
\end{theorem}
\begin{proof}
  The solution of \eqref{ellpde} is $u\equiv \Gee$, the Green's function
  of the elliptic operator $L$ in $\Omega$.  Let $\delta_H$ be an
  approximation of $\delta$ that satisfies the compact $m$-moment
  condition. Then for $x\in \Omega \setminus B(0,H)$,
  \begin{equation}
    \label{proofpointwiseestimate}
    \begin{aligned}
      \absv{u(x) - {u}_H(x) } &= \absv{ \Gee(x) - \int_{\Omega} \delta_H(y)\Gee(x-y)dy} 
      &\le \absv{ C(\Gee, m) H^{m+1}},
    \end{aligned}  
  \end{equation}
  where the last inequality follows from the estimate in
  \eqref{weak-star-estimate} and the fact that the Green's function
  $\Gee$ is infinitely differentiable away from the origin (see Theorem
  6.5 in \cite{mclean}). Therefore, $\|u-u_H\|_{BH} =
  \|(u(x)-u_H(x))(1-\chi_{B(0,H)})\|_{L^\infty(\Omega)} \leq
  C(\Gee,m)H^{m+1}$.
\end{proof}

We also examine the difference $u-u_H$ in a weighted Sobolev norm in
$\mathbb{R}^n$ for $n=2,3$. Our starting point is the recent work of
\cite{Morin} and \cite{Dangelo}.  Specifically, we use the key result in
Section~2.1 of \cite{Morin}: given $F\in W_{-\alpha}^*(\Omega)$, the
constant-coefficient second order elliptic PDE $Lw=F$ in $\Omega$ with zero Dirichlet
data
possesses a unique solution $w\in W_{\alpha}(\Omega)$ provided that $\alpha
\in (\frac{n}{2}-1,1)$ (the result in \cite{Morin} was actually proved
for more general elliptic operators).  Moreover, the solution $w$
satisfies 
\begin{equation}
  \label{morinbound}
  \norm{w}{W_\alpha}\leq C_* \norm{F}{(W_{-\alpha})^*}.
\end{equation}
Here $C_*>0$ is a positive constant that depends on $\Omega$, the PDE
coefficients and $\alpha$; as $\alpha \rightarrow \frac{n}{2}-1$, the
constant $C_*$ blows up. The regularity and bounds can be improved under
certain assumptions on the coefficients of $L$.  Now suppose that $u \in
W_{\alpha}$ solves the PDE when $F=\delta$, whereas $u_H$ solves the
problem when $F=\tdelta_H$, a regularization that satisfies the compact
$m$-moment condition with support size $H$. By linearity and the bound
in \eqref{uniformboundedness}, it is easy to see that
\begin{equation}
  \label{errorbound}
  \norm{u-u_H}{W_\alpha} \leq C_* \norm{\delta -
    \tdelta_H}{(W_{-\alpha})^*} \leq \tilde{C}_*H^{\frac{1}{\beta}(\alpha
    + (\beta-1)(m+1))}.
\end{equation}

We can use this result to interpret the rate of convergence of different
numerical schemes for solving elliptic PDEs.  We can also make an
statement about convergence in the $L^2_{\alpha}$ spaces.  Given the
solution $u \in W_{\alpha}$ and the approximation $u_H$ as before,
suppose that $H=h^\beta$ where $h$ is a discretization parameter for the
PDE, and assume that $\|u-u_H\|_{L^2(\Omega)\setminus B(0,h)}$ is small.
We then have from the weighted Poincar\'e inequality that
\begin{equation}
  \label{L2bound} 
  \norm{u - \tilde{u}}{L^2_{\alpha}(B(0,h))} \leq C(\alpha) h
  H^{\frac{1}{\beta}(\alpha + (\beta -1)(m+1))} =  C(\alpha)
  H^{\frac{1}{\beta}(\alpha + (\beta -1)(m+1)+1)} .
\end{equation}
Consequently, { in 2D and in the limit as $\alpha \to 0$ and $\beta \to 1$}, we
cannot obtain better than first-order convergence in $L^2(\Omega)$.

\subsection{Numerical experiments with the Helmholtz equation} 
\label{Section:Helmholtznumerics}

We now present numerical experiments that support our estimates of the
regularization error in the case
of elliptic PDEs with singular source terms. In particular, we
solve the Helmholtz equation in one and two dimensions with homogeneous
boundary conditions.  Let $u$ and $u_H$ denote solutions of the
problem
\begin{equation}
  \label{primary-helm}
  \Delta u + k_0^2 u =F \quad \text{in } B(0,1) 
  \qquad \text{and } u = 0 \quad\text{on } \partial \Omega,
\end{equation}
with $F=\delta$ and $F=\delta_H$ respectively. Then set the wavenumber
$k_0=10$ and consider solutions of \eqref{primary-helm} on the unit
ball $B(0,1)$ in dimensions $n=1,2$.  Our goal is to study the
convergence of $u_H$ to $u$ using different measures of the error as $H
\to 0$.

The solution $u_H$ of the regularized PDE does not have a closed form
expression, and so we consider a numerical approximation $u_{H,h}$.  In
the following we use {\tt ChebFun} \cite{chebfun-guide} to solve for
$u_{H,h}$ on a fine collocation grid, so that numerical errors $\|u_h
-u_{H,h}\|$ are negligible compared to the approximation error $\|u-u_H
\|$.  For an elliptic problem like the Helmholtz equation, we expect
high regularity away from the source.  Therefore, a good choice of norm
to compare $u$ and $u_{H,h}$ is $\|\cdot\|_{BH}$ as defined in
\eqref{BHnorm}. Because we want to vary both $H$ and $h$, we make a
specific choice of norm
\begin{equation}
  \label{infterrornorm}
  E_{\text{pointwise}}(H):=\|u-u_{H,h}\|_{B\bar{H}}= \norm{ (u -
    u_{H,h})(1 -  \chi_{B(0,\bar{H})}) }{L^\infty(B(0,1))}, 
\end{equation}
where $\chi_{B(0,\bar{H})}$ is the standard $C^{\infty}$-cutoff function
for the ball centered at the origin with radius $\bar{H}$ given by
the largest support size of $\delta_H$ in our experiments.  As a result,
when $H\rightarrow 0$ we are always comparing $u$ and $u_{H,h}$
pointwise over the same set.  We also define the quantity 
\begin{equation}
  \label{conv-ratio}
  R_{\text{pointwise}}(H) := \log_2 \left(
    \frac{E_{\text{pointwise}}(H)}{E_{\text{pointwise}}(H/2)} \right), 
\end{equation}
which will be used with multiple values of $H$ to measure the rate of
convergence of numerical solutions. In particular, as $H \to 0$ the
value of $R_{\text{pointwise}}(H)$ should saturate to the expected rate
of convergence from our analysis.

\subsubsection{Helmholtz in $\reals^1$: Pointwise convergence
  in a deleted neighborhood} 

We begin by considering the Helmholtz equation \eqref{primary-helm} in
1D, with $\Omega=B(0,1)\equiv [-1,1]$.  
The fundamental solution in this case is
\begin{equation}
  \label{Poisson1Dfund}
  u(x) =- \frac{\sin \left( \frac{k_0}{2} (1+  \min \left\{ x, 0
      \right\} )\right) \sin \left( \frac{k_0}{2} \left( 1 -  \max
        \left\{ x,0 \right\}\right) \right)}{k_0 \sin(k_0)},
\end{equation}
and belongs to $\mathcal{H}_0^1([-1,1])$. We take several values of the
support size $H = 1/2^m$ for $m= 2,3,4,5$ and take the cut-off radius
$\bar{H} = 1/4$. Computed results for continuous delta approximations 
$\eta_{\text{cos}}$ and
$\eta_{\text{cubic}}$ in the 1D Legendre basis are presented in Figure~\ref{fig:helm1d}. In order
to study the asymptotic rate of convergence of our approximations we
report values of $R_{\text{pointwise}}(H)$ in Table~\ref{tab:helm1d} for
different support sizes, and these results indicate that the ratios
saturate to the expected rates of pointwise convergence of
\eqref{pointwiseestimate}.

\begin{table}[htbp]  \centering\small
  \begin{tabular}{ | c | c | c | c| c| c| }
    \hline
    \multicolumn{5}{|c|}{$R_{\text{pointwise}}(H)$} & \multirow{2}{*}{expected rate} \\ \cline{1-5}
    support size ($H$) &$1/4$ & $1/8$ & $1/16$ & $1/32$ & \\   
    \hline\hline
    $\eta_{0,1}$          &  1.7770  &  1.9438  &  1.9859  &  1.9965 & 2\\
    \hline
    $\eta_{1,2}$          &  1.4589  &  1.8761  &  1.9696  &  1.9924 & 2 \\
    \hline
    $\eta_{2,3}$          &  3.7442  &  3.9371  &  3.9843  &  3.9961 & 4\\
    \hline
    $\eta_{\text{cos}}$   &  1.3331  &  1.8249  &  1.9558  &  1.9889 & 2 \\
    \hline
    $\eta_{\text{cubic}}$ &  3.3330  &  3.8319  &  3.9579  &  3.9895 & N/A \\
    \hline  
  \end{tabular}
  \caption{Convergence rates for the 1D Helmholtz solution using the
    measure $E_\text{pointwise}(H)$ as defined in \eqref{infterrornorm}
    for pointwise solution error away from the support of the delta
    distribution. The expected rates of convergence are obtained as in
    \eqref{pointwiseestimate} and depend on the number of moment
    conditions that are satisfied by the approximation.} 
  \label{tab:helm1d}
\end{table}

\begin{figure}[htbp]
  \centering
  \includegraphics[width=0.4\textwidth]{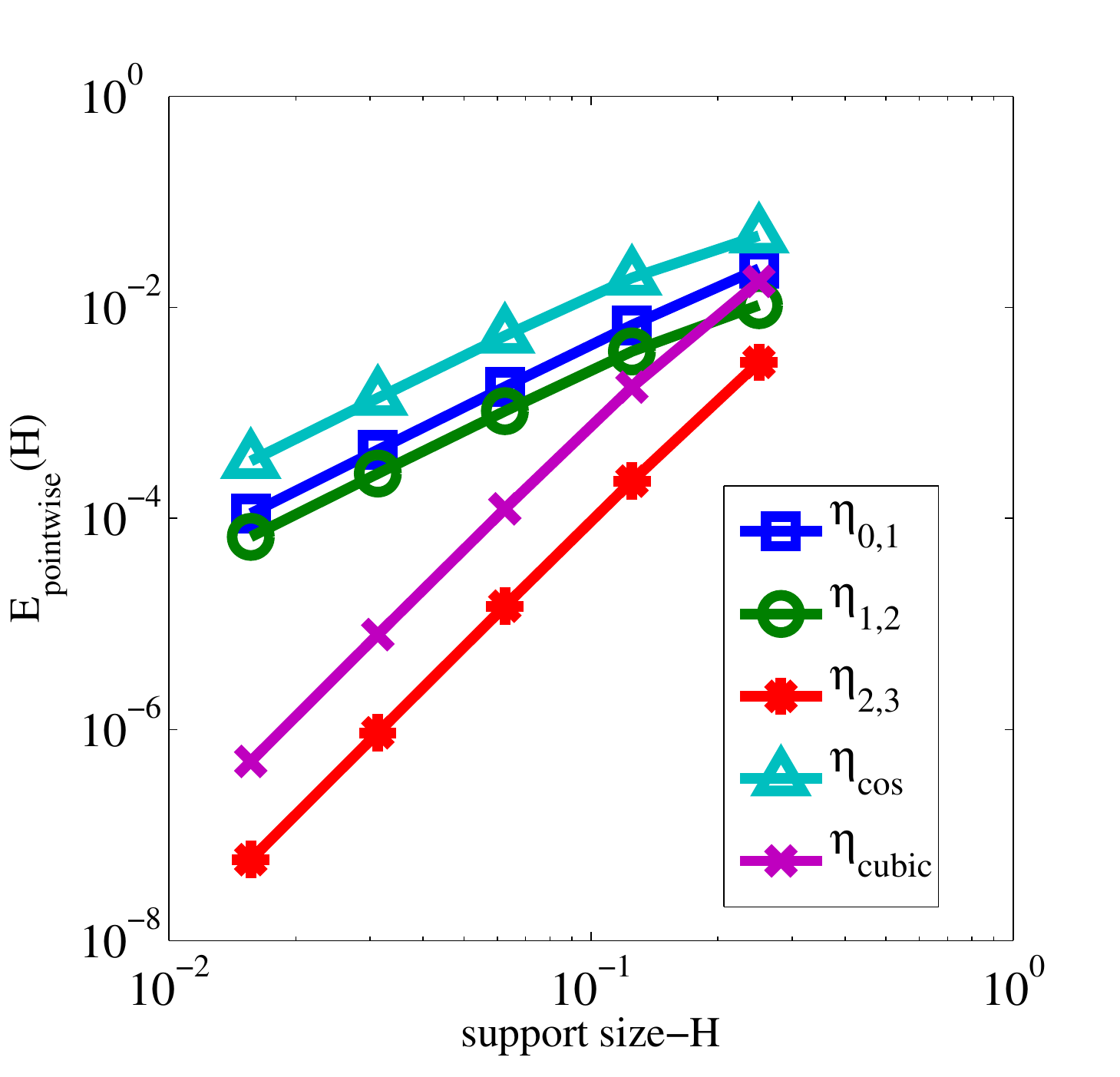}
  \caption{Pointwise error in the 1D Helmholtz solution, with
    three continuous polynomial approximations of the delta distribution
    (see Table~\ref{table:delta-summary}) compared to two approximations
    \eqref{cos} and \eqref{cubic}.}
  \label{fig:helm1d}
\end{figure}

Although these results demonstrate that $\eta_{\text{cubic}}$ and
$\eta_{2,3}$ both exhibit fourth-order convergence, Figure~\ref{fig:helm1d}
shows that the solution using $\eta_{2,3}$ is almost an order of
magnitude more accurate for any given value of the support size
$H$. This hints at a trade-off between choosing a more
regular solution that is easier to compute numerically but is less
accurate, and an approximation that is more difficult to resolve but
gives a more accurate solution.

\subsubsection{Helmholtz in $\reals^2$: Pointwise convergence
  in a deleted neighborhood}

We next consider the Helmholtz equation \eqref{primary-helm} on the unit
disk in 2D,
with the main purpose of this example being to test the radially
symmetric delta approximations in Table~\ref{table:delta-summary}.  The
Green's function for the Helmholtz equation on the unit disk can be
written in terms of Bessel functions as~\cite{duffy}
\begin{equation}
  \label{helmholtz2D_anal}
  u(r,\theta) = -\frac{1}{4} Y_0(k_0 r) + \frac{1}{4}
  \sum_{n=-\infty}^\infty \frac{J_n(0) Y_n(k_0)}{J_n(k_0)} J_n(k_0
  r)e^{in\theta} . 
\end{equation}
Because the radial delta approximations are symmetric, the solution
$u_H$ is also clearly symmetric.

We then perform a numerical convergence study with parameters identical
to those used in the previous one-dimensional Helmholtz example.  Plots
of the pointwise error are presented in Figure~\ref{fig:helm2d} and the
corresponding ratios of successive errors $R_{\text{pointwise}}(H)$ are
listed in Table~\ref{tab:helm2d}. We see once again that the
numerical rates of convergence are in agreement with the estimates from
\eqref{pointwiseestimate}.
\begin{figure}[htbp]
  \centering
  \includegraphics[width=0.4\textwidth]{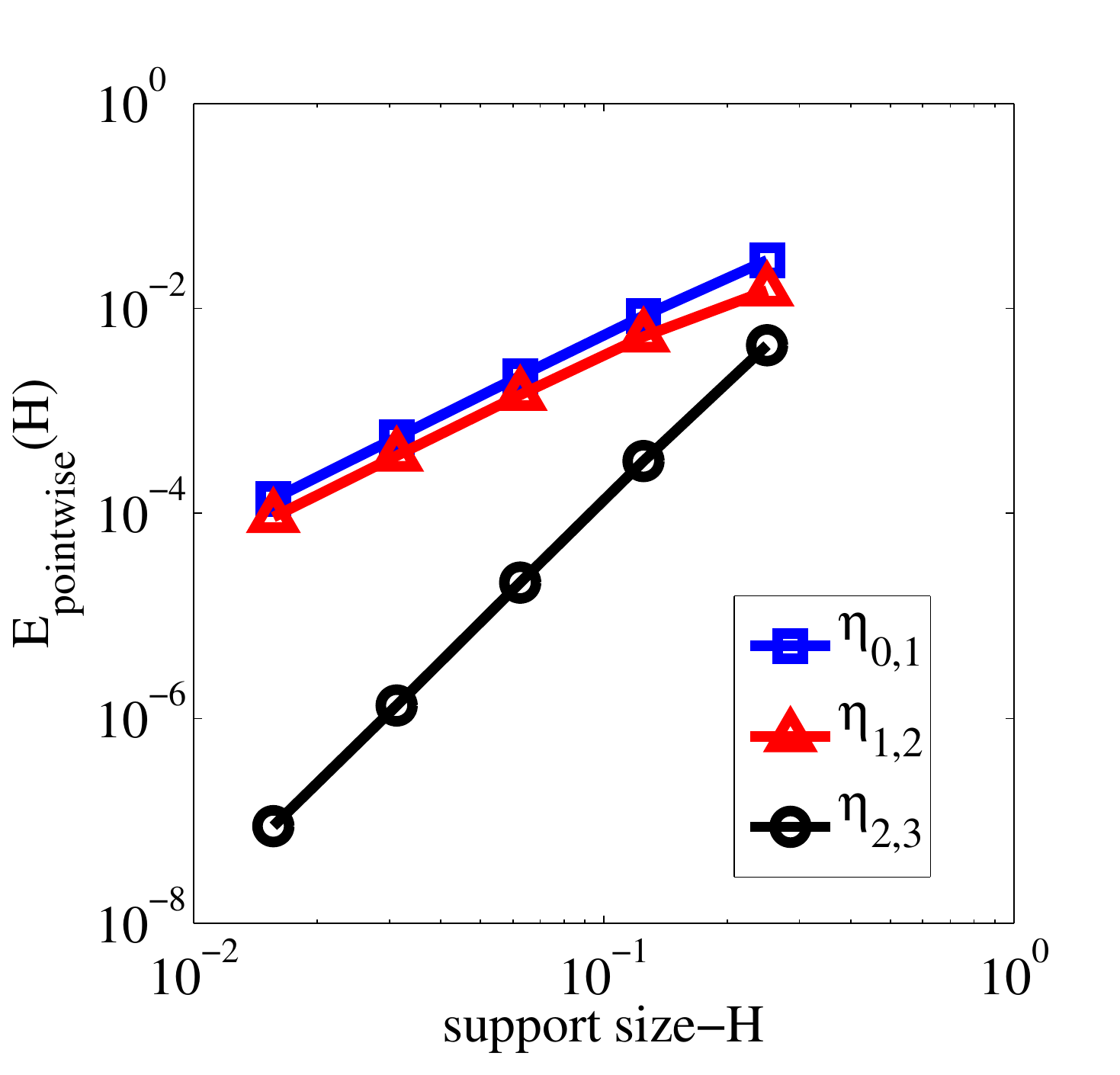}
  \caption{Convergence rates for the 2D Helmholtz solution using the
    error measure \eqref{infterrornorm} and several radially-symmetric
    delta approximations.}
  \label{fig:helm2d}
\end{figure}

\begin{table}[htbp]
  \centering\small
  \begin{tabular}{ | c | c | c | c| c| c|}
    \hline        
    \multicolumn{5}{|c|}{$R_{\text{pointwise}}(H)$} &  \multirow{2}{*}{expected rate}\\ \cline{1-5}
    support size ($H$) &$1/4$ & $1/8$ & $1/16$ & $1/32$ & \\   
    \hline\hline
    $\eta_{0,1}$ & 1.8010  &  1.9498  &  1.9874  &  1.9968 & 2 \\
    \hline
    $\eta_{1,2}$ & 1.4964  &  1.8840  &  1.9715  &  1.9930 & 2 \\
    \hline
    $\eta_{2,3}$ & 3.7560  &  3.9401  &  3.9822  &  3.9127 & 4\\
    \hline
  \end{tabular}
  \caption{Convergence rates for the Helmholtz solution in 2D, compared 
    to the expected analytic rates of convergence in
    \eqref{pointwiseestimate}.}   
  \label{tab:helm2d}
\end{table}

\subsubsection{Helmholtz in $\reals^2$: Convergence in
  $\|\cdot \|_{W_\alpha(B(0,1))}$} 
\label{Numerics:weighted-sobolev}

In this subsection we study the rate of convergence of numerical
solutions to the 2D Helmholtz equation in the weighted Sobolev norms of
Section~\ref{sec:conv-sobolev}. We consider a unit disk as above and
study the same radial regularizations of the delta distribution. The
main difference is that now the error is measured in the
$W_{\alpha}(B(0,1))$ norm using
\begin{equation}
  \label{weightednormfortest}
  E_{W}(H):= \norm{u - u_{H,h}}{W_{\alpha}} := \int_{B(0,1)}
  \absv{\nabla u - \nabla u_{H,h}}^2 \absv{x}^{2\alpha} dx. 
\end{equation}
For our numerical simulations, we consider three different values of
$\alpha = \{ 0.25, 0.5, 0.9 \}$ to investigate the estimate in
\eqref{errorbound} and also take support of size $H=1/2^{n}$ for $n =
\{2, 3, \dots, 8\}$.  Figure \ref{fig:Walphaconv} depicts convergence
plots for various radial approximations (see Table \ref{table:delta-summary}) to the delta distribution and
for different values of $\alpha$.  We also list the error ratios
$R_W(H)$ computing using \eqref{weightednormfortest} and
\eqref{conv-ratio}, and report the corresponding results in
Table~\ref{tab:helm2d-sobolev}.  Note that in the limit as $h \to 0$,
the rates saturate toward the estimate derived in
\eqref{uniformboundedness} as $\beta$ approaches 1. The reason why we
observe this mode of convergence is that $\beta > 1$ corresponds to the
case where the support of the associated $\delta_H$ goes to zero faster than the
resolution of the numerical method. If we assume that this resolution is
of the same order as the mesh size, then $\beta >1$ means that the
support becomes smaller than the mesh size; but this is precisely the
case when our quadrature schemes fail since there will not be sufficient
quadrature points to integrate the function accurately. In our numerical
experiments, we can study the limit of $h \to 0$ while the mesh is small
enough so that the quadrature is still sufficiently accurate.

\begin{figure}[htbp]
  \begin{subfigure}[t]{0.32\textwidth}
    \centering
    \includegraphics[width=1\textwidth]{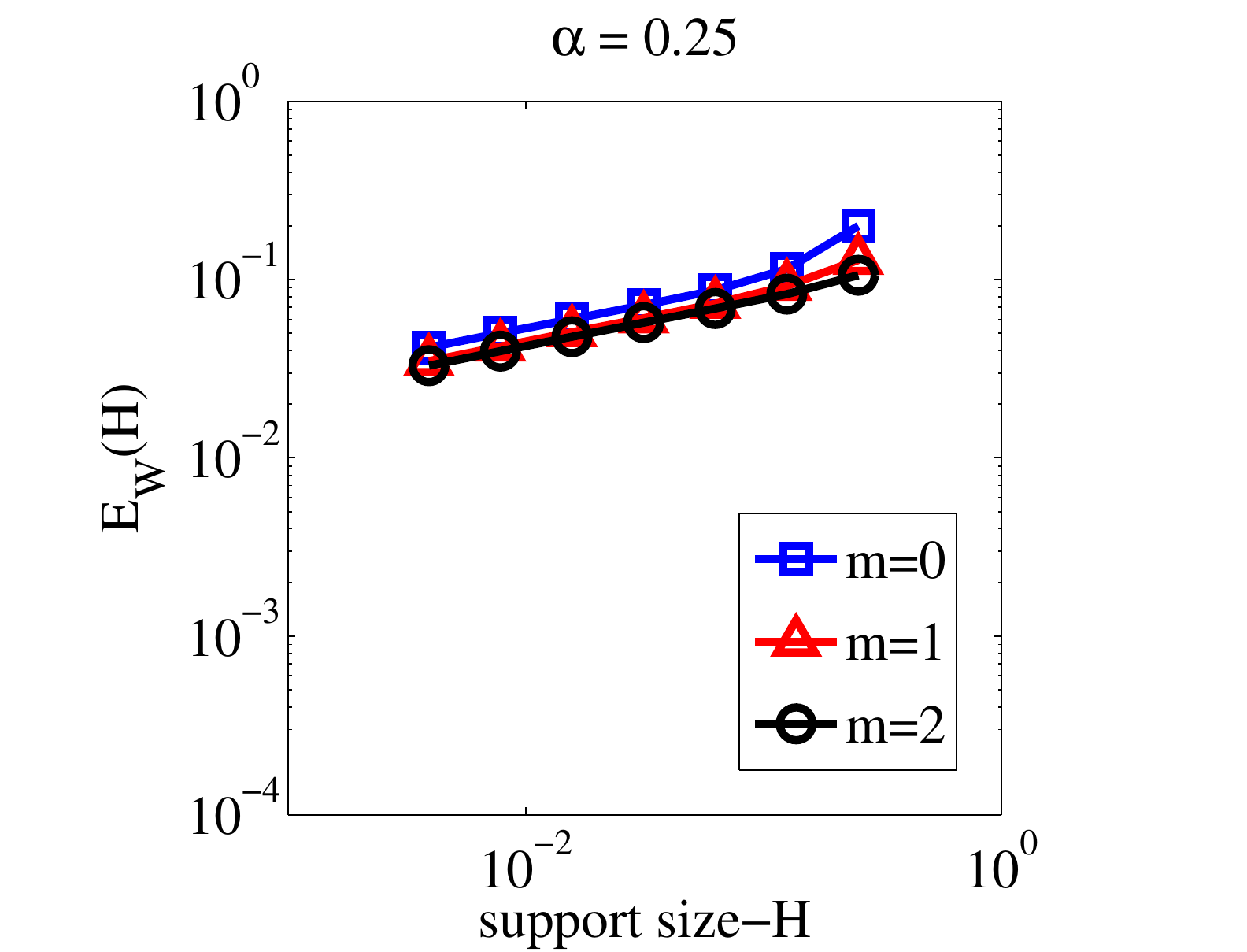}
  \end{subfigure}
  \begin{subfigure}[t]{0.32\textwidth}
    \centering
    \includegraphics[width=1\textwidth]{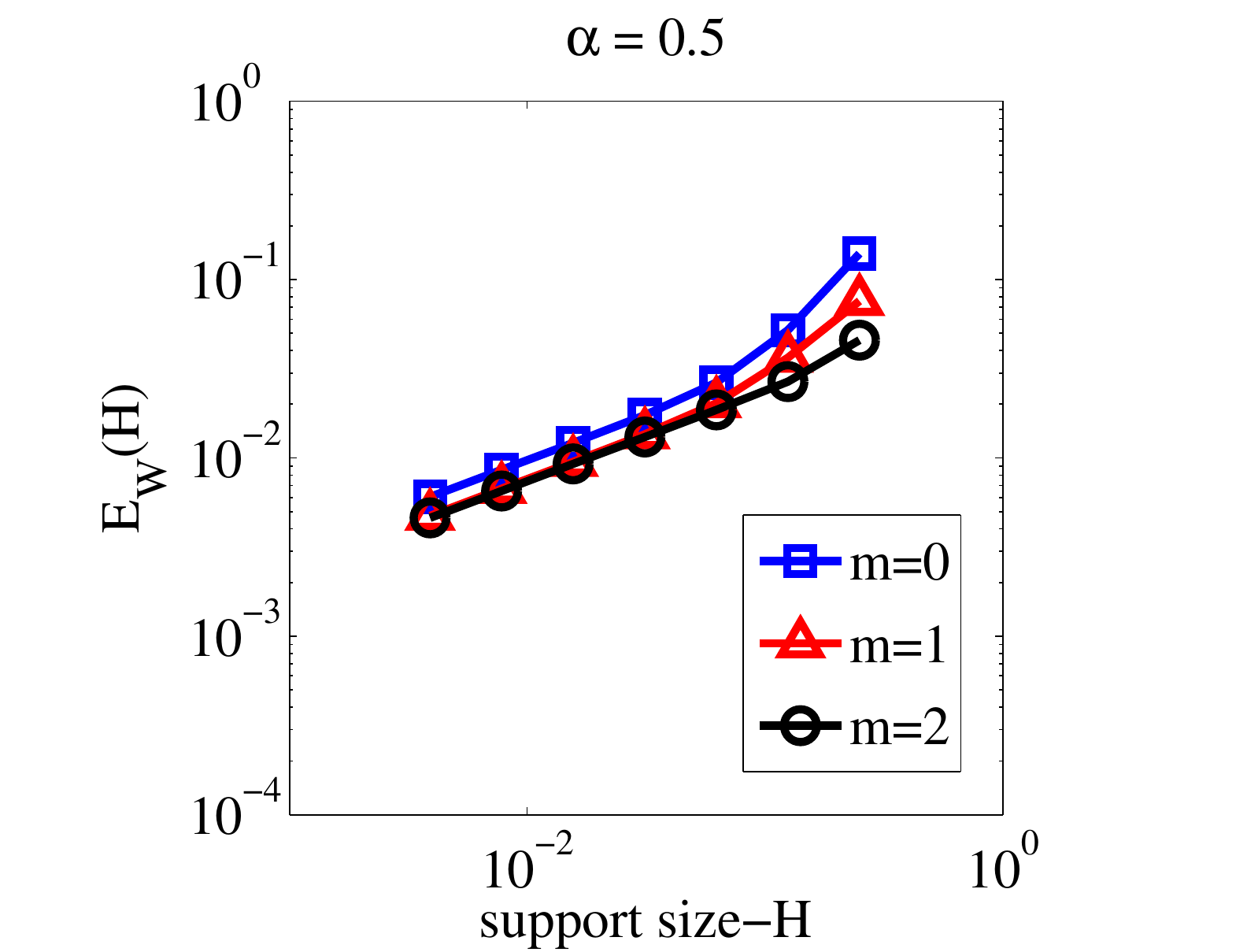}
  \end{subfigure}
  \begin{subfigure}[t]{0.32\textwidth}
    \centering
    \includegraphics[width=1\textwidth]{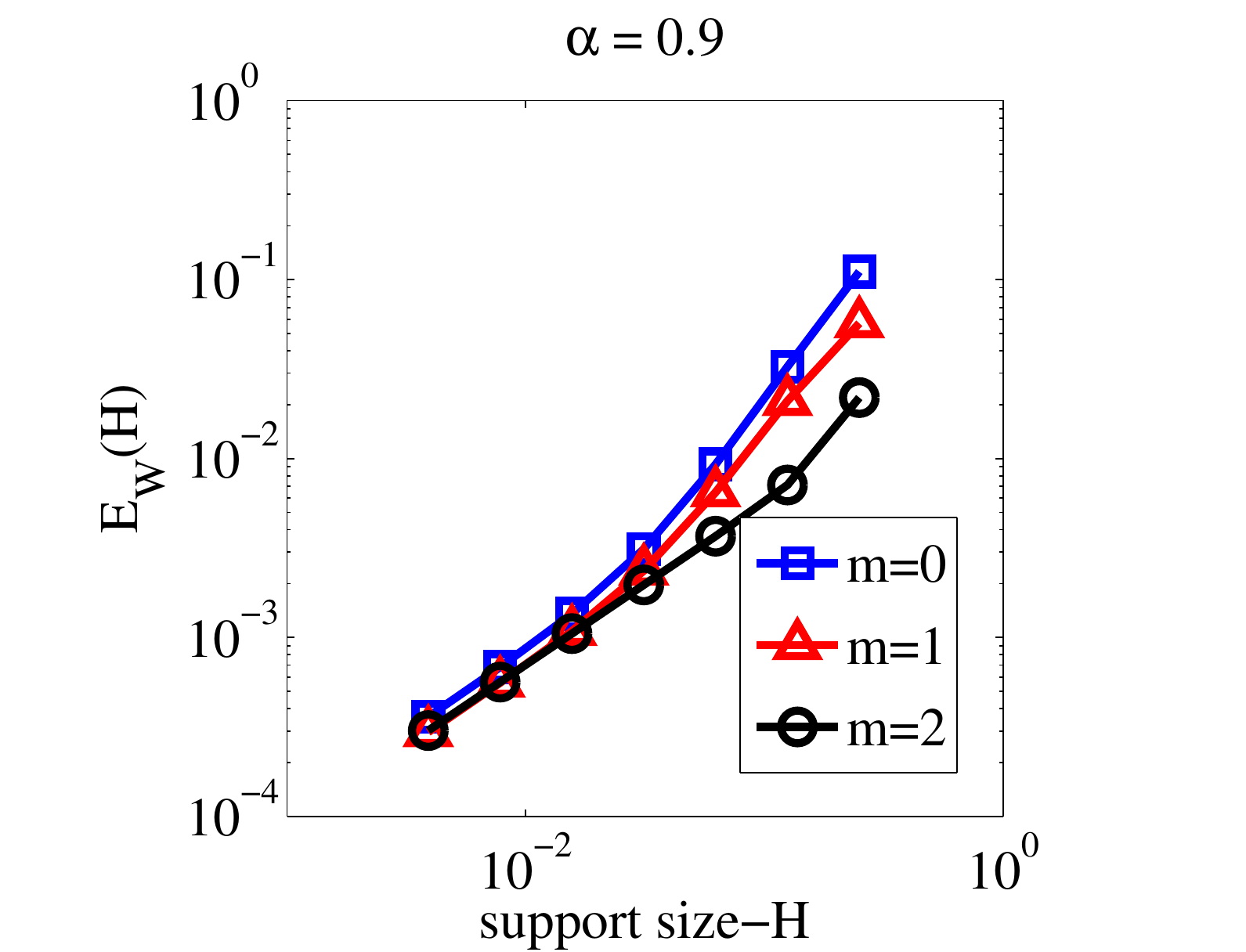}
  \end{subfigure}
  \caption{Error in the 2D Helmholtz solution using the $E_W(H)$ error
    measure from \eqref{weightednormfortest} based on the weighted Sobolev
    norms $W_\alpha$ for  several values of $\alpha$.}
  \label{fig:Walphaconv}
\end{figure}

\renewcommand\arraystretch{1}

\begin{table}[htbp]
  \centering\footnotesize
  \begin{tabularx}{0.99\textwidth}{ |X |X  | X | X | X | X| X |X| c|}
    \hline        
    \multicolumn{8}{|c|}{$R_W(H)$} & expected rate \\ \cline{1-8}
    \multicolumn{2}{|c|}{support size ($H$)} &$1/4$ & $1/8$ & $1/16$ & $1/32$ & $1/64$ & $1/128$ &  for $\beta =1$ \\   
    \hline\hline
    \multirow{3}{*}{$\alpha=0.25$}& $\eta_{0,1}$ &  0.8020  &  0.3946  &  0.2777  &  0.2613  &  0.2615   & 0.2656& \multirow{3}{*}{$0.25$}  \\
    &$\eta_{1,2}$ & 0.4834  &  0.3309  &  0.2695  &  0.2625  &  0.2654  &  0.2719 & \\
    &$\eta_{2,3}$ &   0.3540  &  0.2684  &  0.2611  &  0.2624  &  0.2670  &  0.2746& \\
    % \multicolumn{7}{|c|}{$\alpha=0.5$} \\ \hline
    \hline
    \multirow{3}{*}{$\alpha=0.5$} &$\eta_{0,1}$ &  1.4268  &  0.9757  &  0.6074 &   0.5178  &  0.5035  &  0.5018 & 
    \multirow{3}{*}{$0.5$}\\
    &$\eta_{1,2}$ & 1.0581 &   0.8471  &  0.5775  &  0.5130 &   0.5031  &  0.5025 & \\
    &$\eta_{2,3}$ &   0.7701  &  0.5264  &  0.5061 &   0.5020 &   0.5015   & 0.5024 &\\
    \hline
    \multirow{3}{*}{$\alpha=0.9$}&  $\eta_{0,1}$ &   1.7643   & 1.8157   & 1.5731  &  1.2033   & 0.9878  &  0.9207 & 
    \multirow{3}{*}{$0.9$} \\
    &$\eta_{1,2}$ &   1.4424   & 1.6975   & 1.4557  &  1.1223   & 0.9599  &  0.9138 &\\
    &$ \eta_{2,3}$ &   1.6225 &   0.9483 &   0.9073&    0.9018 &   0.9004&    0.9001 & \\
    \hline  
  \end{tabularx} \noindent
  \caption{Rate of convergence of the 2D Helmholtz solution in the
    weighted Sobolev norm of $W_{\alpha}$ for different 
    number of moments and different values of $\alpha$.  The rates are
    compared to the expected rate of convergence from \eqref{errorbound}
    when $\beta \to 1$. This shows that the rates are independent of the
    number of moments conditions in the limit as support size
    approaches the resolution of the numerical scheme.}
  \label{tab:helm2d-sobolev}
\end{table}

%%%%%%%%%%%%%%%%%%%%%%%%%%%%%%%%%%%%%%%%%%%%%%%%%%%%%%%%%%%%%%%%%%%%%%%%%%%%%%%%%%%%%%%%%%%%%%%%%%%%%%%%%%%%%%%Section 5%%%%%%%%%%%%%%%%%%%%%%%%%%%%%%%%%%%%%%%%%%%%%%%%%%%%%%%%%%%%%%%%%%%%%%%%

\subsection{Hyperbolic problems and the wave equation}

In the previous section we applied regularized point sources to the
solution of elliptic PDEs, and found that the solution exhibits
pointwise convergence as long as we are sufficiently far away from the
source.  Also, the numerical solution converges weakly to the
fundamental solution.  In the following, we perform the analogous
simulations for hyperbolic PDEs and see that these statements do not
necessarily hold.

\subsubsection{First-order wave equation in 1D}
\label{wave-in-1D} 

Consider the first-order wave equation on a periodic domain with an approximate impulse
initial condition:
\begin{equation}
  \label{wave-first-order}
  \left\{
    \begin{aligned}
      &u_{H,t} + u_{H,x} = 0 &&\text{ in } \mathbb{T}(0,2\pi) \times (0,T), \\
      & u_H(x,0) = \delta_H(x) &&\text{ on } \mathbb{T}(0,2\pi) \times \{t=0\}.
    \end{aligned} \right.
\end{equation}
The analytic solution $u_H(x,t)$ to this problem is a simple translation
of the initial condition: $u_H(x,t) = \delta_H(x-t)$. If $u(x,t)$
denotes the solution of \eqref{wave-first-order} with initial condition
$\delta$, then $u(x,t)$ is a delta distribution supported at $(x-t)$. This
means that the difference $|u-u_H|$ is exactly the regularization error
$|\delta-\delta_H|$ and so we can only study convergence of the solution
in the weak-$\ast$ sense.
However, if we solve \eqref{wave-first-order} {\it numerically}, then we
can study the pointwise error $|u_{H,h}(x,T)- u_H(x,T)|$ for large
$t$. Note that $u_H(x,t)=0$ outside the support of $\delta_H(x-t)$, but
as we shall see shortly this is not true for long-time discrete
solutions $u_{H,h}(x,T)$ owing to numerical dispersion.

To compute $u_{H,h}(x,t)$, we use a Fourier spectral collocation method
in space with a leap-frog scheme in time, and implement a code based on
Program 6 of \cite{trefethen}.   In our simulations we use a uniform
spatial grid of 1000 nodes to ensure that the initial condition
$\delta_H(x)$ is captured accurately by the method.  Time increments of
size $\Delta t= \frac{1}{8} \Delta x$ are used to ensure that the CFL
condition is satisfied.

We report the pointwise error at times $t=0$ and $t=36\pi$ (after 13
periods) and since $u_{H,h}(x,0)= \delta_H(x)=u_{H}(x,36\pi)$ on the grid, the
pointwise error becomes
\begin{equation}
  \label{pointwise-error-wave-1D}
  E(x) = \absv{ u_{H} (x,36\pi) - u_{H,h}(x, 36\pi)}
  = \absv{ u_{H,h} (x,0) - u_{H,h}(x, 36\pi)}. 
\end{equation}
\begin{figure}[htbp]
 \centering
 \begin{subfigure}{0.40\textwidth}
   \includegraphics[width=0.9\textwidth]{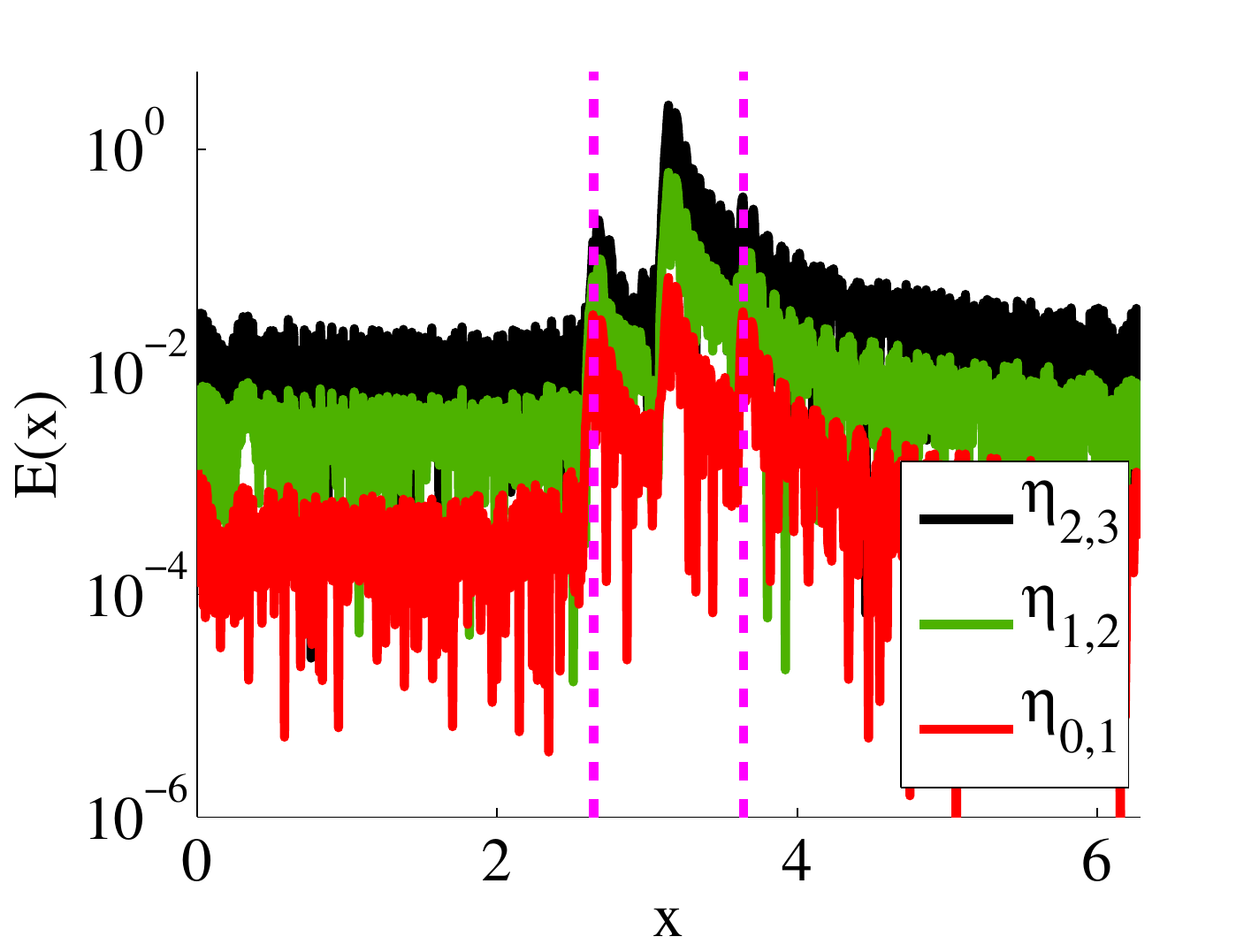}
 \end{subfigure}
 \begin{subfigure}{0.40\textwidth} 
   \includegraphics[width=0.9\textwidth]{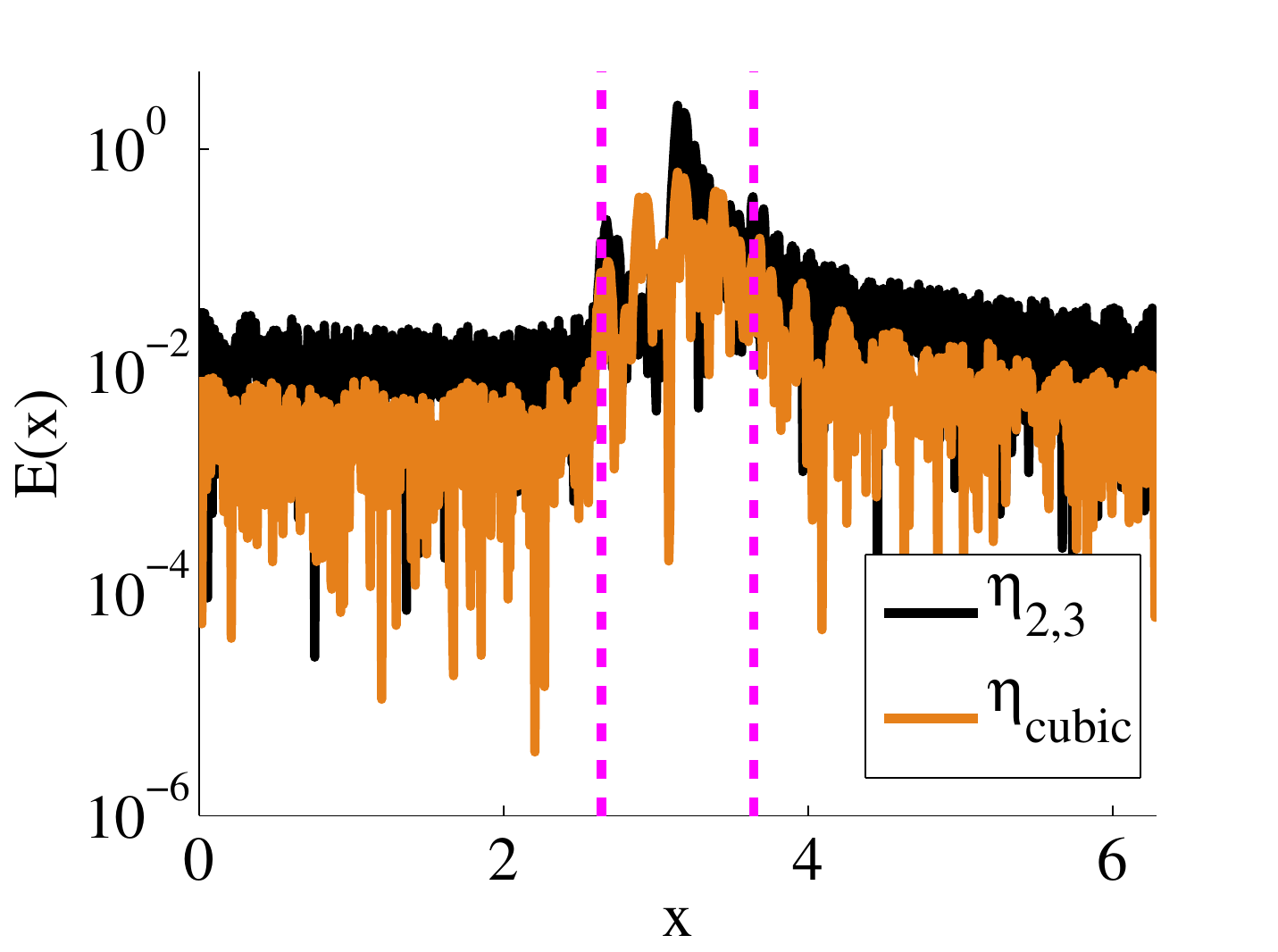}
 \end{subfigure}
 \caption{Pointwise error of solutions to the first order wave equation
   \eqref{wave-first-order} after 13 periods. Dashed lines indicate the
   support of the $\delta_H$ approximations.}
\label{wave-1D-pointwise-error}
\end{figure}
The results are displayed in Figure \ref{wave-1D-pointwise-error}, from
which we see that the error is larger for higher moment
approximations both inside and outside the support of the regularized
distributions. This demonstrates that the the error arises purely from
the discretization of the PDE and not from the regularization, since
$u_H(x,t)=\delta_H(x-t)$ is zero outside the interval $|x-t|\leq H$.
More importantly, $u_{H,h}(x,t)$ does not converge to the true solution
$u(x,t)$ outside the support of $\delta_H$.

Further insight into the cause of this growing numerical error is
afforded by viewing the results from Figure
\ref{wave-1D-pointwise-error} in the Fourier domain. Figures
\ref{wave-1D-fourier-loglog} compare the
discrete Fourier transform of the initial and final solutions for
various approximate delta distributions.  Keeping in mind that the
Fourier coefficients of the Dirac delta distribution are identically
equal to 1, the quality of an approximation $\delta_H$ to the delta
distribution can be measured by looking for Fourier modes that decay as
slowly as possible with increasing wavenumber. However, we note that it
is precisely the higher frequency modes that result in large
(accumulating) discretization errors.
\begin{figure}[htbp]
  \centering
  \begin{subfigure}{0.40\textwidth}
    \includegraphics[width=0.9\textwidth]{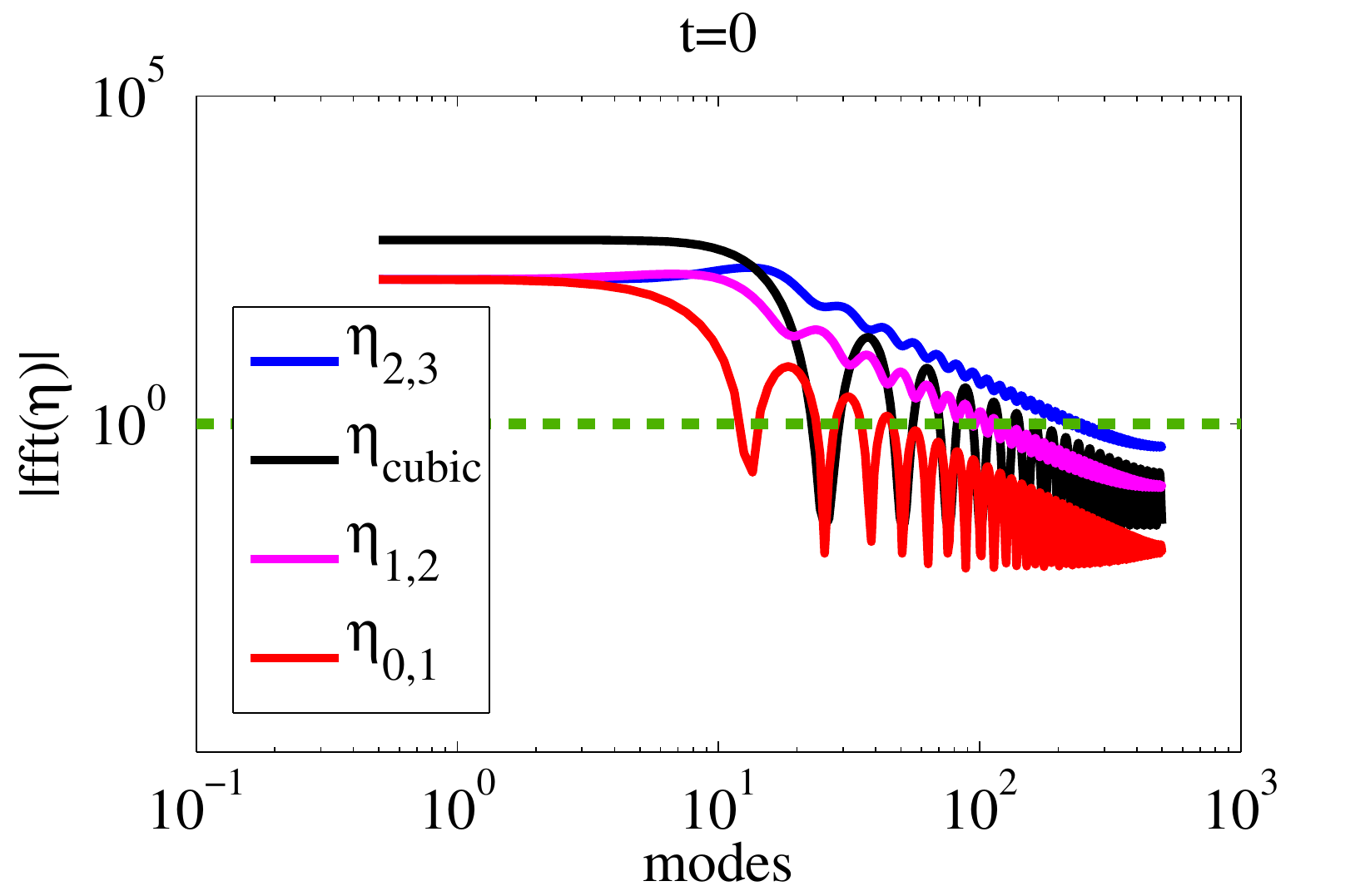}
  \end{subfigure}
  \begin{subfigure}{0.40\textwidth} 
    \includegraphics[width=0.9\textwidth]{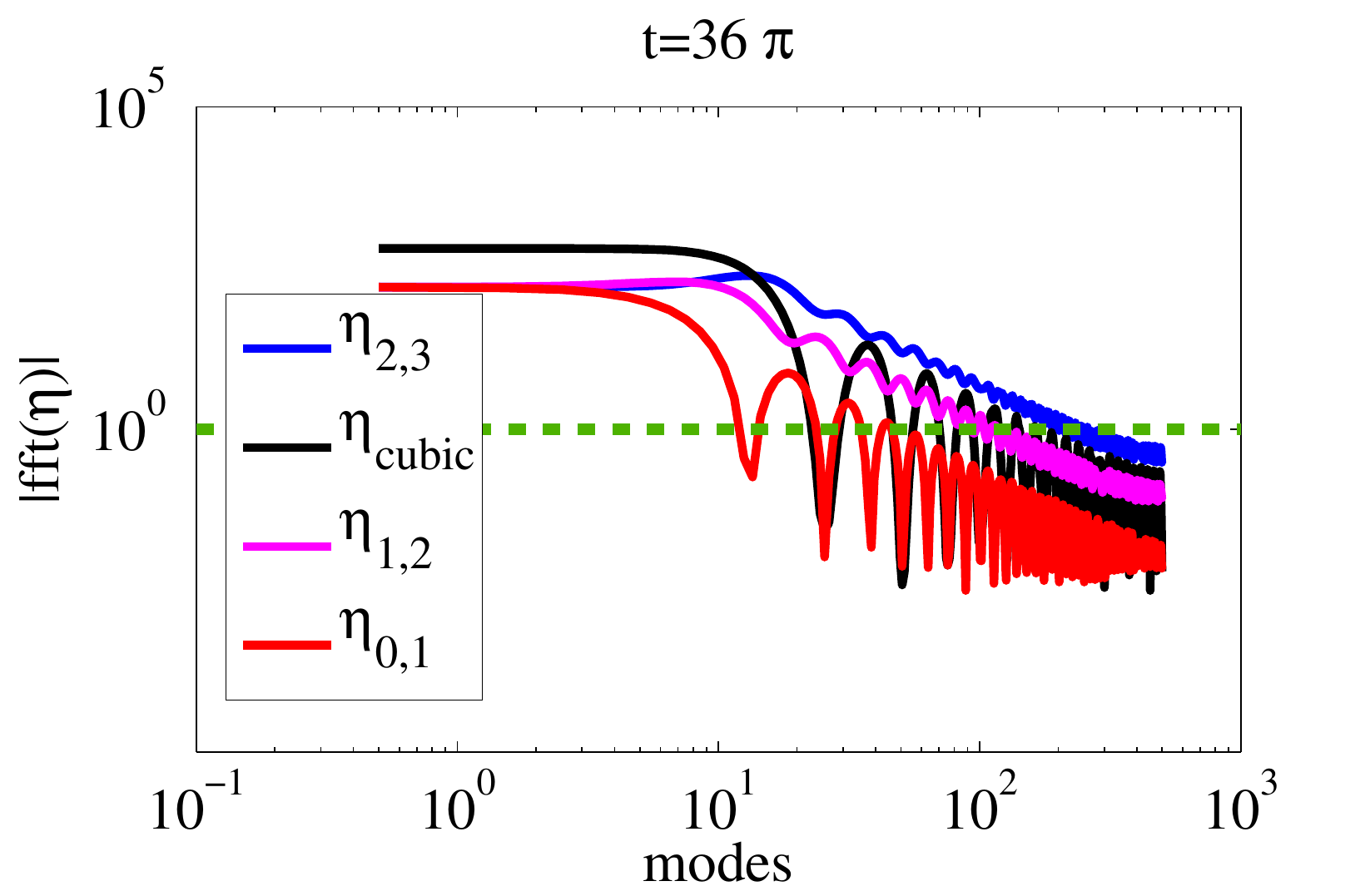}
  \end{subfigure}
  \caption{Comparison of the Fourier transform of the solution to the
    first order wave equation \eqref{wave-first-order} at $t=0$ and
    $t=36\pi$. The green dashed line indicates the exact Fourier
    transform of the delta distribution.}
  \label{wave-1D-fourier-loglog}
\end{figure}

This simple example identifies a number of issues that arise in the
numerical solution of hyperbolic problems with singular delta sources
when approximations to the delta distribution are used.  In essence,
the numerical solutions are dominated by dispersive errors that grow as
the simulation time $T$ increases.

\subsubsection{Second-order wave equation}
\label{Numerics:wave-2D}

Let $u(x,t)$ be the solution to the free-space 1-D wave equation $u_{tt}
= u_{xx}$ with initial condition $u(x,0)=\delta(x)$ and zero initial
velocity.  Suppose that $u_H$ solves the same problem but with initial
condition $u_H(x,0)= \tdelta_H(x)$, where $\tdelta_H$ has support $H$ and
satisfies the compact $m$-moment condition.  For fixed $x$, $u_H=0$ for
$t \le \absv{x}-H$, and $u_H(x,t)$ has a transient behavior for
$\absv{x} - H \le t \le \absv{x} + H$.  But for $t \ge \absv{x} + H$,
$u_H$ will converge to the fundamental solution $u$ with
$\bigO(H^{m+1})$ where $m$ is the number of moment conditions satisfied
by $\tdelta_H$. This becomes clear by noting that for fixed $x$, $B(0,H)
\subseteq B(x,t)$ for $t \ge \absv{x} + H$ and so
\begin{equation*}
  \label{convwaveeven}
  \begin{aligned}
    \int_{ B(x,t)}\delta_H(y)(t^2 - \absv{y - x}^2)^{-1/2} dy =
    (t^2 - \absv{x}^2)^{-1/2} + \bigO(H^{m+1}).  
  \end{aligned}
\end{equation*}
We expect, therefore, that $u_H$ will be close to the free-space
fundamental solution, at least away from the support of $\delta_H$.  In
fact, $u_H(x,t)$ converges pointwise to the free-space fundamental
solution $u(x,t)$ away from the wave front, regardless of the
symmetry of the approximation $\delta_H$. We expect to see the lack of
symmetry only within the wave front and not away from it.

While this result holds for the analytic solution $u_H(x,t)$, when
computing a numerical solution $u_{H,h}(x,t)$ dispersive errors lead to
qualitative differences. To illustrate this effect, we consider the
wave equation on the unit square with homogeneous Dirichlet boundary
conditions over the time interval $t\in [0,0.7]$:
\begin{equation}
  \label{wave-2D-test}
  \begin{cases}
    u_{tt} = u_{xx}, &\text{in } [-1.1]^2 \times (0, 0.7],\\
    u(x,0) = \delta_H,  &u_x(x,0) = u(-1,t) = u(1,t)=0.
  \end{cases}
\end{equation}
For short times, we expect solutions to be the same as for the
free-space wave equation, and in particular we expect them to be
radially symmetric.

We use the radially symmetric
discontinuous approximation $ \delta_H(r)= \frac{1}{H^2}
\eta_{2,2}(r/H)$ and the discontinuous tensor product approximation
$\delta_H(x,y) = \frac{1}{H^2} \eta_{2,2}(x/H) \eta_{2,2}(y/H)$. Both
regularizations satisfy two moment conditions, and we take $H=1/4$. As
we shall see, the discontinuity at the endpoints amplifies the effect of
dispersive errors and so this problem can be viewed as a worst case
scenario.

We use two numerical methods to demonstrate the interplay between
numerical dispersion and lack of symmetry for the tensor-product
$\tdelta_H$.  In the left column of Figure~\ref{fig:wave2D}, we present
the numerical solution using a spectral collocation method with
Chebyshev basis functions and a fourth order Runge-Kutta time-stepping
scheme. This code is available as program~20 of \cite{trefethen}, and we
choose $128^2$ collocation points.  In the right column of
Figure~\ref{fig:wave2D}, we applied a finite element method on a uniform
mesh with $128^2$ elements, using piecewise linear Lagrange
elements for the spatial discretization and Crank-Nicolson
time-stepping.  The code in this case is available as step-23 in the
tutorials for the {\tt deal.II} software package \cite{deal-II}.  We
anticipate that dispersive numerical errors from this finite element
scheme are larger than those in the spectral method.
\begin{figure}[htbp]
\noindent
  \begin{subfigure}[t]{0.40\textwidth}
    \centering
    \includegraphics[width=0.85\textwidth,tics=10]{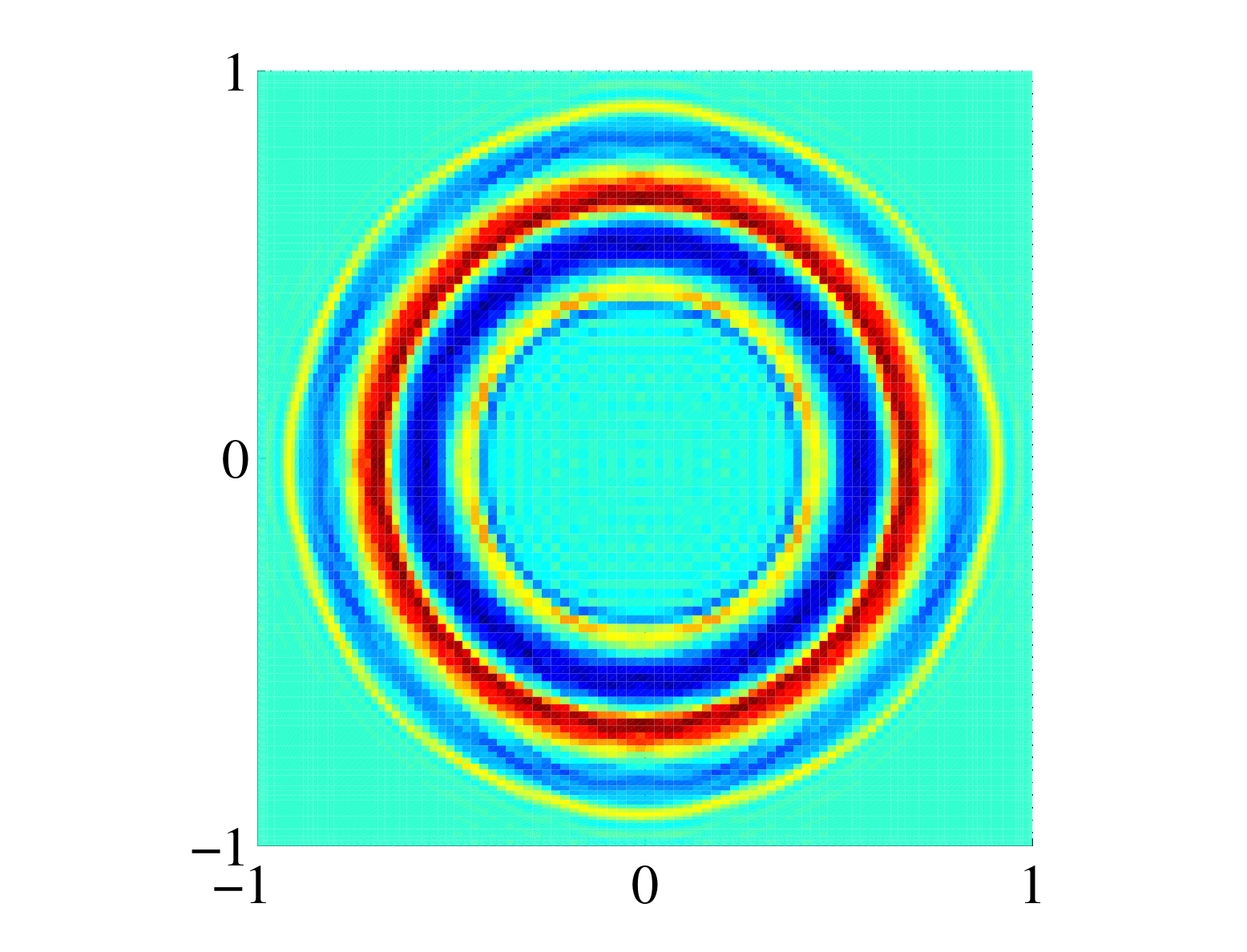}
  \end{subfigure}
  \begin{subfigure}[t]{0.40\textwidth}
    \centering
    \includegraphics[width=0.8\textwidth,tics=10]{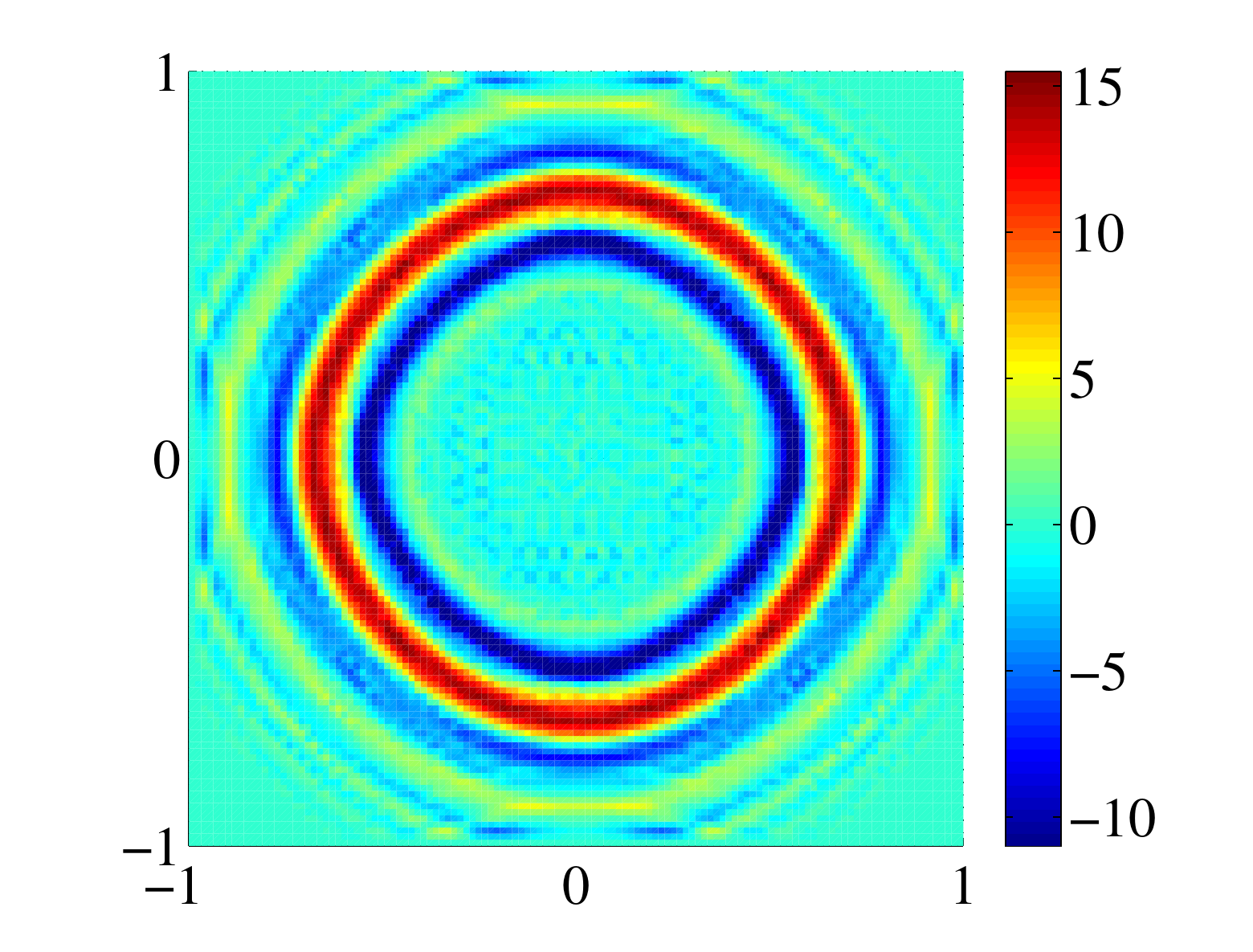}
  \end{subfigure}\\
  \begin{subfigure}[t]{0.40\textwidth}
    \centering
    \includegraphics[width=0.8\textwidth,tics=10]{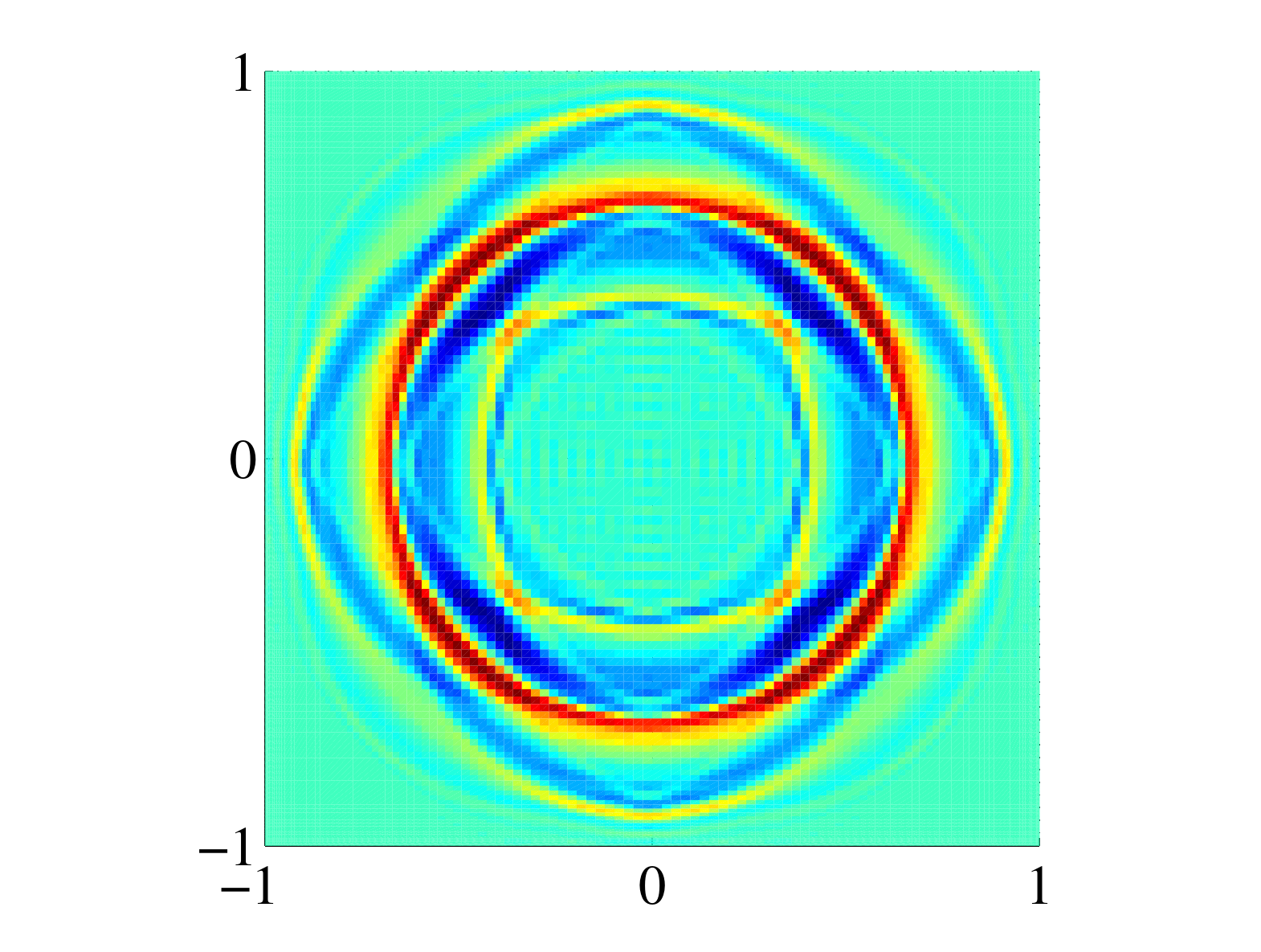}
  \end{subfigure}
  \begin{subfigure}[t]{0.40\textwidth}
    \centering
    \includegraphics[width=0.8\textwidth,tics=10]{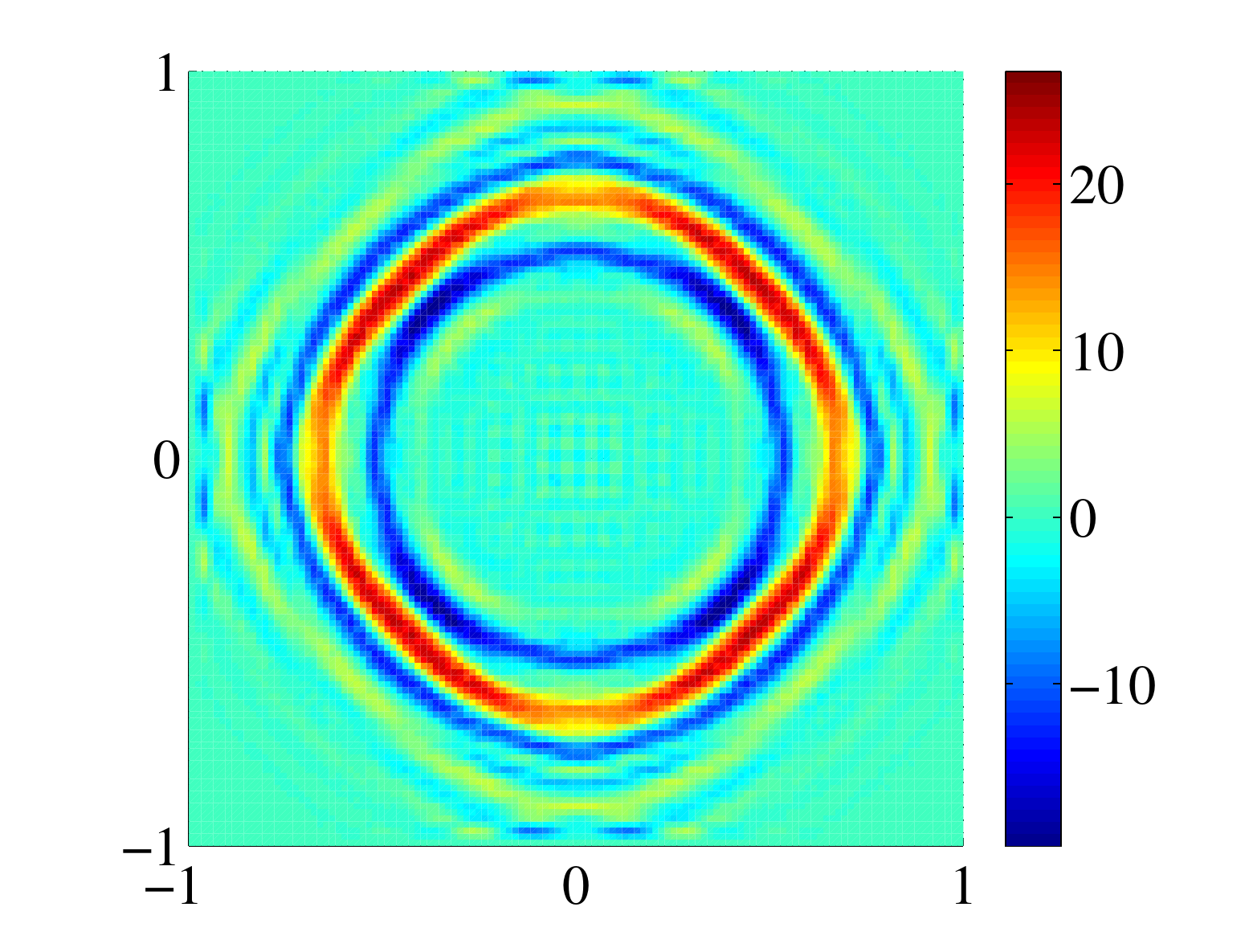}
  \end{subfigure}
  \caption{Numerical solution of the 2D wave equation using radially
    symmetric (top row) and tensor product (bottom row) approximations
    to the delta distribution.  The left column contains the spectral
    collocation results and the right column those with the finite
    element method. Note the reflected waves that appear in the images
    on the right, that indicate waves moving faster than the original
    wave front and which are reflected by the boundary.}
  \label{fig:wave2D}
\end{figure}

For both numerical methods, it is clear that the solution $u_{H,h}(x,t)$
for the tensor-product $\tdelta_H$ does not possess the expected radial
symmetry of the exact solution.  Furthermore, dispersive errors are 
more severe in case of the finite element solver, as expected; in
particular, the finite element solution exhibits 
artificial waves moving faster than the original wave
front. This issue is present for both radial and tensor-product
$\tdelta_H$ approximations, but the fluctuations behind the wave front
(closer to the origin) are significantly larger in case of the tensor
product approximation for both spectral and finite element solvers.

Two points are illustrated here. First of all, a tensor-product
approximation to the delta source term will yield numerical solutions
that do not possess the symmetries of the exact solution, and this
result is independent of the discretization method used. Secondly, the
choice of numerical method is vital in terms of controlling dispersive
errors.

%%%%%%%%%%%%%%%%%%%%%%%%%%%%%%%%%%%%%%%
%%%%%%%%%%%%%%%%%%%%%%%%%%%%%%%%%%%%%%%
%%%%%%%%%%%%%                                     %%%%%%%%%%%%%%
%%%%%%%%%%%%%         Section 6                   %%%%%%%%%%%%%%
%%%%%%%%%%%%%                                     %%%%%%%%%%%%%%
%%%%%%%%%%%%%%%%%%%%%%%%%%%%%%%%%%%%%%%
%%%%%%%%%%%%%%%%%%%%%%%%%%%%%%%%%%%%%%%

\subsection{Korteweg-de Vries (KdV) equation}
\label{Numerics:KDV}

In this final section we present numerical examples that illustrate
issues that can arise when computing approximate solutions of
nonlinear evolution PDEs with a regularized distribution as an initial
condition.  This is closely related to a common class of
nonlinear problems with singular source terms. For example, in the
immersed boundary framework the force from an elastic membrane enters
the fluid momentum equations via a singular line source or chain of
delta distributions \cite{peskin}.  We pick as a nonlinear test
problem the KdV equation on the periodic domain $[-8\pi, 8\pi]$:
\begin{equation}
  \label{KdV}
  %\left\{
  %  \begin{aligned}
      u_t + 6uu_x + u_{xxx} = 0 \quad \text{ in } \;\; \mathbb{T}([-8\pi,
      8\pi]) \times (0, 0.05] \quad \mbox{and}\;\;  u(x,0) = \tdelta_H.
\end{equation}
We have chosen a relatively large spatial domain and a short time interval in
order to minimize boundary effects; sufficiently accurate numerical
solutions of \eqref{KdV} in this setting will be good approximations of
the free-space solution. 

The free-space KdV equation with delta initial condition 
has no closed form solution.  This problem has nonetheless been
studied extensively and is known to consist of a single soliton moving
to the right along with radiative waves propagating to the left (see
\cite{drazin}), with the soliton portion of the solution given by
\begin{equation}
  \label{KDV-single-soliton}
  u(x,t) \approx \frac{1}{2} \text{sech}^2\left( \frac{1}{2} (x - t)
  \right). 
\end{equation}
In numerical solutions of \eqref{KdV}, we therefore expect to see a
similar soliton combined with radiative waves, at least for short times.

To solve this problem numerically, we use the Fourier spectral solver
implemented in Program 27 of \cite{trefethen} with $N=512$ Fourier
modes, and a fourth-order Runge-Kutta time-stepping method.  We observe
that the solution is very sensitive to the smoothness of the $\delta_H$
approximation. Furthermore, controlling dispersive error is
challenging, as is usually the case whenever simulating nonlinear wave
equations.

As a first test, we aim to characterize the impact of support size for
$\delta_H$ on the computed solutions. We perform our initial simulation
using the delta approximation $\eta_{2,5}(r)$ from
Section~\ref{Section:etadefinitions}, which is a 5th order polynomial
satisfying $m=2$ moment conditions and is
differentiable everywhere (see Table~\ref{table:delta-summary}). We
construct $\delta_H$ from $\eta_{2,5}(r)$ as 
\begin{gather*}
  \delta_H(x)=\frac{1}{H}\eta_{2,5}\left(\frac{|x|}{H} \right) \quad
  \text{where } x\in [-8\pi,8\pi],
\end{gather*}
and solve the KdV equation using our spectral scheme for $H = \pi,
\pi/2, \pi/4$. At least in the limit as $H\rightarrow 0$, we expect
$u_{H,h}$ to exhibit the expected soliton and radiative waves. Figure
\ref{fig:KDV-1} depicts the numerical solutions $u_{H,h}(x,t)$ for
several choices of $H$, with the left column showing solution in the
$(x,t)$ plane, while the right column shows the time evolution of the
Fourier modes. Note that the soliton portion of the solution
(the light green strip in the middle of the plots) is captured well for
all $H$.  However, the lower order radiative waves differ considerably
with $H$.
\begin{figure}[htbp]
  \noindent
  \centering
  \begin{subfigure}[t]{0.45\textwidth}
    \centering
    \includegraphics[width=0.9\textwidth]{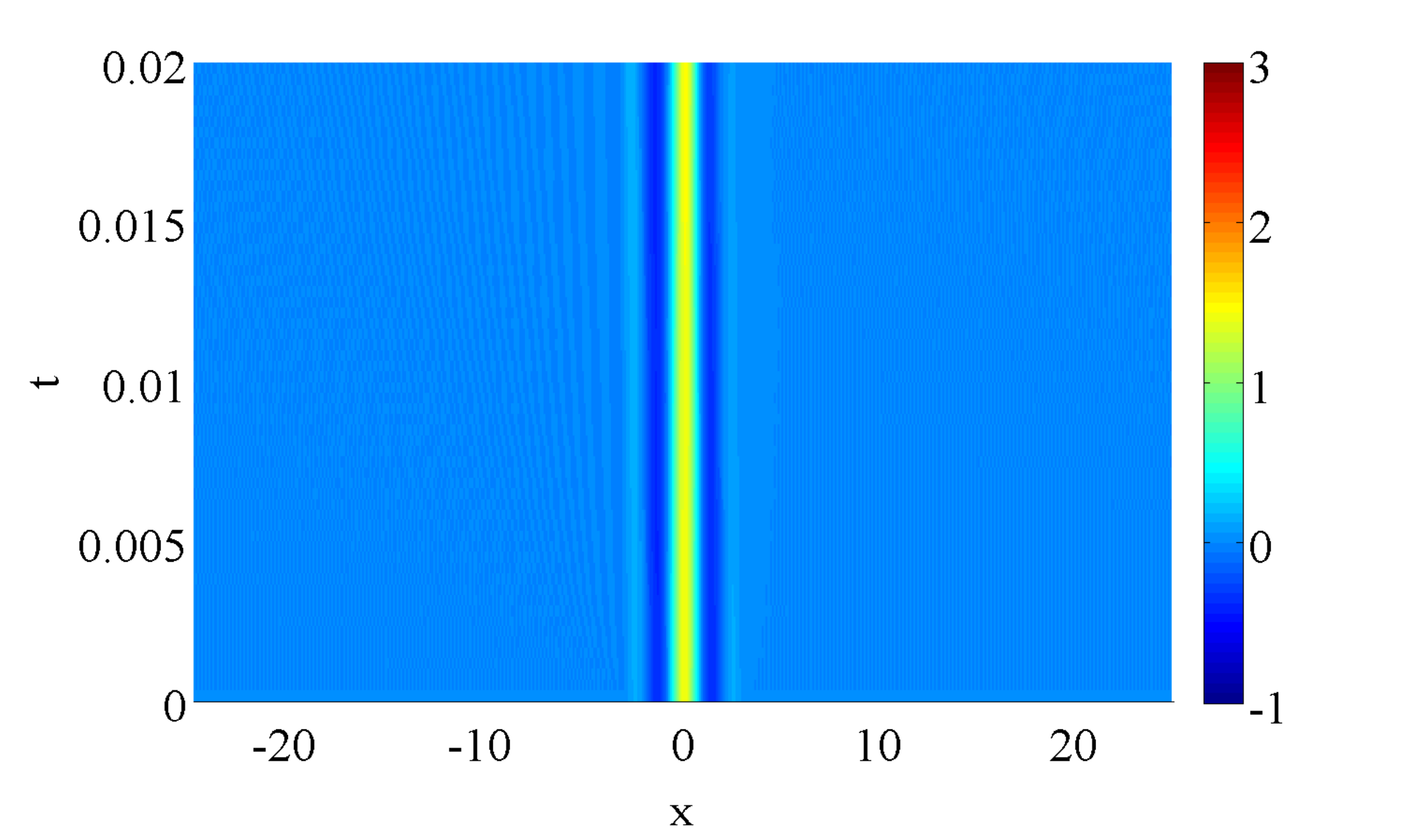}
  \end{subfigure}
  \begin{subfigure}[t]{0.45\textwidth}
    \centering
    \includegraphics[width=0.9\textwidth]{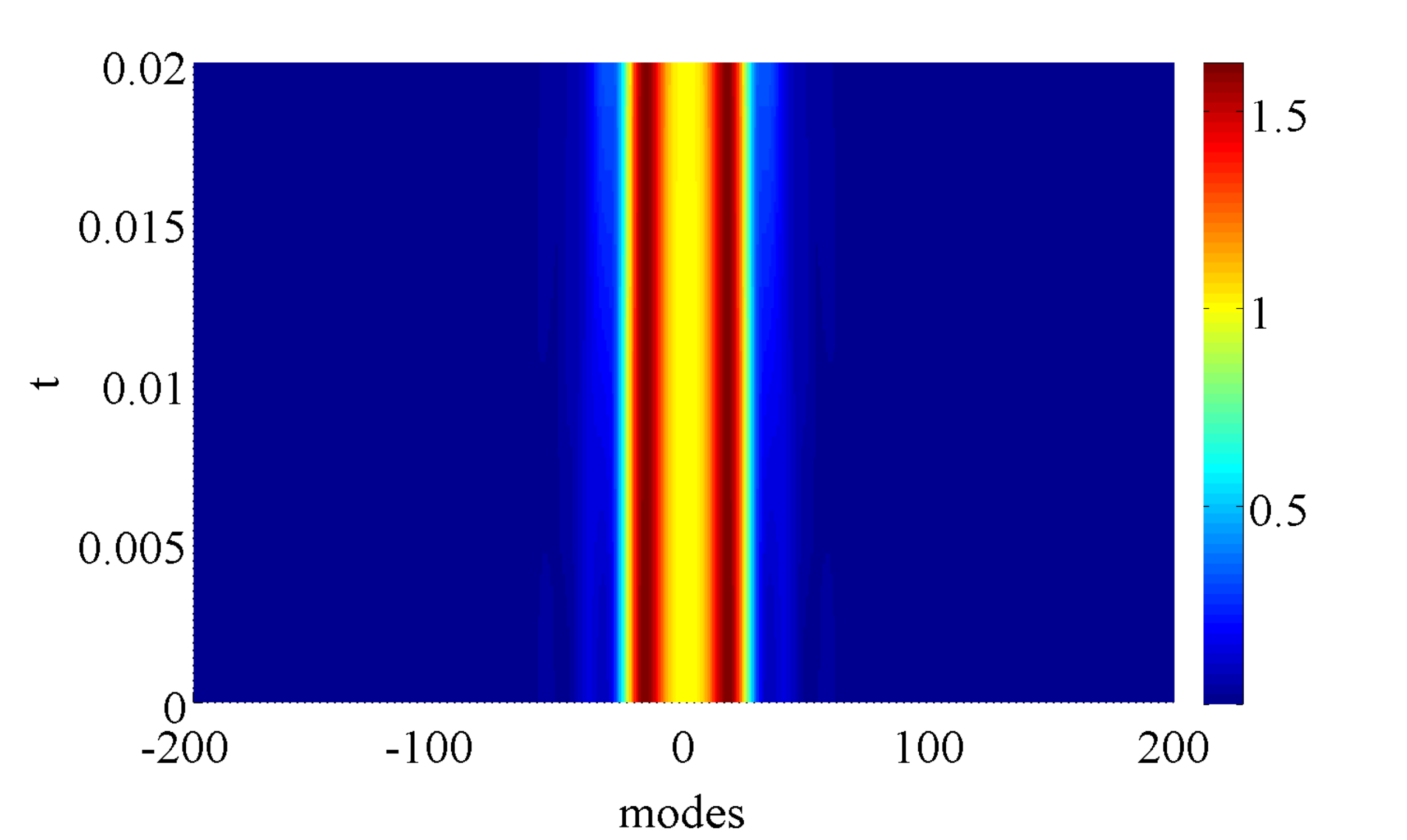}
  \end{subfigure} \\
  \begin{subfigure}[t]{0.45\textwidth}
    \centering
    \includegraphics[width=0.9\textwidth]{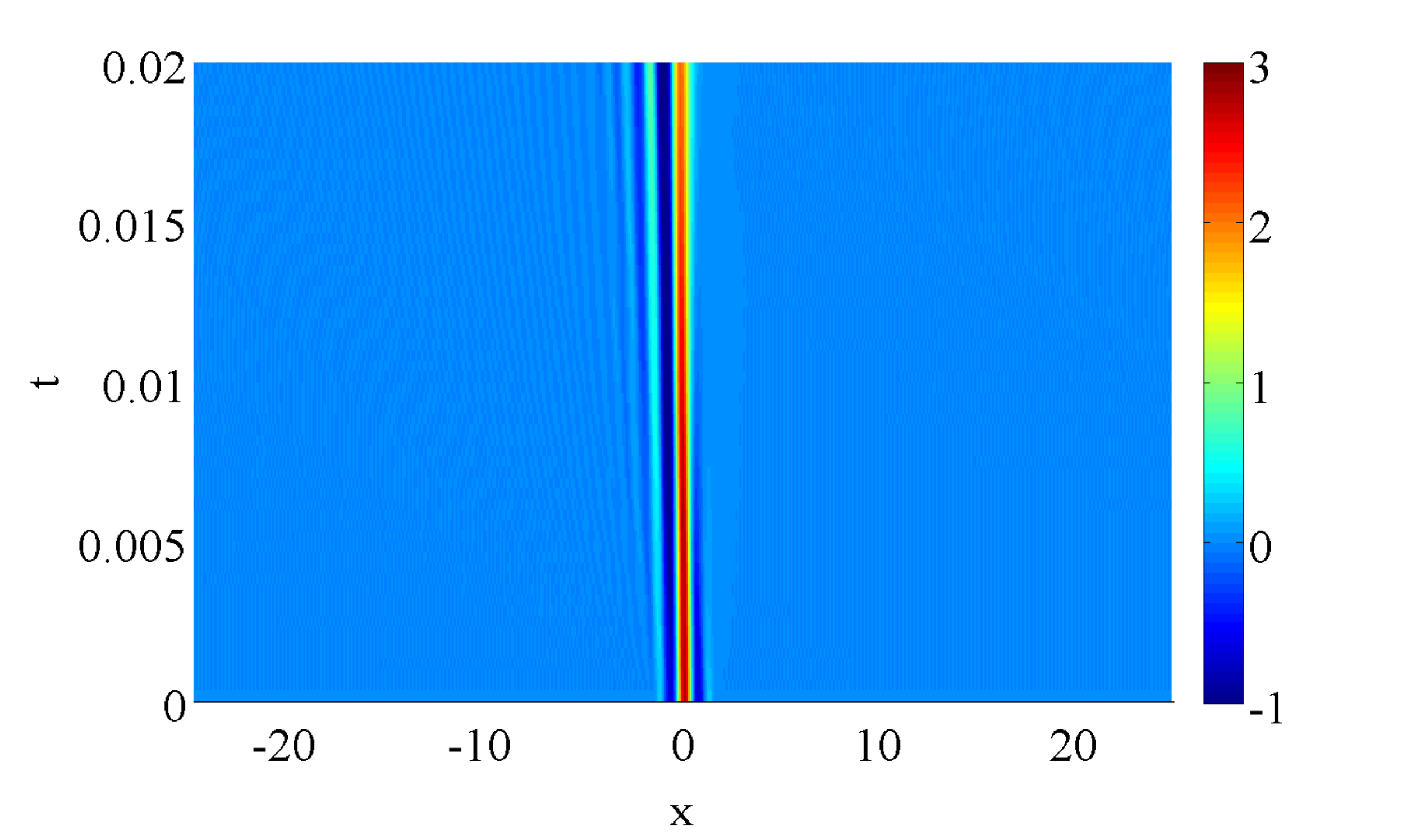}
  \end{subfigure}
  \begin{subfigure}[t]{0.45\textwidth}
    \centering
    \includegraphics[width=0.9\textwidth]{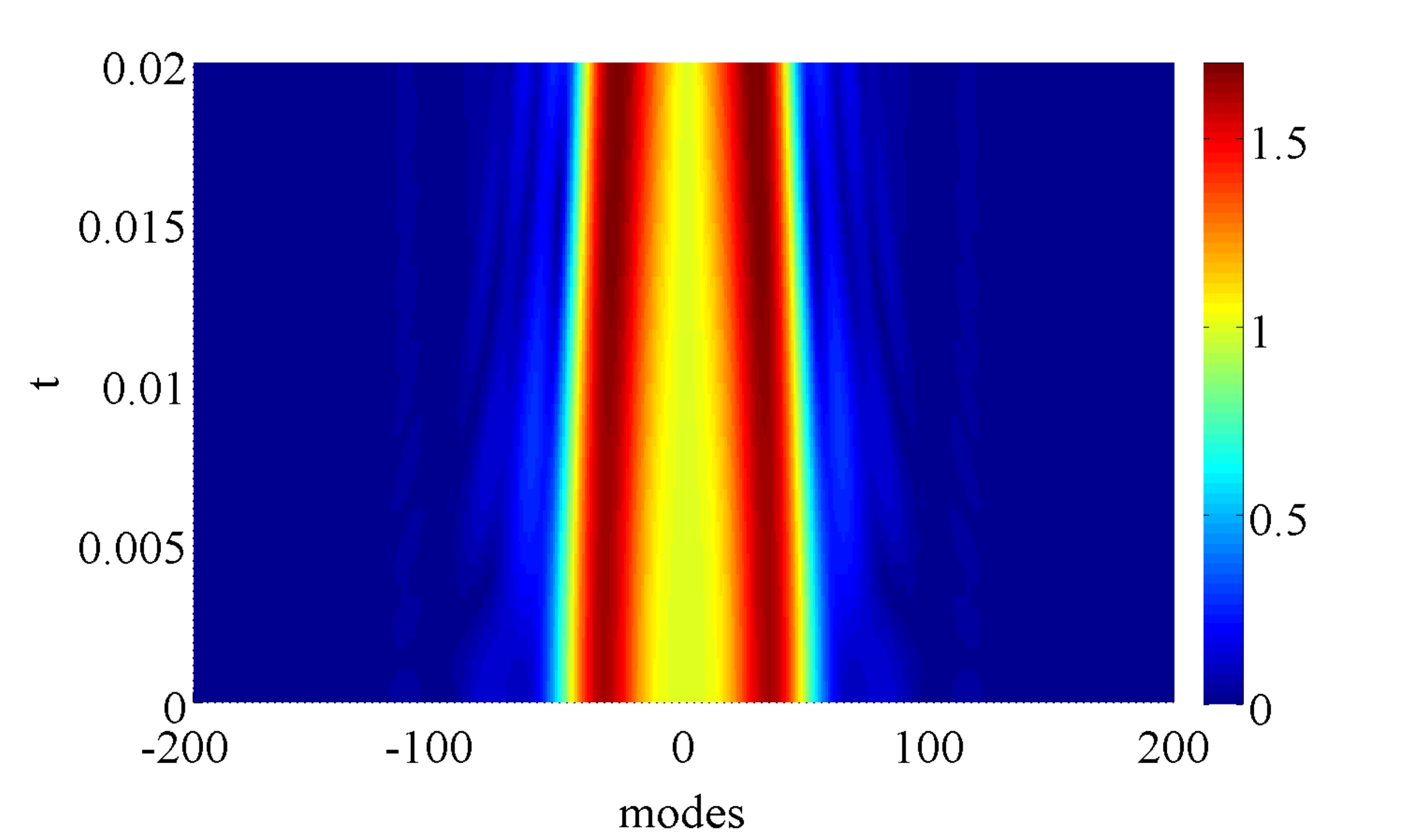}
  \end{subfigure}\\
  \begin{subfigure}[t]{0.45\textwidth}
    \centering
    \includegraphics[width=0.9\textwidth]{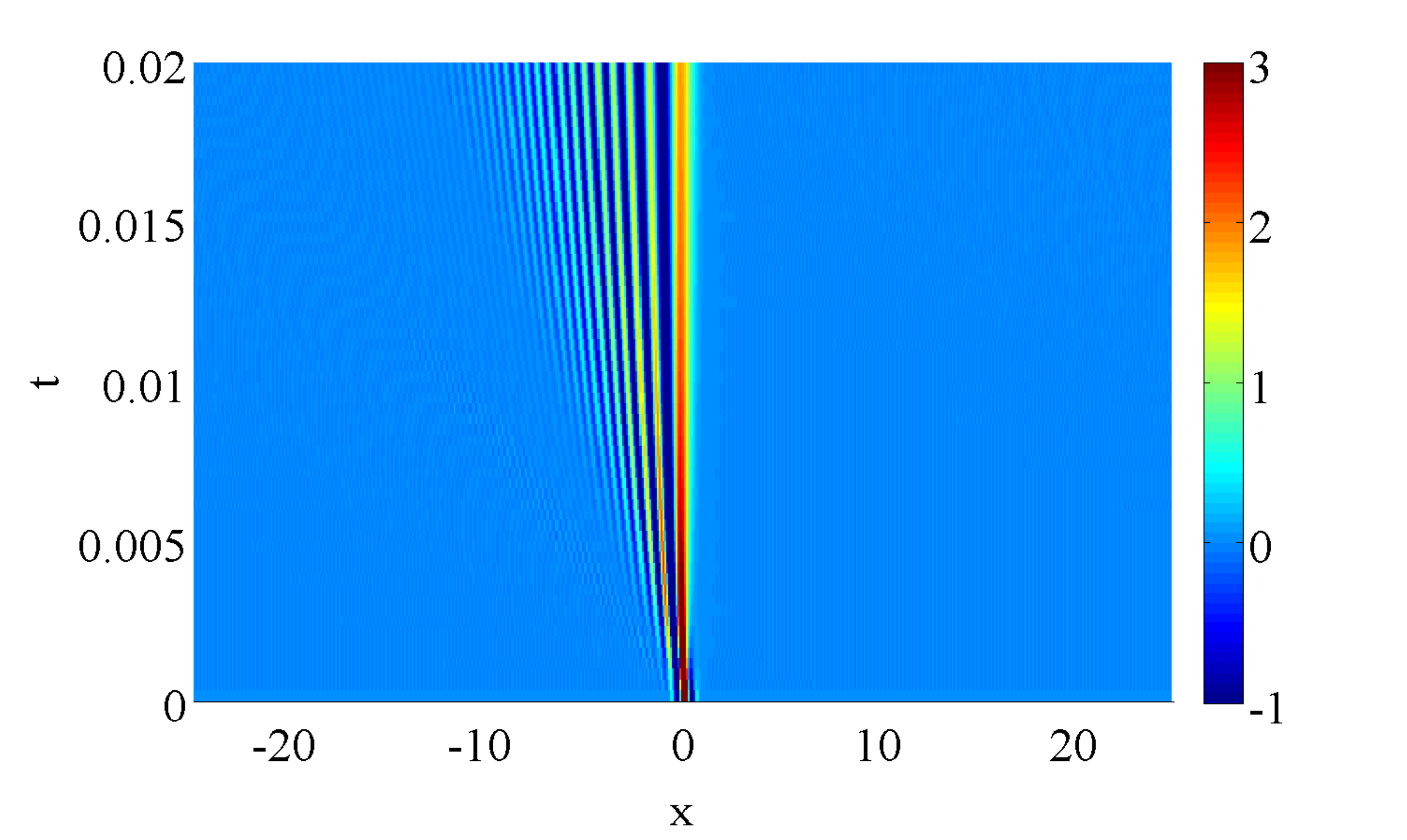}
  \end{subfigure}
  \begin{subfigure}[t]{0.45\textwidth}
    \centering
    \includegraphics[width=0.9\textwidth]{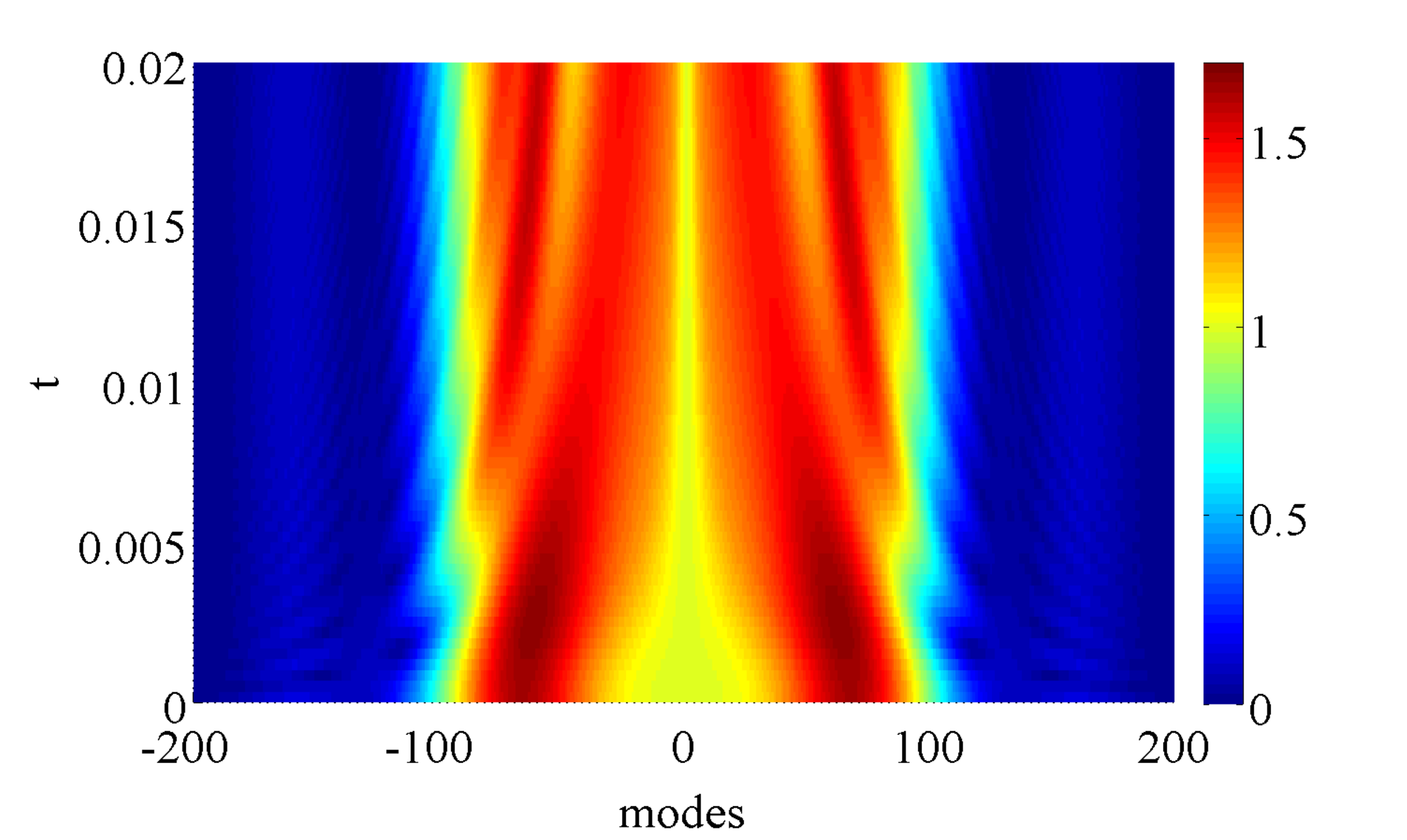}
  \end{subfigure}
  \caption{Solutions of the KdV equation with impulse initial condition
    using a smooth second moment approximation, $\eta_{2,5}$. Rows from
    top to bottom depict solutions for support of size $H=\pi$,
    $\frac{\pi}{2}$ and $\frac{\pi}{4}$. Images in the right column show
    the Fourier transform of the solution.}
  \label{fig:KDV-1}
\end{figure}

Next, we want to demonstrate the effect of different choices of moments
on the computed solution. Since $\tdelta_H$ approximates $\delta$ with a
rate depending on $m$, one may expect that the error $\|u-u_H\|_X$
should be improved with higher $m$. However, for the nonlinear wave
equations, 
the error $\|u-u_{H,h}\|_X$
critically depends on the choice of the PDE discretization method; that
is, the error is dominated by $\|u_H-u_{H,h}\|_X$. To demonstrate this, we
study solutions of \eqref{KdV} with $\delta_H$ based on $\eta_{0,1}$,
$\eta_{cubic}$ and $\eta_{2,5}$, and fix $H=\pi/4$ in each
case. Recall that $\eta_{\text{cubic}}$ is a widely-used regularization
that satisfies two discrete moment conditions \cite{tornberg}. The
$\delta_H$ based on $\eta_{2,5}$ satisfies two continuous moment
conditions, and is smooth. To provide a point of reference we also
compute the solution to equation \eqref{KdV} with a Gaussian source term
of the form
\begin{equation}
  \label{eq:1}
  \delta_\sigma(x) = \frac{1}{\sqrt{2 \pi \sigma^2}} \exp \left( \frac{- (x -
  1/2)^2}{\sigma^2} \right)
\end{equation}
where we take $\sigma = 1/64 \pi$ and solve the problem on a very fine
mesh of 8196 collocation points. The Gaussian source term is
infinitely 
differentiable and so has optimal decay of Fourier modes.

\begin{figure}[htbp] 
  \centering
  \begin{subfigure}[t]{0.45\textwidth}
    \centering
    \includegraphics[width=0.9\textwidth]{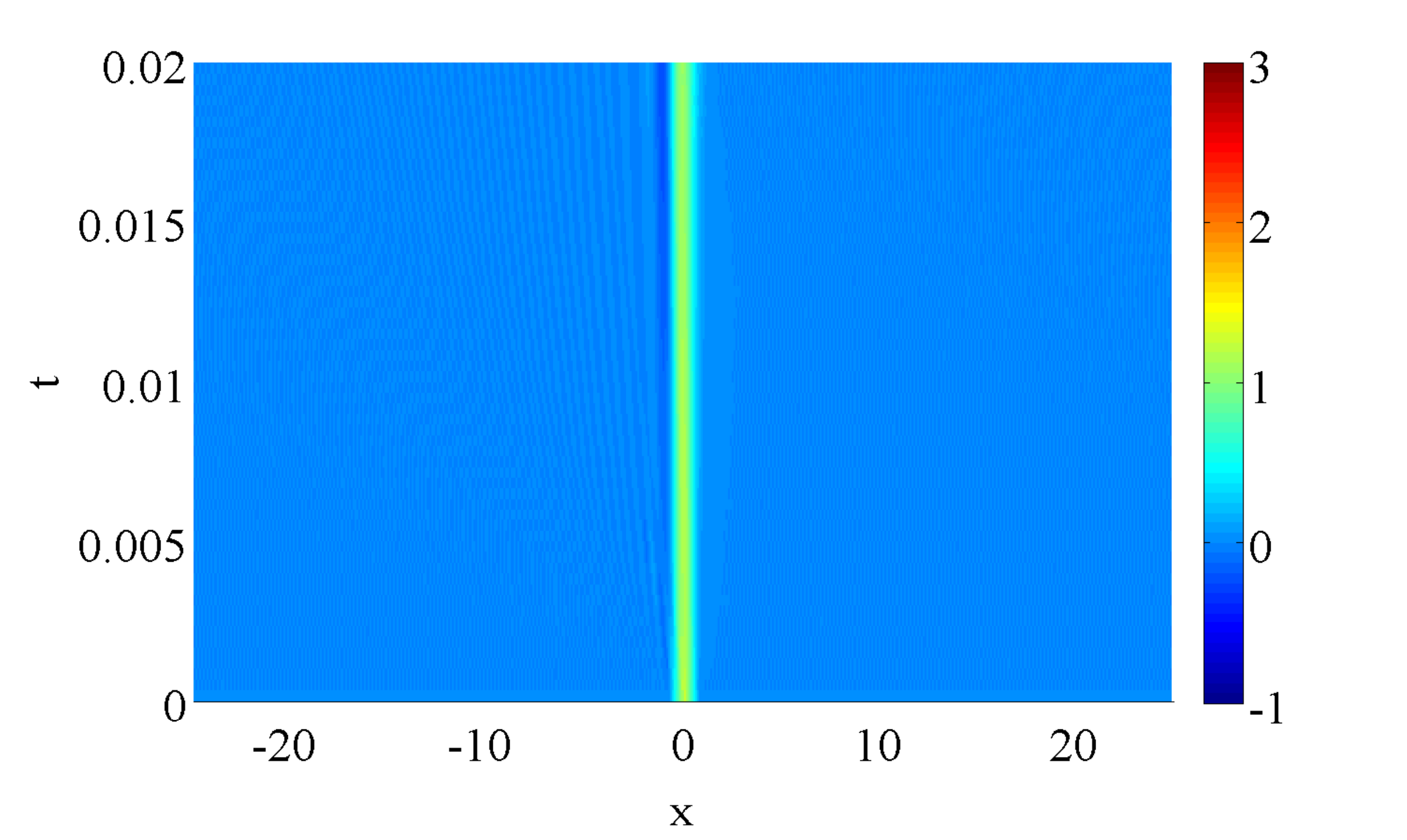}
  \end{subfigure}
  \begin{subfigure}[t]{0.45\textwidth}
    \centering
    \includegraphics[width=0.9\textwidth]{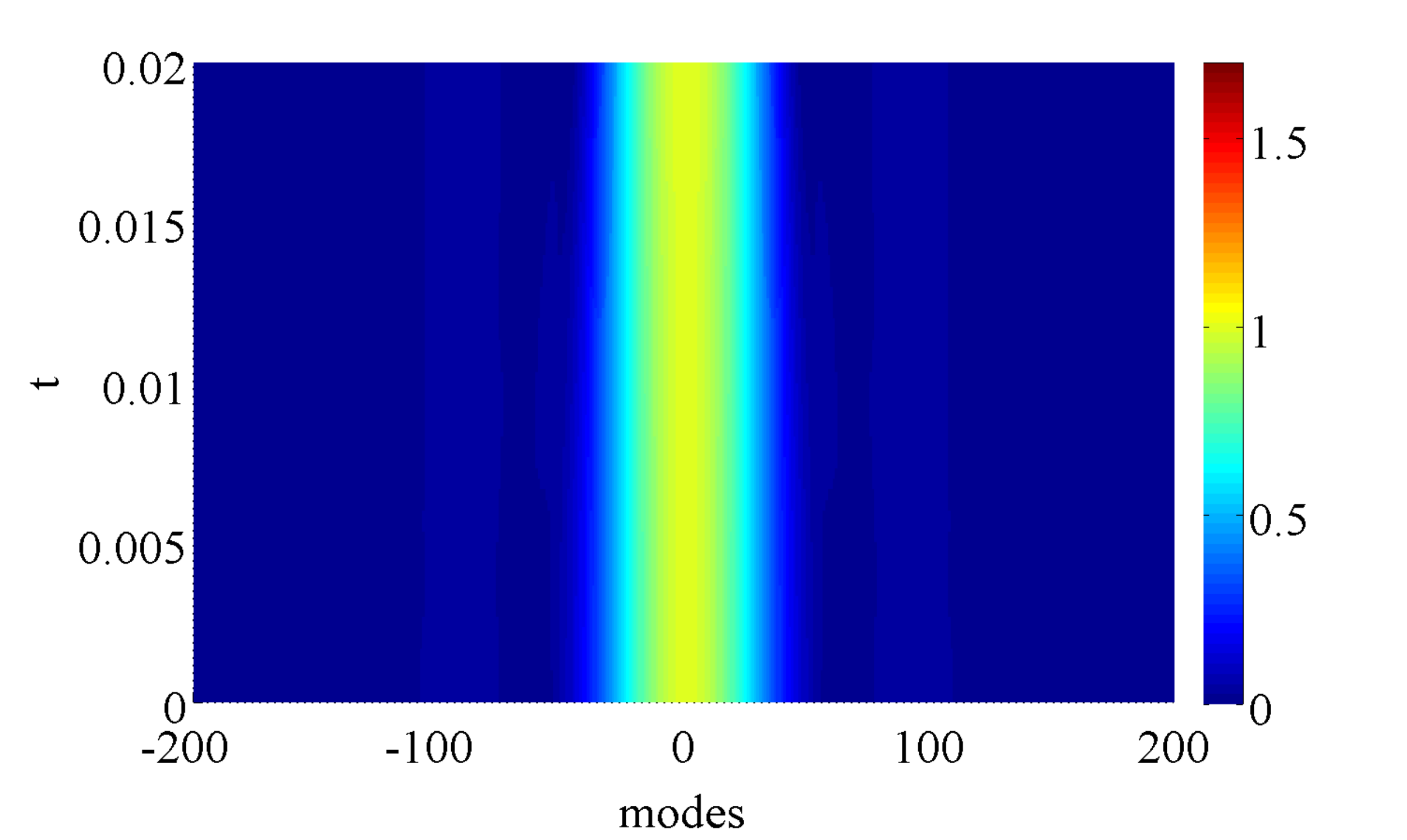}
  \end{subfigure} \\
  \begin{subfigure}[t]{0.45\textwidth}
    \centering
    \includegraphics[width=0.9\textwidth]{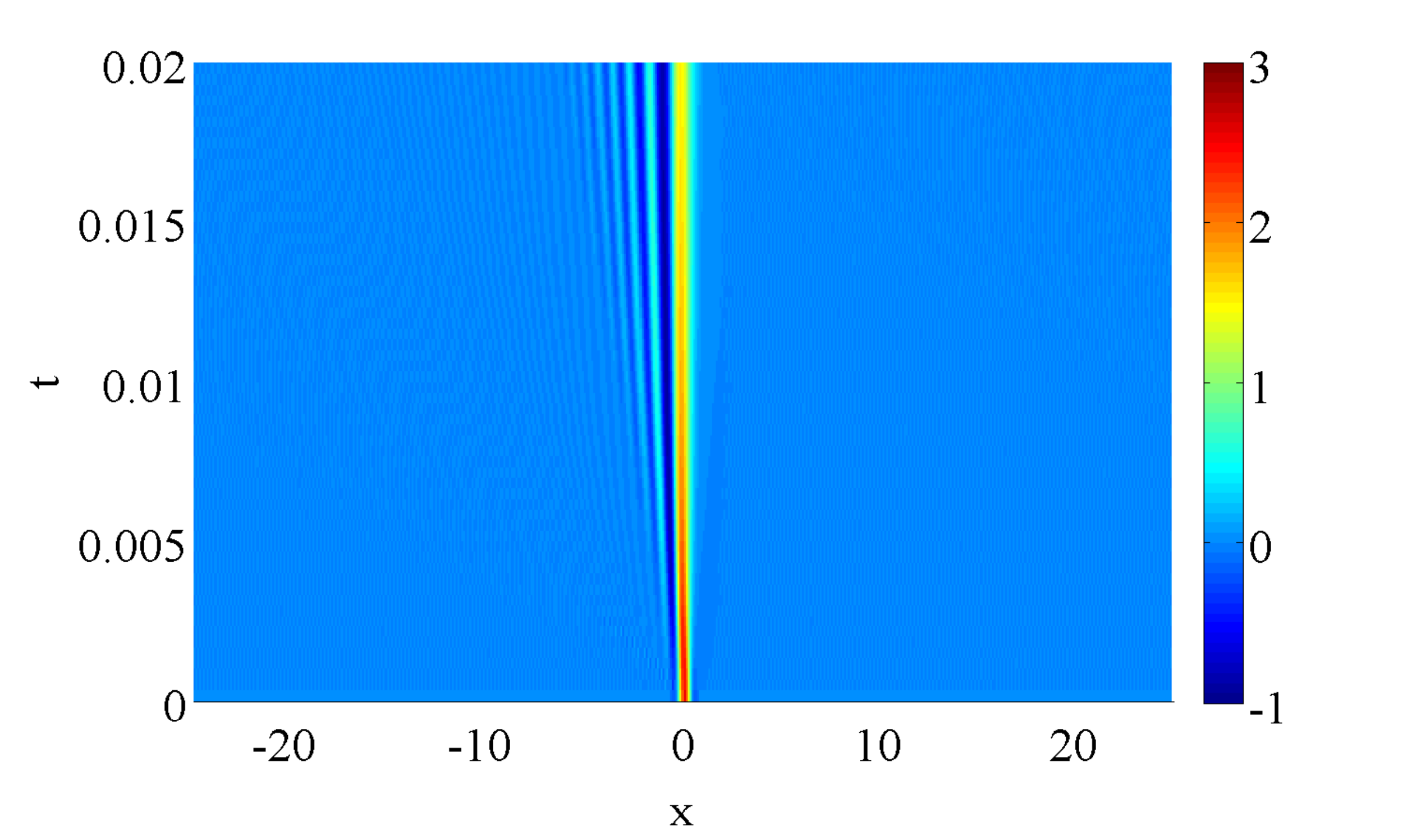}
  \end{subfigure}
  \begin{subfigure}[t]{0.45\textwidth}
    \centering
    \includegraphics[width=0.9\textwidth]{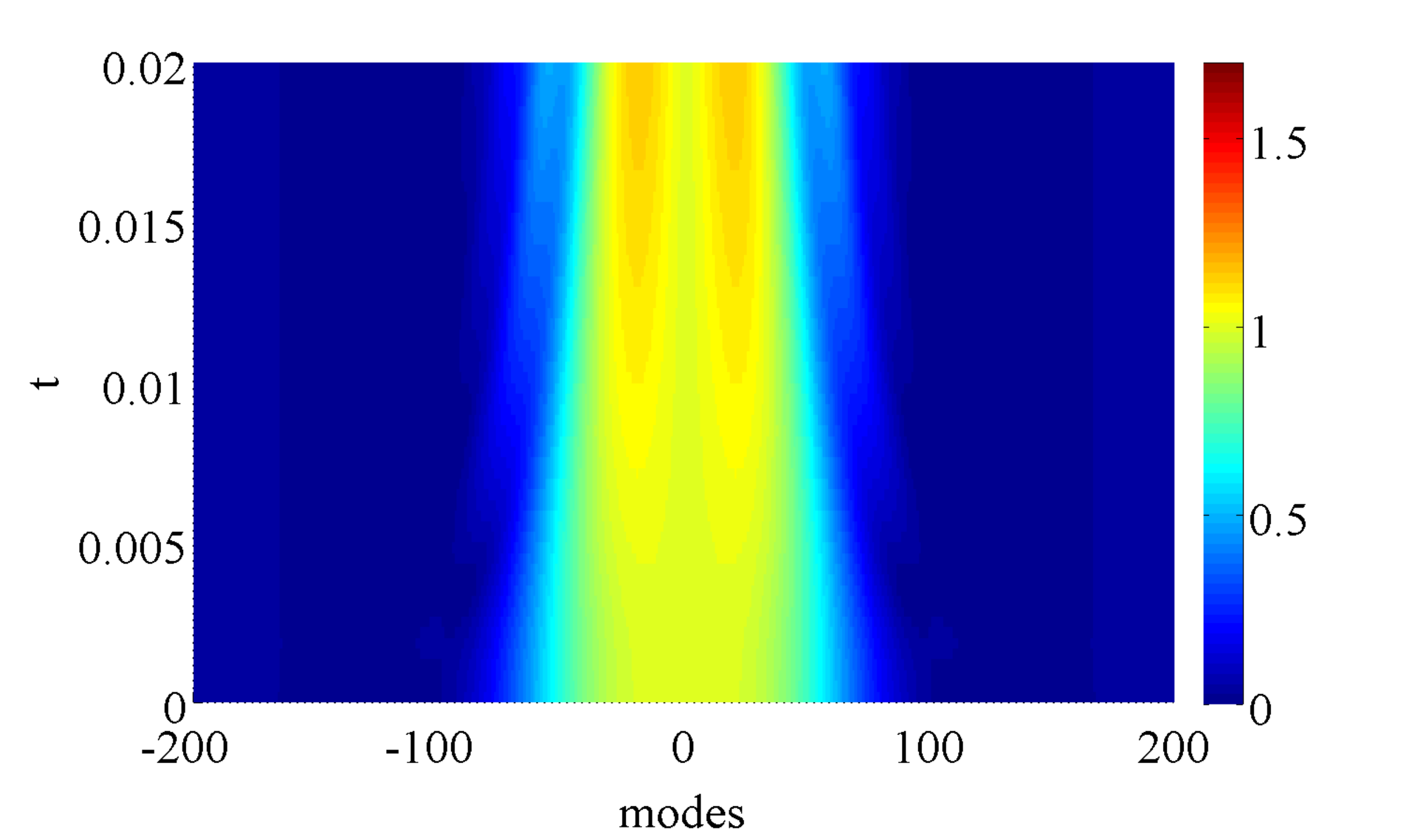}
  \end{subfigure}\\
  \begin{subfigure}[t]{0.45\textwidth}
    \centering
    \includegraphics[width=0.9\textwidth]{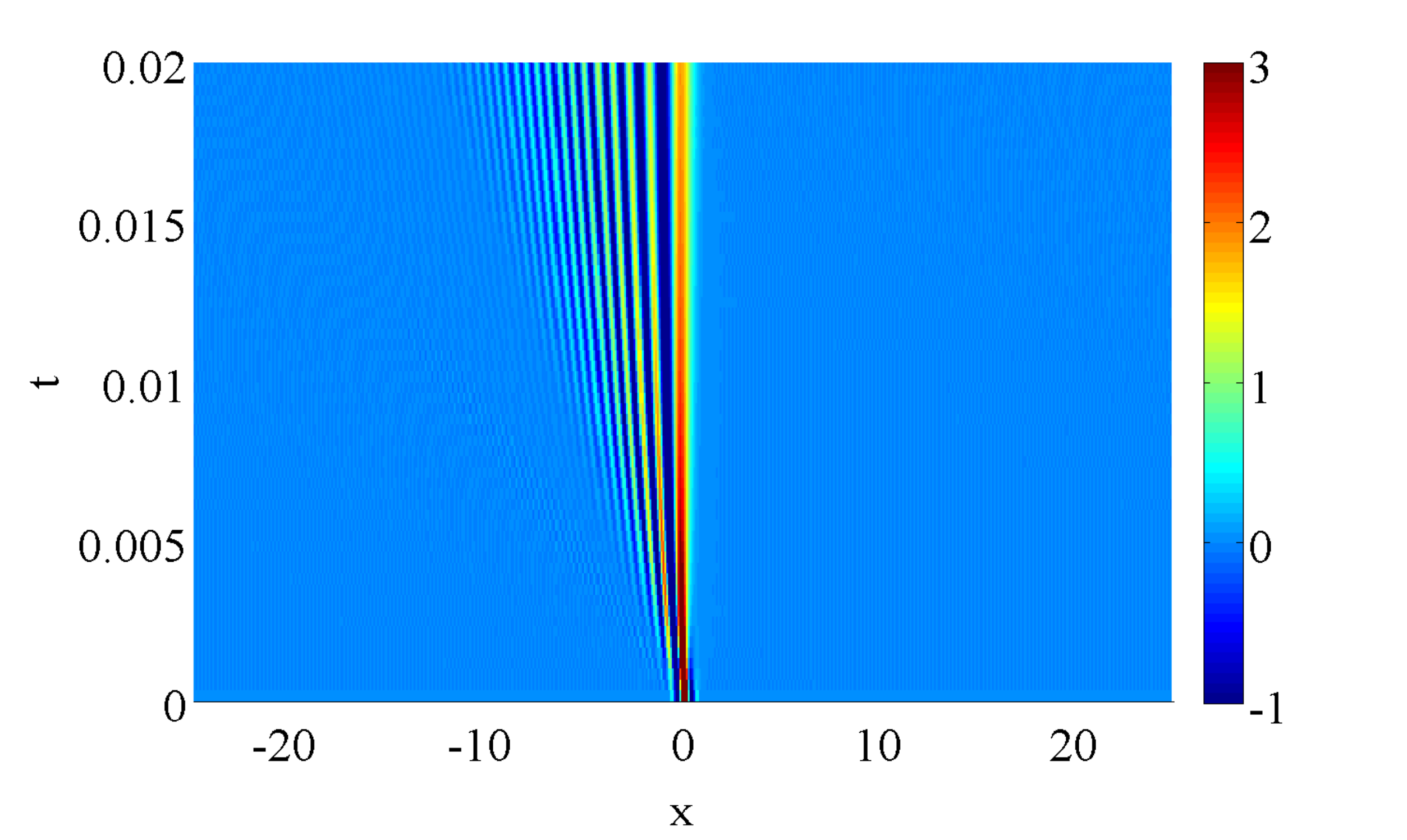}
  \end{subfigure}
  \begin{subfigure}[t]{0.45\textwidth}
    \centering
    \includegraphics[width=0.9\textwidth]{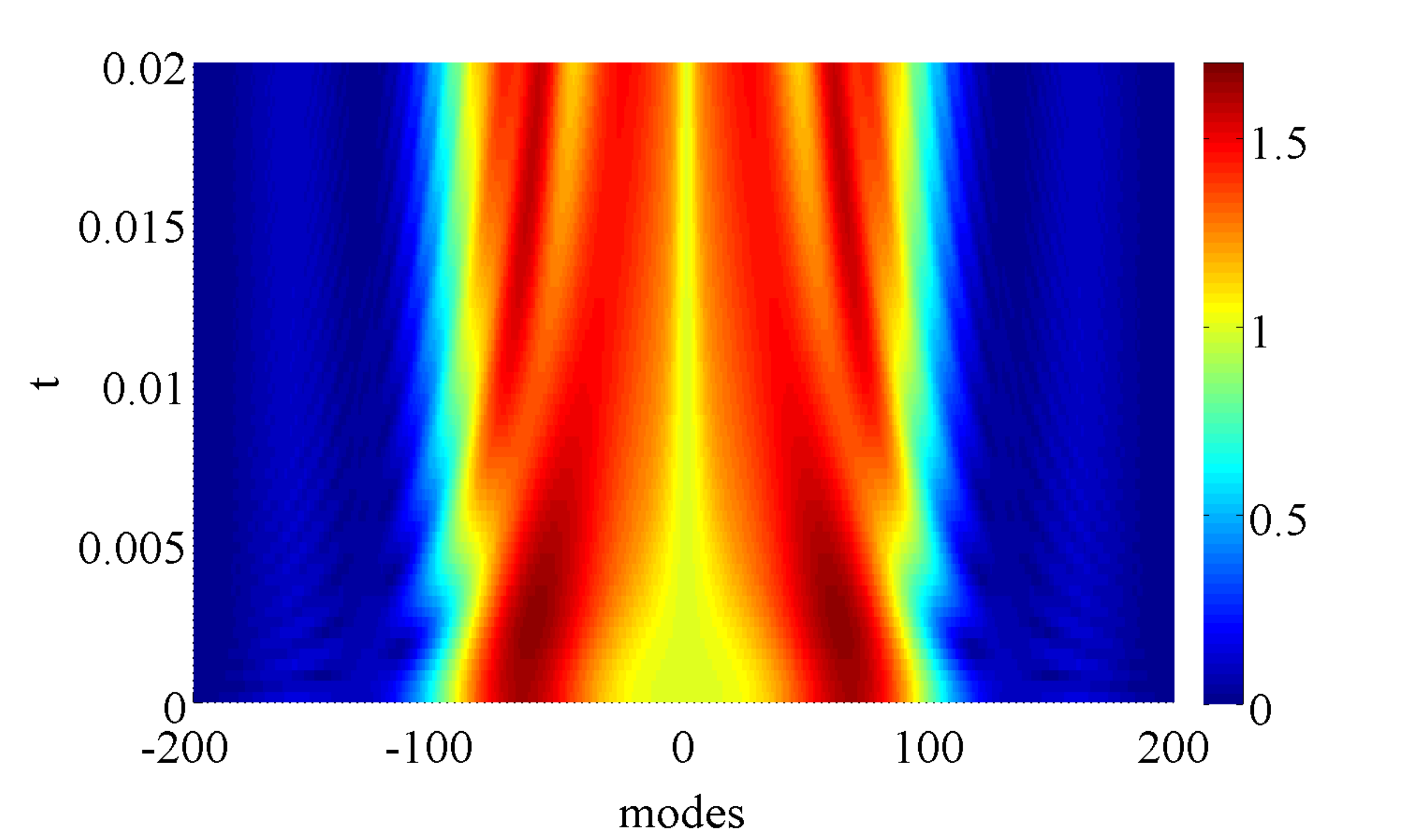}
  \end{subfigure}
  \begin{subfigure}[t]{0.45\textwidth}
    \centering
    \includegraphics[width=0.9\textwidth]{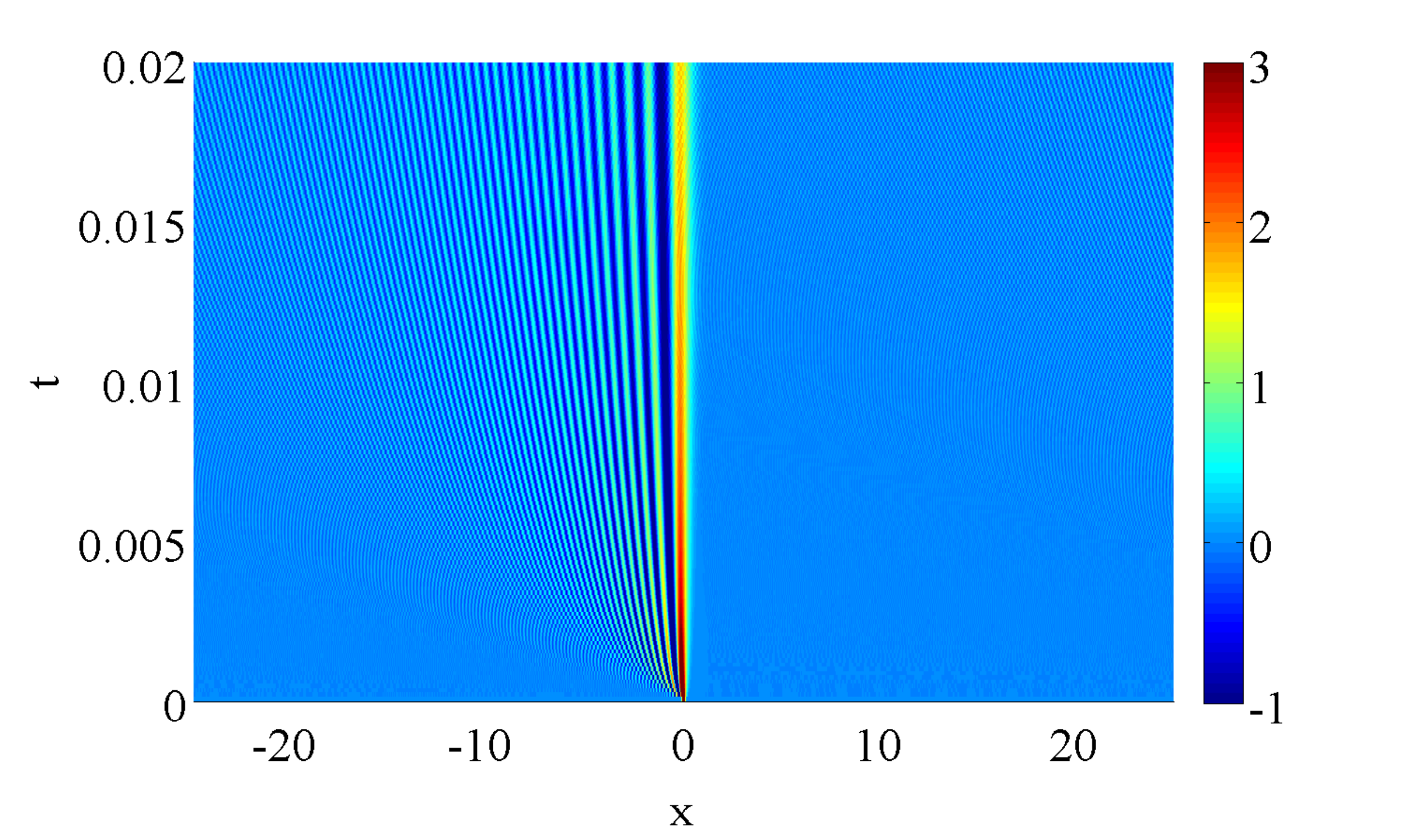}
  \end{subfigure}
  \begin{subfigure}[t]{0.45\textwidth}
    \centering
    \includegraphics[width=0.9\textwidth]{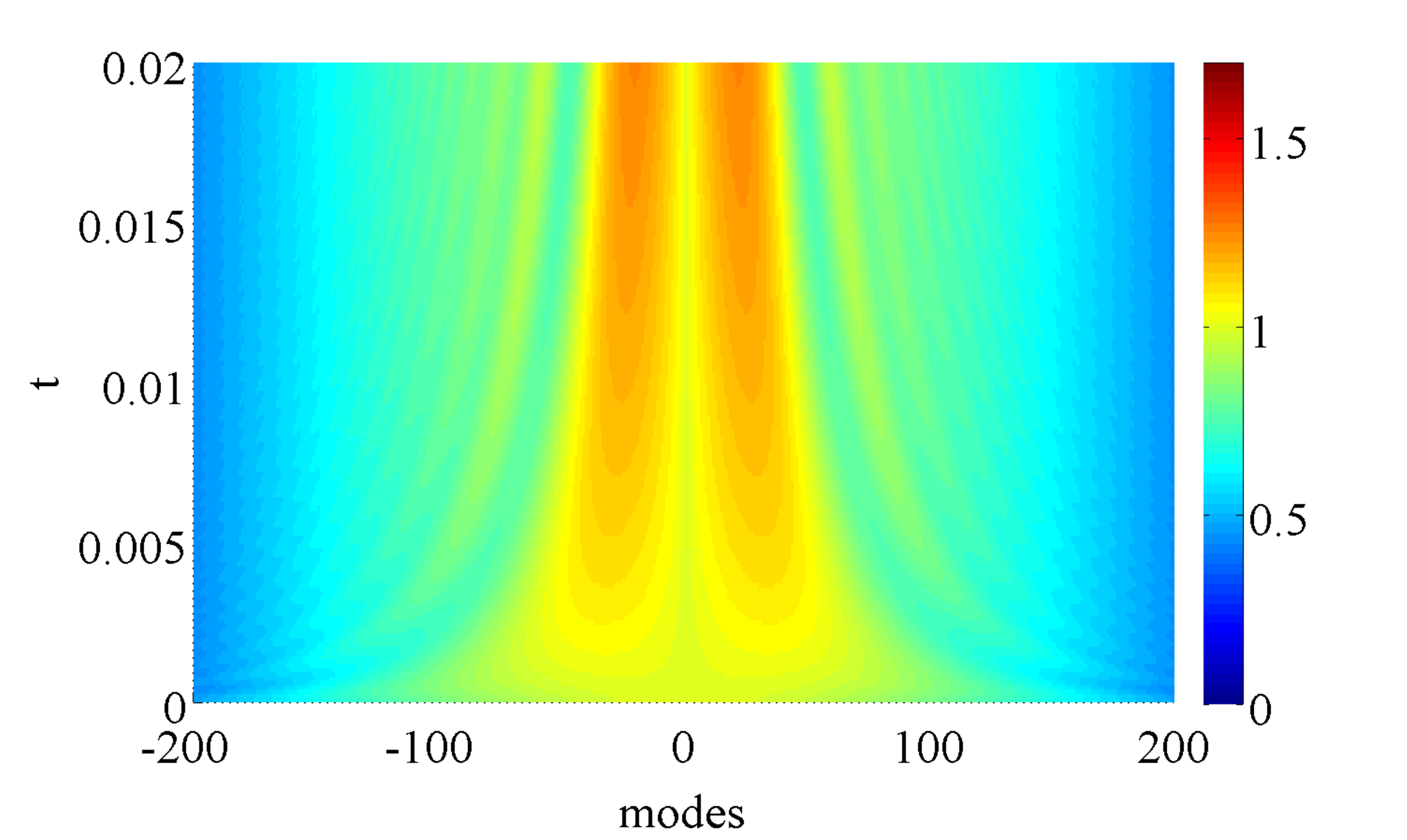}
  \end{subfigure}
  \caption{Solutions of the KdV equation with impulse initial condition
    using a first moment approximation, $\eta_{0,1}$ (top row), a
    discrete second moment approximation, $\eta_{cubic}$ (second row),
 a smooth second moment approximation $\eta_{2,5}$ (third
    row) and the reference solution using $\delta_\sigma$ on a very
    fine mesh (bottom row). The images in the right column are 
    solutions in the Fourier domain.}
\label{fig:KDV-2}
\end{figure}
Figure \ref{fig:KDV-2} depicts our solutions $u_{H,h}(x,t)$ for
different choices of $\delta_H$ along with the reference solution with
$\delta_\sigma$. The left column shows the computed solution in the
$(x,t)$ plane. The right column shows the Fourier modes of the computed
solution as time evolves. As in Figure \ref{fig:KDV-1}, we see that the
single soliton is captured for all choices of $\delta_H$. Furthermore,
the numerical solutions using $\eta_{0,1}$ and $\eta_{\text{cubic}}$ are
less noisy (and numerically better behaved) compared to
$\eta_{2,5}$. However, the lower order radiative waves are very
different in these approximations (compare to the reference solution in
the bottom row). This example poses an interesting question.  Recall
that $\eta_{\text{cubic}}$ satisfies the same number of moment
conditions as $\eta_{2,5}$ but the radiative waves look very different
between the two.  So if one is interested in capturing the radiative
waves, it seems that $\eta_{2,5}$ is the better choice for fixed $H$. On
the other hand, $\eta_{2,5}$ clearly has larger Fourier modes and so is
more difficult to handle numerically.

This example clearly demonstrates that the choice of $\tdelta_H$ must be
made carefully.  If one is constrained to using a specific algorithm for
the discretization of a PDE, then the choice of $\tdelta_H$ must be made
accordingly. However, if the goal is to accurately capture the solution
of a PDE, then one may want to {first} select a $\tdelta_H$ to
minimize the regularization error $\|u-u_H\|_X$, and then construct a
numerical method that achieves a controlled discretization error for
$\|u_H-u_{H,h}\|_X$.

%%%%%%%%%%%%%%%%%%%%%%%%%%%%%%%%%%%%%%%
%%%%%%%%%%%%%%%%%%%%%%%%%%%%%%%%%%%%%%%
%%%%%%%%%%%%%                             %%%%%%%%%%%%%%
%%%%%%%%%%%%%      Conclusions and        %%%%%%%%%%%%%%
%%%%%%%%%%%%%       Bibliography          %%%%%%%%%%%%%%
%%%%%%%%%%%%%                             %%%%%%%%%%%%%%
%%%%%%%%%%%%%%%%%%%%%%%%%%%%%%%%%%%%%%%
%%%%%%%%%%%%%%%%%%%%%%%%%%%%%%%%%%%%%%%

\section{Conclusions}

We began this article in Section~\ref{sec:intro} by posing four questions
concerning approximations of singular source terms in PDEs. We argued
that answers to the last two questions are required before the first two
can be answered; that is, we first need to consider different modes of
convergence when approximate distributions $\mathcal{S}_H \to
\mathcal{S}$ and the corresponding solutions $u_{H,h} \to u$ before it
is possible to make statements about what it means to have good
approximations. Our response to these questions can be summarized as
follows:
\begin{enumerate}[{\it {Question}~1.}]
  \setcounter{enumi}{2}
\item What form of convergence should be used to examine $\mathcal{S}_H
  \rightarrow \mathcal{S}$?
\end{enumerate}

We provide two alternative answers to this question.  Convergence in the
weak-$\ast$ topology yields a system of moment conditions that permit
the construction of sequences of distributions with arbitrary regularity
and rates of convergence. We also consider convergence in the weighted
Sobolev norm $\| \cdot \|_{(W_{-\alpha})^*}$ that allows one to study
the interplay between the resolution of a numerical scheme and the
support of the regularized source terms, and their effect on the rate of
convergence of the numerical scheme.
\begin{enumerate}[{\it {Question}~1.}]
  \setcounter{enumi}{3}
\item [{\it Question 4.}]What form of convergence should be used to
  examine $u_{H,h} \rightarrow u$?
\item[{\it Question 2.}] How does the choice of approximation
  $\mathcal{S}_H$ affect the the convergence of $u_{H,h} \rightarrow u$?
\end{enumerate}

The response to these two questions is problem-dependent and is
conditioned on the type of PDE being considered.  We showed that for a
general class of elliptic PDEs, as well as the first- and second-order
hyperbolic wave equations, numerical solutions converge pointwise in
some parts of the domain. That is, we obtain pointwise convergence away
from the support of the source term in case of elliptic PDEs and away
from the wave front in hyperbolic problems. The rate of convergence
depends on the rate of convergence of $\mathcal{S}_H \to \mathcal{S}$ in
the weak-$\ast$ topology.  We also showed that for elliptic PDEs we can
obtain convergence in the weighted Sobolev norms $\| \cdot
\|_{W_{\alpha}}$, where the convergence rate again depends on that of the 
regularization in the $\| \cdot \|_{(W_{-\alpha})^*}$
norm. Consequently, convergence of distributions $\mathcal{S}_H \to
\mathcal{S}$ controls convergence of the solution.
Finally, we can return to
\begin{enumerate}[{\it {Question}~1.}]
\item How do we construct `good' approximations $\mathcal{S}_H$ to
  $\mathcal{S}$?
\end{enumerate}

If one is interested in approximating singular sources in the sense of
distributions only, then arbitrarily high rates of convergence can be
achieved simply by satisfying higher moment conditions. We proposed a
general framework for solution of finite dimensional moment
problems. Our approach is very flexible and allows construction of
approximations using different bases. Furthermore, the lack of
uniqueness in the moment equations affords us the advantage of being
able to impose additional constraints on our constructions, such as
smoothness at certain points in the domain.

If the final goal is instead to approximate solutions of an elliptic
PDE, then we can obtain higher rates of pointwise convergence sufficiently away from the
source by
satisfying more moment conditions. Over the entire domain, the problem behaves
differently. Convergence in the weighted Sobolev norm depends on the
problem resolution and the support of the regularizations.  In
the limit when the support of the regularizations is the same order as
the mesh size of a numerical solver, the rate of convergence is
independent of the number of moment conditions and so there is no
difference between a simple 0-moment approximation and a 2-moment
approximation.

Applying our results in the context of hyperbolic PDEs proved more
challenging. We presented numerical evidence that numerical errors due
to dispersion become important for long-time solutions, and for this
reason we were unable to observe the expected regularization
error. Furthermore, for linear hyperbolic problems pointwise
convergence should be considered away from the wave front rather than at
a distance away from the support of the initial condition. 

We also demonstrated that the choice of regularization can have a
significant impact on solutions of PDEs. We looked in particular at the
second-order wave equation and compared the use of a radial delta with a
tensor-product approximation, showing that the latter produces
non-symmetric solutions.  Tensor product approximations to singular
sources are used commonly in practice, but one must be cautious and pay
careful attention to the qualitative effect of this class of
approximations on the numerical solutions, especially for nonlinear
problems and cases when advective terms dominate.

\section*{Acknowledgments}
We thank the anonymous referees whose suggestions considerably
improved the paper. We also thank Prof. S. Ruuth for helpful discussions and for the insightful
questions that motivated this work.

\bibliographystyle{abbrv}   
\bibliography{references}

\begin{thebibliography}{10}

\bibitem{Morin}
J.~P. Agnelli, E.~M. Garau, and P.~Morin.
\newblock A posteriori error estimates for elliptic problems with {D}irac
  measure terms in weighted spaces.
\newblock {\em ESAIM: Mathematical Modelling and Numerical Analysis},
  48:1557--1581, 2014.

\bibitem{krein}
N.~I. Aheizer and M.~Krein.
\newblock {\em Some questions in the theory of moments}.
\newblock Translations of Mathematical Monographs, Vol. 2. American
  Mathematical Society, Providence, RI, 1962.

\bibitem{majda}
J.~T. Beale and A.~Majda.
\newblock Vortex methods. {II}. {H}igher order accuracy in two and three
  dimensions.
\newblock {\em Mathematics of Computation}, 39:29--52, 1982.

\bibitem{Benvenuti-XFEM}
E.~Benvenuti.
\newblock A regularized {XFEM} framework for embedded cohesive interfaces.
\newblock {\em Computational Methods in Applied Mechanics and Engineering},
  2008:4367--4378, 2008.

\bibitem{Benvenuti-reg-sing}
E.~Benvenuti, G.~Ventura, N.~Ponara, and A.~Tralli.
\newblock Accuracy of three-dimensional analysis of regularized singularities.
\newblock {\em International Journal for Numerical Methods in Engineering},
  101:29--53, 2014.

\bibitem{beyer-leveque}
R.~P. Beyer and R.~J. LeVeque.
\newblock Analysis of a one-dimensional model for the immersed boundary method.
\newblock {\em SIAM Journal on Numerical Analysis}, 29(2):332--364, 1992.

\bibitem{boffi-gastaldi-2003}
D.~Boffi and L.~Gastaldi.
\newblock A finite element approach for the immersed boundary method.
\newblock {\em Computers \&\ Structures}, 81(8-11):491--501, 2003.

\bibitem{boffi-gastaldi-2014}
D.~Boffi and L.~Gastaldi.
\newblock Discrete models for fluid-structure interactions: {T}he finite
  element immersed boundary method, July 20, 2014.
\newblock arXiv:1407.5261v1~[math.NA].

\bibitem{Brezis}
H.~Brezis.
\newblock {\em Functional Analysis, Sobolev Spaces and Partial Differential
  Equations}.
\newblock Springer, 2011.

\bibitem{CortezMinion}
R.~Cortez and M.~Minion.
\newblock The blob projection method for immersed boundary problems.
\newblock {\em Journal of Computational Physics}, 161(2):428--453, 2000.

\bibitem{Dangelo}
C.~D'Angelo.
\newblock Finite element approximation of elliptic problems with {D}irac
  measure terms in weighted spaces: {A}pplications to one-and three-dimensional
  coupled problems.
\newblock {\em SIAM Journal on Numerical Analysis}, 50(1):194--215, 2012.

\bibitem{deal-II}
deal.ii finite element package, version 8.0.0.
\newblock \url{http://www.dealii.org/8.0.0}, July 2013.

\bibitem{drazin}
P.~G. Drazin and R.~S. Johnson.
\newblock {\em Solitons: An Introduction}, volume~2.
\newblock Cambridge University Press, 1989.

\bibitem{chebfun-guide}
T.~A. Driscoll, N.~Hale, L.~N. Trefethen, and editors.
\newblock {\em Chebfun Guide}.
\newblock Pafnuty Publications, Oxford, 2014.

\bibitem{duffy}
D.~G. Duffy.
\newblock {\em Green's Functions With Applications}.
\newblock CRC Press, 2010.

\bibitem{engquist}
B.~Engquist, A.-K. Tornberg, and R.~Tsai.
\newblock Discretization of {D}irac delta functions in level set methods.
\newblock {\em Journal of Computational Physics}, 207:28--51, 2005.

\bibitem{evans}
L.~C. Evans.
\newblock {\em Partial Differential Equations}, volume~19 of {\em Graduate
  Studies in Mathematics}.
\newblock American Mathematical Society, second edition, 2010.

\bibitem{fabes}
E.~B. Fabes, C.~E. Kenig, and R.~P. Serapioni.
\newblock The local regularity of solutions of degenerate elliptic equations.
\newblock {\em Communications in Partial Differential Equations}, 7(1):77--116,
  1982.

\bibitem{Friedman}
A.~Friedman.
\newblock {\em Generalized Functions and Partial Differential Equations}.
\newblock Dover Publications, 2005.

\bibitem{gelfand}
I.~M. Gel'fand and N.~Y. Vilenkin.
\newblock {\em Generalized Functions. {V}ol. 4: {A}pplications of Harmonic
  Analysis}.
\newblock Academic Press, New York, 1964.

\bibitem{Heinonen}
J.~Heinonen, T.~Kilpel\"ainen, and O.~Martio.
\newblock {\em Nonlinear Potential Theory for Degenerate Elliptic Equations}.
\newblock Oxford Science Publications, 1993.

\bibitem{kabanikhin}
S.~I. Kabanikhin.
\newblock {\em Inverse and Ill-posed Problems: Theory and Applications},
  volume~55 of {\em Inverse and Ill-Posed Problems Series}.
\newblock Walter De Gruyter, 2011.

\bibitem{liumori}
Y.~Liu and Y.~Mori.
\newblock Properties of discrete delta functions and local convergence of the
  immersed boundary method.
\newblock {\em SIAM Journal on Numerical Analysis}, 50(6):2986--3015, 2012.

\bibitem{liumori2}
Y.~Liu and Y.~Mori.
\newblock {$L^p$} convergence of the immersed boundary method for stationary
  {S}tokes problems.
\newblock {\em SIAM Journal on Numerical Analysis}, 52(1):496--514, 2014.

\bibitem{mclean}
W.~McLean.
\newblock {\em Strongly Elliptic Systems and Boundary Integral Equations}.
\newblock Cambridge University Press, 2000.

\bibitem{mori-CPAM}
Y.~Mori.
\newblock Convergence proof of the velocity field for a {S}tokes flow immersed
  boundary method.
\newblock {\em Communications on Pure and Applied Mathematics},
  {LXI}:1213--1263, 2008.

\bibitem{osher}
S.~J. Osher and R.~P. Fedkiw.
\newblock {\em Level Set Methods and Dynamic Implicit Surfaces}.
\newblock Springer, 2003.

\bibitem{peskin}
C.~S. Peskin.
\newblock The immersed boundary method.
\newblock {\em Acta Numerica}, 11:479--517, 2002.

\bibitem{suarez}
J.-P. Suarez, G.~B. Jacobs, and W.-S. Don.
\newblock A high-order {D}irac-delta regularization with optimal scaling in the
  spectral solution of one-dimensional singular hyperbolic conservation laws.
\newblock {\em SIAM Journal on Scientific Computing}, 36(4):A1831--A1849, 2014.

\bibitem{tornberg-quad}
A.-K. Tornberg.
\newblock Multi-dimensional quadrature of singular and discontinuous functions.
\newblock {\em BIT}, 42(3):644--6695, 2002.

\bibitem{tornberg-engq}
A.-K. Tornberg and B.~Engquist.
\newblock Regularization techniques for numerical approximation of {PDEs} with
  singularities.
\newblock {\em Journal of Scientific Computing}, 19(1--3):527--552, 2003.

\bibitem{tornberg}
A.-K. Tornberg and B.~Engquist.
\newblock Numerical approximations of singular source terms in differential
  equations.
\newblock {\em Journal of Computational Physics}, 200(2):462--488, 2004.

\bibitem{trefethen}
L.~N. Trefethen.
\newblock {\em Spectral Methods in {MATLAB}}.
\newblock SIAM, Philadelphia, PA, 2000.

\bibitem{Walden}
J.~Wald{\'e}n.
\newblock On the approximation of singular source terms in differential
  equations.
\newblock {\em Numerical Methods for Partial Differential Equations},
  15(4):503--520, 1999.

\bibitem{shu}
Y.~Yang and C.-W. Shu.
\newblock Discontinuous {G}alerkin method for hyperbolic equations involving
  {$\delta$}-singularities: negative-order norm error estimates and
  applications.
\newblock {\em Numerische Mathematik}, 124(4):753--781, 2013.

\bibitem{tornberg-zahedi}
S.~Zahedi and A.-K. Tornberg.
\newblock Delta function approximations in level set methods by distance
  function extension.
\newblock {\em Journal of Computational Physics}, 229:2199--2219, 2010.

\end{thebibliography}

\end{document}